\documentclass[english]{article}
\usepackage[T1]{fontenc}
\usepackage[latin9]{inputenc}
\usepackage{babel}
\usepackage{verbatim}
\usepackage{units}
\usepackage{textcomp}
\usepackage{mathrsfs}
\usepackage{mathtools}
\usepackage{amsmath}
\usepackage{amsthm}
\usepackage{amssymb}
\usepackage{stackrel}
\usepackage{graphicx}
\usepackage[authoryear]{natbib}
\usepackage[unicode=true]
 {hyperref}

\makeatletter

\newcommand{\noun}[1]{\textsc{#1}}
\providecommand{\tabularnewline}{\\}

\theoremstyle{plain}
\newtheorem{thm}{\protect\theoremname}[section]
\theoremstyle{plain}
\newtheorem{lem}[thm]{\protect\lemmaname}

\usepackage{lineno}
\usepackage{amsthm}
\usepackage{thmtools}
\usepackage{thm-restate}

\newtheoremstyle{mystyle}
  {}
  {}
  {\itshape}
  {}
  {\bfseries}
  {.}
  { }
  {}

\theoremstyle{mystyle}

\@ifundefined{showcaptionsetup}{}{%
 \PassOptionsToPackage{caption=false}{subfig}}
\usepackage{subfig}
\makeatother

\providecommand{\lemmaname}{Lemma}
\providecommand{\theoremname}{Theorem}

\begin{document}
\title{Using Exponential Histograms to Approximate the Quantiles of Heavy-
and Light-Tailed Data}
\author{Philip T.\ Labo\thanks{Author contact: \protect\href{http://mailto:plabo@alumni.stanford.edu}{plabo@alumni.stanford.edu}.}}
\maketitle
\begin{abstract}
Exponential histograms, with bins of the form $\left\{ \left(\rho^{k-1},\rho^{k}\right]\right\} _{k\in\mathbb{Z}}$,
for $\rho>1$, straightforwardly summarize the quantiles of streaming
data sets (\citet{MRL19}). While they guarantee the relative accuracy
of their estimates, they appear to use only $\log n$ values to summarize
$n$ inputs. We study four aspects of exponential histograms---size,
accuracy, occupancy, and largest gap size---when inputs are i.i.d.\ $\mathrm{Exp}\left(\lambda\right)$
or i.i.d.\ $\mathrm{Pareto}\left(\nu,\beta\right)$, taking $\mathrm{Exp}\left(\lambda\right)$
(or, $\mathrm{Pareto}\left(\nu,\beta\right)$) to represent all light-
(or, heavy-) tailed distributions. We show that, in these settings,
size grows like $\log n$ and takes on a Gumbel distribution as $n$
grows large. We bound the missing mass to the right of the histogram
and the mass of its final bin and show that occupancy grows apace
with size. Finally, we approximate the size of the largest number
of consecutive, empty bins. Our study gives a deeper and broader view
of this low-memory approach to quantile estimation.
\end{abstract}

\section{Introduction \label{sec:Introduction}}

Modern organizations collect tons of data, and yet storage is expensive.
A great deal of research has thus gone into the invention of techniques
for distilling the data into summaries. The idea is to selectively
throw away most of the data while keeping enough of it to accurately
approximate the answers to future questions. (\citet{CY20} gives
an overview.) Of the many possible questions and data types, we consider
the approximation of numerical quantiles.\footnote{This is a challenging task. \citet{MP80} shows that computing an
inner order statistic in one pass requires $\mathcal{O}\left(n\right)$
space.}

Say we have multiple data centers, each summarizing the data it receives.
The approach we discuss solves the following problem: 1.\ When a
new numerical input value $X_{n}$ reaches a data center, it should
immediately incorporate it into its summary. 2.\ If we supply $0\leq q\leq1$
to a data center, it should immediately respond with a provably-accurate
approximation of the $q^{\mathrm{th}}$ quantile of $X_{1},X_{2},\ldots,X_{n}$.
3.\ If we supply the system a list of centers, it should quickly
build a summary of the data summarized by those centers. This summary
must guarantee accurate quantile estimation for the data it summarizes;
it need not summarize new data received by the data centers following
its assembly.

An exponential histogram is a pair $\left(\mathbf{B},\rho\right)$,
where $\rho>1$ and
\[
B_{n,k}\coloneqq\sum_{i=1}^{n}\mathbf{1}_{\left\{ \rho^{k-1}<X_{i}\leq\rho^{k}\right\} },\textrm{ for }n\geq1\textrm{ and }k\in\mathbb{Z},
\]
gives the number of data values $X_{1},X_{2},\ldots,X_{n}$ that fall
into bin $k$, $\left(\rho^{k-1},\rho^{k}\right]$.\footnote{$\mathbf{1}_{\left\{ \mathscr{S}\right\} }$ equals one if statement
$\mathscr{S}$ is true; zero otherwise.} In practical settings, one keeps a separate counter for $n$---and
does not index by it. We assume that $\Pr\left(X_{i}>0\right)=1$.
In settings with positive \emph{and} negative $X_{i}$, two exponential
histograms, and a counter if $\Pr\left(X_{i}=0\right)>0$, suffice.
Note that: 1.\ Incorporating new data value $X_{n}$ requires only
$B_{n,k}\leftarrow B_{n-1,k}+1$, for $k=\left\lceil \log_{\rho}X_{n}\right\rceil $.
2.\ See below. 3.\ Combining exponential histograms that use the
same $\rho$ requires only the summing of corresponding entries, the
combined histogram retaining the accuracy of its progenitors.

In considering accuracy, let us assume for a moment that $X_{i}\stackrel{\mathrm{iid}}{\sim}F$.
Then,

\begin{equation}
\left(B_{n,j}\right)_{j\in\mathbb{Z}}\sim\mathrm{Multinomial}\left(n,\,\left(F\left(\rho^{j}\right)-F\left(\rho^{j-1}\right)\right)_{j\in\mathbb{Z}}\right),\label{eq:distBn}
\end{equation}
so that, in particular, \emph{
\begin{equation}
\frac{B_{n,k}}{n}\sim\frac{\mathrm{Binomial}\left(n,\,F\left(\rho^{k}\right)-F\left(\rho^{k-1}\right)\right)}{n}\stackrel[\infty]{n}{\longrightarrow}F\left(\rho^{k}\right)-F\left(\rho^{k-1}\right),\label{eq:distBnk}
\end{equation}
}where convergence occurs with probability one by the strong law of
large numbers. By the same reasoning, we have 
\begin{equation}
\frac{1}{n}\bar{B}_{n,k}\coloneqq\frac{1}{n}\sum_{j=-\infty}^{k}B_{n,j}\sim\frac{\mathrm{Binomial}\left(n,\,F\left(\rho^{k}\right)\right)}{n}\stackrel[\infty]{n}{\longrightarrow}F\left(\rho^{k}\right).\label{eq:distSumBnj}
\end{equation}
That is to say, the accumulation of data leads to the accurate approximation
of probabilities $\left(F\left(\rho^{j}\right)\right)_{j\in\mathbb{Z}}$.
Further, the expected error in approximations (\ref{eq:distBnk})
and (\ref{eq:distSumBnj}) is $\mathcal{O}\left(\nicefrac{1}{\sqrt{n}}\right)$.

While statements (\ref{eq:distBn}) to (\ref{eq:distSumBnj}) are
true of any histogram (modulo replacing $\rho^{k}$ with general $c_{k}$),
one might ask, what makes exponential histograms special? For one
thing, their size: §\ref{sec:Size} shows that exponential histograms
require only $\log n$ space to store a sample of size $n$ from the
light-tailed $\mathrm{Exp}\left(\lambda\right)$ or the heavy-tailed
$\mathrm{Pareto}\left(\nu,\beta\right)$. For another, their accuracy:
The literature on quantile estimation uses two measures of accuracy
(\emph{e.g.}, \citet{GK01,CKLTV21}). For $\hat{X}_{q}$ an estimate
of the $q^{\mathrm{th}}$ quantile and $X_{\left(1\right)}\leq X_{\left(2\right)}\leq\cdots\leq X_{\left(n\right)}$
the order statistics of the $X_{i}$, \emph{absolute} and \emph{relative}
accuracy guarantee that 
\begin{align}
\left|X_{\left(\left\lfloor 1+\left(n-1\right)q\right\rfloor \right)}-\hat{X}_{q}\right| & \leq\epsilon n\textrm{ and}\label{eq:absError}\\
\left|X_{\left(\left\lfloor 1+\left(n-1\right)q\right\rfloor \right)}-\hat{X}_{q}\right| & \leq\epsilon X_{\left(\left\lfloor 1+\left(n-1\right)q\right\rfloor \right)},\label{eq:relError}
\end{align}
for $0<\epsilon<1$, \emph{e.g.}, $\epsilon=0.01$, and $X_{\left(\left\lfloor 1+\left(n-1\right)q\right\rfloor \right)}$
the (lower) $q^{\mathrm{th}}$ quantile. Which is better? In many
cases, $X_{\left(\left\lfloor 1+\left(n-1\right)q\right\rfloor \right)}\ll n$
with high probability; \emph{e.g.}, for $X_{i}$ i.i.d.\ $\mathrm{Exp}\left(1\right)$,
$\mathbb{E}X_{\left(n\right)}\sim\log n$, as $n\rightarrow\infty$.
In others, $X_{\left(\left\lfloor 1+\left(n-1\right)q\right\rfloor \right)}\gg n$
with high probability; \emph{e.g.}, for $X_{i}$ i.i.d.\ $\mathrm{Pareto}\left(1,1\right)$,
$\mathbb{E}X_{\left(n\right)}=\infty$. The distribution of the expected
data plays a role in weighing (\ref{eq:absError}) against (\ref{eq:relError});
usually (\ref{eq:relError}) wins.

Exponential histograms give quantile estimates satisfying relative
accuracy guarantee (\ref{eq:relError}) (\citet{MRL19}). To see this
note that:

\begin{restatable}{myfact}{minmaxinterval}

\label{fact:MinMaxInterval}If $0<a<b<\infty$, then $\arg\min_{\theta\in\left[a,b\right]}\max_{x\in\left[a,b\right]}\frac{\left|x-\theta\right|}{x}=\frac{2ab}{a+b}$,
which implies that $\min_{\theta\in\left[a,b\right]}\max_{x\in\left[a,b\right]}\frac{\left|x-\theta\right|}{x}=\frac{b-a}{a+b}$.

\end{restatable}
\begin{proof}
This is \citet{HS20} Proposition 3.16. For $a\leq\theta\leq b$,
\[
\max_{x\in\left[a,b\right]}\frac{\left|x-\theta\right|}{x}=\max\left\{ \frac{\left|a-\theta\right|}{a},\frac{\left|b-\theta\right|}{b}\right\} =\max\left\{ \frac{\theta}{a}-1,1-\frac{\theta}{b}\right\} 
\]
since $1-\nicefrac{\theta}{x}$ grows in $x$. Because $\nicefrac{\theta}{a}-1$
increases in $\theta$ from $0$ while $1-\nicefrac{\theta}{b}$ decreases
in $\theta$ to $0$, the above maxima is smallest when $\nicefrac{\theta}{a}-1=1-\nicefrac{\theta}{b}$;
\emph{i.e.}, the minimizing $\theta$ is $\nicefrac{2ab}{\left(a+b\right)}$,
which gives $\min_{\theta\in\left[a,b\right]}\max_{x\in\left[a,b\right]}\nicefrac{\left|x-\theta\right|}{x}=\nicefrac{\left(b-a\right)}{\left(a+b\right)}$.
\end{proof}
Applying Fact \ref{fact:MinMaxInterval}, we fix $0<\epsilon<1$ and
let $\rho\coloneqq\nicefrac{\left(1+\epsilon\right)}{\left(1-\epsilon\right)}$.\footnote{In practice, one picks $0<\epsilon<1$ first and then sets $\rho\coloneqq\nicefrac{\left(1+\epsilon\right)}{\left(1-\epsilon\right)}$.
We use ``$\rho$'' for ratio.} Then, $\forall j\in\mathbb{Z}$,
\[
\min_{\theta\in\left[\rho^{j-1},\rho^{j}\right]}\max_{x\in\left[\rho^{j-1},\rho^{j}\right]}\frac{\left|x-\theta\right|}{x}=\frac{\rho^{j}-\rho^{j-1}}{\rho^{j-1}+\rho^{j}}=\frac{\rho-1}{\rho+1}=\epsilon.
\]
Furthermore, for $0\leq q\leq1$, find $k_{q}\in\mathbb{Z}$ such
that $X_{\left(\left\lfloor 1+\left(n-1\right)q\right\rfloor \right)}\in\left(\rho^{k_{q}-1},\rho^{k_{q}}\right]$.
Then, $\hat{X}_{q}\coloneqq\frac{2\rho^{k_{q}-1}\rho^{k_{q}}}{\rho^{k_{q}-1}+\rho^{k_{q}}}=\frac{2\rho^{k_{q}}}{\rho+1}$
approximates $X_{\left(\left\lfloor 1+\left(n-1\right)q\right\rfloor \right)}$
and satisfies (\ref{eq:relError}).

\subsection{Heavy- and Light-Tailed Data Distributions \label{subsec:Heavy-and-Light-Tailed}}

\begin{table}
\hfill{}%
\begin{tabular}{c|ccc}
\noun{Name} & $\mathrm{Exp}\left(\lambda\right)$ & $\mathrm{Pareto}\left(\nu,\beta\right)$ & $\mathrm{Gumbel}\left(\mu,\sigma\right)$\tabularnewline
\hline 
\noun{Support} & $\left(0,\infty\right)$ & $\left(\nu,\infty\right)$ & $\mathbb{R}$\tabularnewline
\noun{Location} & --- & $\nu>0$ & $\mu\in\mathbb{R}$\tabularnewline
\noun{Rate/Scale} & $\lambda>0$ & $\beta>0$ & $\sigma>0$\tabularnewline
$f_{X}\left(x\right)$ & $\lambda e^{-\lambda x}$ & $\left(\nicefrac{\beta}{x}\right)\left(\nicefrac{\nu}{x}\right)^{\beta}$ & $\frac{1}{\sigma}e^{-\left(\frac{x-\mu}{\sigma}+e^{-\frac{x-\mu}{\sigma}}\right)}$\tabularnewline
$F_{X}\left(x\right)$ & $1-e^{-\lambda x}$ & $1-\left(\nicefrac{\nu}{x}\right)^{\beta}$ & $e^{-e^{-\left.\left(x-\mu\right)\right/\sigma}}$\tabularnewline
$\mathbb{E}X$ & $\nicefrac{1}{\lambda}$ & $\beta>1$: $\frac{\nu\beta}{\beta-1}$ & $\mu+\gamma\sigma$\tabularnewline
$\mathrm{Var}\left(X\right)$ & $\nicefrac{1}{\lambda^{2}}$ & $\beta>2$: $\frac{\nu^{2}\beta}{\left(\beta-1\right)^{2}\left(\beta-2\right)}$ & $\frac{\pi^{2}\sigma^{2}}{6}$\tabularnewline
$\mathrm{Skew}\left(X\right)$ & 2 & $\beta>3$: $\frac{2\left(1+\beta\right)}{\beta-3}\sqrt{\frac{\beta-2}{\beta}}$ & $\nicefrac{12\sqrt{6}\zeta\left(3\right)}{\pi^{3}}$\tabularnewline
$\mathbb{E}e^{tX}$ & $t<\lambda$: $\frac{\lambda}{\lambda-t}$ & $\infty$ & $\Gamma\left(1-\sigma t\right)e^{\mu t}$\tabularnewline
\end{tabular}\hfill{}

\caption{Probability distributions of interest. For $X\sim\mathrm{Pareto}\left(1,\beta\right)$,
$\mathbb{E}X=\infty$ when $\beta\protect\leq1$, $\mathrm{Var}\left(X\right)=\infty$
when $\beta\protect\leq2$, and $\mathrm{Skew}\left(X\right)=\infty$
when $\beta\protect\leq3$. The $\gamma$ and $\zeta\left(3\right)$
in the Gumbel mean and skewness are Euler and Apéry's constants.}

\label{tab:distributions}
\end{table}

This paper studies how data of different distributions populate the
bins of an exponential histogram. While we focus on the $\mathrm{Exp}\left(\lambda\right)$
and $\mathrm{Pareto}\left(\nu,\beta\right)$ settings, the implications
of our analysis go beyond these constraints. For example, if $X\sim\mathrm{Pareto}\left(\nu,\beta\right)$,
then $f_{X^{-1}}\left(x\right)=\beta\nu^{\beta}x^{\beta-1}$ on $\left(0,\nicefrac{1}{\nu}\right)$.
Putting $\nu=\beta=1$ gives $\nicefrac{1}{X}\sim\mathrm{Uniform}\left(0,1\right)$.
Although we do not focus on $\mathrm{Uniform}\left(0,1\right)$, many
of the results for $\mathrm{Uniform}\left(0,1\right)$, in particular
size, match those for $\mathrm{Pareto}\left(1,1\right)$.

Following \citet{FKZ13} we call a distribution:
\begin{eqnarray}
\textrm{heavy-tailed}\iff & \forall t>0, & \int_{-\infty}^{\infty}e^{tx}f\left(x\right)dx=\infty\label{eq:heavy}\\
\textrm{light-tailed}\iff & \exists t>0,\textrm{ such that } & \int_{-\infty}^{\infty}e^{tx}f\left(x\right)dx<\infty,\label{eq:light}
\end{eqnarray}
where we assume the existence of density function $f$. Put another
way, we call $X$ heavy-tailed (light-tailed) if its moment generating
function $\mathbb{E}e^{tX}$ is infinite for all $t>0$ (is finite
for some $t>0$). We take exemplars $\mathrm{Exp}\left(\lambda\right)$
and $\mathrm{Pareto}\left(\nu,\beta\right)$ to stand in for all light-
and heavy-tailed distributions (Table \ref{tab:distributions}). As
Theorem 2.6 of \citet{FKZ13} points out, a distribution $F$ is heavy-tailed
if and only if $\limsup_{x\rightarrow\infty}\left(1-F\left(x\right)\right)e^{tx}=\infty$,
for all $t>0$; \emph{i.e.}, ``heavy-tailed'' is a tail property.

That said, from the histogram's perspective, heavy-tailed $\mathrm{Pareto}\left(1,1\right)$
looks exactly like compactly-supported $\mathrm{Uniform}\left(0,1\right)$.
In both settings $F_{\mathrm{Size}}^{-1}\left(q\right)\sim\log_{\rho}n-\log_{\rho}\left(-\log q\right)$
(see Theorem \ref{thm:ParetoCDF}). Why is this? Two things: 1.\ Larger
and larger bins to the right soften $\mathrm{Pareto}\left(1,1\right)$'s
creation of extreme outliers. 2.\ The exponential histogram has countably
many smaller and smaller bins to the left. With $\mathbb{E}U_{\left(1\right)}=\nicefrac{1}{\left(n+1\right)}=1-\mathbb{E}U_{\left(n\right)}$,
we expect $U_{i}$ i.i.d.\ $\mathrm{Uniform}\left(0,1\right)$ to
occupy bins $\left\{ -\left\lfloor \log_{\rho}\left(n+1\right)\right\rfloor ,\ldots,-1,0\right\} $.
Exponential histograms are blessed with (cursed with) larger and larger
(smaller and smaller) bins to the right (to the left).

As the above correctly suggests, the study of exponential histograms
touches on extreme value theory (\citet{R87,HF06}). Standardized
sizes of exponential histograms holding i.i.d.\ $\mathrm{Exp}\left(\lambda\right)$
or i.i.d.\ $\mathrm{Pareto}\left(\nu,\beta\right)$ data belong to
the Gumbel domain of attraction (Propositions \ref{prop:asympExpM}
and \ref{prop:ParetoLim-M}), whereas $\mathrm{Exp}\left(\lambda\right)$
and $\mathrm{Pareto}\left(\nu,\beta\right)$ themselves belong to
the Gumbel and Fréchet domains of attraction (see page 83 of \citet{D05}).
The maxima of $n$ i.i.d.\ $\mathrm{Pareto}\left(\nu,\beta\right)$
variables has a heavy tail, which the logarithm in the size attenuates.
Finally,

\begin{restatable}{myfact}{lingumbel}

\label{fact:linGumbel}Fixing $\mu,b\in\mathbb{R}$ and $\sigma,a>0$,
we note that, if $X\sim\mathrm{Gumbel}\left(\mu,\sigma\right)$, then
$aX+b\sim\mathrm{Gumbel}\left(a\mu+b,a\sigma\right)$.

\end{restatable}
\begin{proof}
Fixing $y\in\mathbb{R}$ and noting that $X\sim\mathrm{Gumbel}\left(\mu,\sigma\right)$,
we have
\[
\Pr\left(aX+b\leq y\right)=\Pr\left(X\leq\frac{y-b}{a}\right)=\exp\left(-\exp\left(-\frac{y-\left(a\mu+b\right)}{a\sigma}\right)\right),
\]
which implies that $aX+b\sim\mathrm{Gumbel}\left(a\mu+b,a\sigma\right)$. 
\end{proof}

\subsection{Our Contributions \label{subsec:Our-Contributions}}

\begin{figure}[!t]
\subfloat[$\mathrm{Exp}\left(1\right)$ data with expected bin counts $\mathbb{E}B_{1000,j}=1000\left\{ e^{-\rho^{j-1}}-e^{-\rho^{j}}\right\} $.]{\hfill{}\includegraphics[scale=0.62]{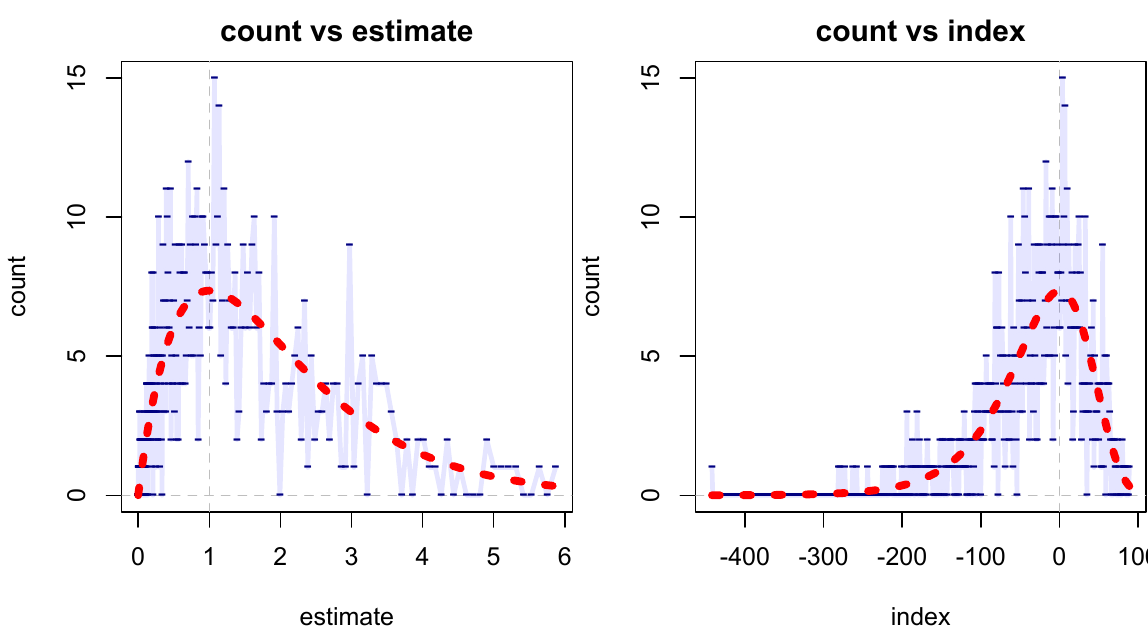}\hfill{}

}

\subfloat[$\mathrm{Pareto}\left(1,1\right)$ data with expected bin counts $\mathbb{E}B_{1000,j}=1000\left\{ \rho^{-j+1}-\rho^{-j}\right\} $.]{\hfill{}\includegraphics[scale=0.62]{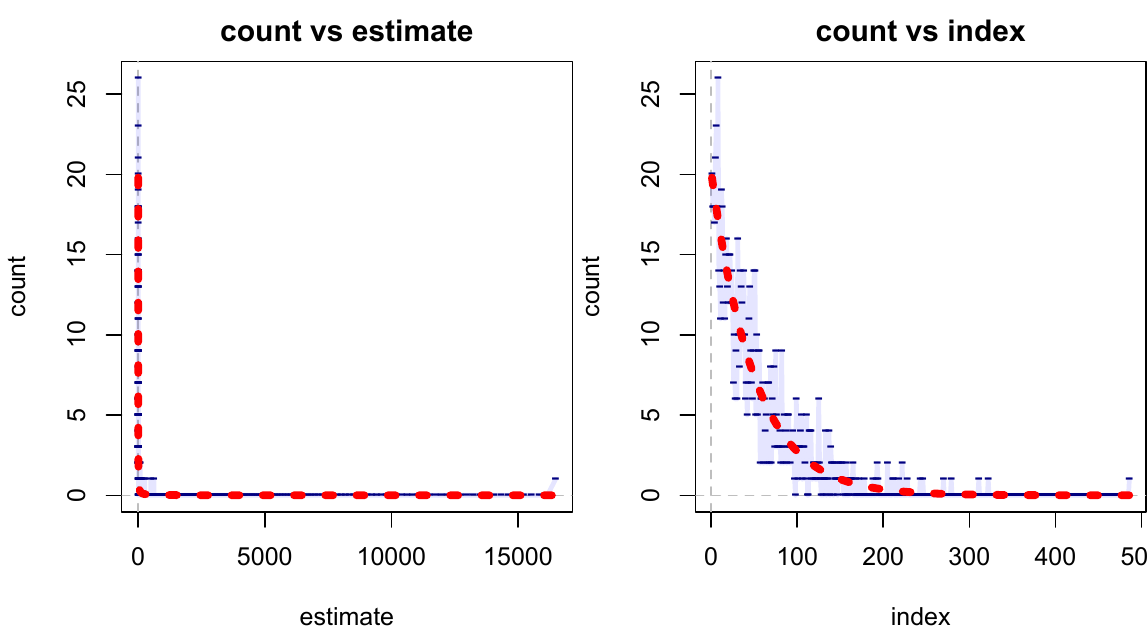}\hfill{}

}

\caption{Exponential histograms with $\rho=\frac{1.01}{0.99}\approx1.02$.
Panels on the left (on the right) show $B_{1000,j}$ versus $\nicefrac{2\rho^{j}}{\left(\rho+1\right)}$
(versus $j$). Red curves give $\mathbb{E}B_{1000,j}$.}

\label{fig:ExpHist}
\end{figure}

Recent work on estimating quantiles on data streams includes \citet{SBAS04,CJLVW06,T08,ACHPWY13,KLL16,MRL19,DE19,CKLTV21}.
\citet{ACHPWY13,KLL16} develop \emph{compactors} (sequences of length-$2m$
buffers that, using coin flips, pass their even or odd order statistics
to the next buffer) for absolute-error-guaranteed quantile estimation,
as in (\ref{eq:absError}). \citet{CKLTV21} develop this further,
presenting compactor-based sketches for relative-error-guaranteed
quantile estimation, as in (\ref{eq:relError}). While the \citet{CKLTV21}
approach requires space 
\begin{equation}
\mathcal{O}\left(\frac{\log^{\nicefrac{3}{2}}\left(\epsilon n\right)}{\epsilon}\sqrt{\log\left(\frac{\log\left(\epsilon n\right)}{\delta\epsilon}\right)}\right)\label{eq:ReqSketchSpace}
\end{equation}
to summarize $n$ data points, with error probability bounded by $0<\delta<1$
and $0<\epsilon<1$ as in (\ref{eq:relError}), the exponential histogram-based
approach of \citet{MRL19} has $F_{\mathrm{Size}}^{-1}\left(q\right)\sim\log_{\rho}n-\log_{\rho}\left(-\log q\right)$
under $\mathrm{Exp}\left(\lambda\right)$ or $\mathrm{Pareto}\left(\nu,1\right)$
sampling (see Theorems \ref{thm:distExpM} and \ref{thm:ParetoCDF})
and is computationally-simpler, and therefore less error-prone. While
\citet{MRL19} shows that exponential histograms give relatively accurate
estimates (see Fact \ref{fact:MinMaxInterval}), their size result
only apply in the light-tailed $\mathrm{Exp}\left(\lambda\right)$
setting---not in the heavy-tailed $\mathrm{Pareto}\left(\nu,\beta\right)$
setting.\footnote{$\mathrm{Pareto}\left(\nu,\beta\right)$, with $\mathbb{E}e^{tX}=\infty$
for all $t>0$, is not sub-exponential, \emph{i.e.}, does not fulfill
(\ref{eq:light}).}

\begin{table}
\makebox[\textwidth][c]{

\begin{tabular}{lll}
\noun{Symbol} & \noun{Definition} & \noun{Additional Notes}\tabularnewline
\hline 
$\gamma=0.57721\ldots$ & $\lim_{n\rightarrow\infty}\left(\sum_{k=1}^{n}\nicefrac{1}{k}-\log n\right)$ & Euler's constant\tabularnewline
$\zeta\left(3\right)=1.20205\ldots$ & $\sum_{k=1}^{\infty}\nicefrac{1}{k^{3}}$ & Apéry's constant\tabularnewline
$a_{n}\sim b_{n}$ & $\lim_{n\rightarrow\infty}\nicefrac{a_{n}}{b_{n}}=1$ & Asymptotically equivalent\tabularnewline
$\Gamma\left(x\right)$, $x>0$ & $\int_{0}^{\infty}t^{x-1}\exp\left(-t\right)dt$ & Gamma function\tabularnewline
$B\left(x,y\right)$, $x,y>0$ & $\int_{0}^{1}t^{x-1}\left(1-t\right)^{y-1}dt=\frac{\Gamma\left(x\right)\Gamma\left(y\right)}{\Gamma\left(x+y\right)}$ & Beta function\tabularnewline
$\psi\left(x\right)$, $x>0$ & $\frac{d}{dx}\log\Gamma\left(x\right)\sim\log x-\frac{1}{2x}$ & Digamma function\tabularnewline
$\psi_{m}\left(x\right)$, $x>0$ & $\frac{d^{m}}{dx^{m}}\psi\left(x\right)=\frac{d^{m+1}}{dx^{m+1}}\log\Gamma\left(x\right)$ & Polygamma function\tabularnewline
$\mathrm{Li}_{2}\left(x\right)$, $x>0$ & $\int_{1}^{x}\frac{\log t}{1-t}dt$ & Dilogarithm function\tabularnewline
$\sinh^{-1}\left(x\right)$, $x\in\mathbb{R}$ & $\log\left(x+\sqrt{x^{2}+1}\right)$ & Inverse hyperbolic sine\tabularnewline
$W_{0}\left(x\right)$, $x\geq0$ & $w$ such that $we^{w}=x$ & Lambert $W$ principal branch\tabularnewline
$W_{-1}\left(x\right)$, $x\in\left[\nicefrac{-1}{e},0\right)$ & $w$ such that $we^{w}=x$ & Lambert $W$ $-1$ branch\tabularnewline
$\mathscr{L}\left(X\right)$ & The distribution of $X$ & $\mathscr{L}$ stands for ``law''\tabularnewline
$X\sim F$ & $\mathscr{L}\left(X\right)=F$ & $X$ has distribution $F$\tabularnewline
$X\stackrel{\cdot}{\sim}F$ & $\mathscr{L}\left(X\right)\approx F$ & Approximate distribution\tabularnewline
$A\stackrel{\mathscr{L}}{=}B$ & $A$ and $B$ have same distribution & Equality in distribution\tabularnewline
$\mathrm{Skew}\left(X\right)$ & $\mathbb{E}\left[\left\{ \nicefrac{\left(X-\mathbb{E}X\right)}{\sqrt{\mathrm{Var}\left(X\right)}}\right\} ^{3}\right]$ & Skewness (asymmetry) of $X$\tabularnewline
$A_{n}\implies B$ & $\Pr\left(A_{n}\leq x\right)\rightarrow\Pr\left(B\leq x\right),\forall x$ & Convergence in distribution\tabularnewline
\end{tabular}

}

\caption{Conventions. Our definition of convergence in distribution assumes
continuous $F\left(x\right)\protect\coloneqq\Pr\left(B\protect\leq x\right)$.
When this does not hold, we have convergence in distribution if $\Pr\left(A_{n}\protect\leq x\right)\rightarrow\Pr\left(B\protect\leq x\right)$
for all continuity points $x$ of $F\left(x\right)$.}

\label{tab:conventions}
\end{table}

Exponential histograms are a simple, low-memory method for relative-error-guaranteed
quantile estimation. We corroborate and extend \citet{MRL19}. Each
of sections \ref{sec:Size}--\ref{sec:Longest-Gap} spends part of
its time in the $\mathrm{Exp}\left(\lambda\right)$ setting and part
of its time in the $\mathrm{Pareto}\left(\nu,\beta\right)$ setting
(Figure \ref{fig:ExpHist}). Section \ref{sec:Size} gives quantile
functions for histogram size and shows that size is approximately
Gumbel-distributed. Section \ref{sec:Accuracy} quantifies upper edge
precision, a concern in industrial settings. Section \ref{sec:Occupancy}
shows that the number of occupied bins grows apace with size, and
section \ref{sec:Longest-Gap} shows that the length of the largest
block of empty bins is a small fraction of the total number of empty
bins. Table \ref{tab:conventions} presents notation, and section
\ref{sec:Conclusions} concludes.

\section{Size \label{sec:Size}}

When we write \emph{size} or \emph{histogram size}, we mean $\left\lceil \log_{\rho}X_{\left(n\right)}\right\rceil -\left\lceil \log_{\rho}X_{\left(1\right)}\right\rceil +1$
\[
=\left|\left\{ \left\lceil \log_{\rho}X_{\left(1\right)}\right\rceil ,\left\lceil \log_{\rho}X_{\left(1\right)}\right\rceil +1,\ldots,\left\lceil \log_{\rho}X_{\left(n\right)}\right\rceil \right\} \right|,
\]
that is, the number of consecutive bins when we count from the one
containing $X_{\left(1\right)}$ to the one containing $X_{\left(n\right)}$.
We assume that $n\geq2$. For ease of computation we drop the ceiling
functions and focus on 
\begin{equation}
M_{n}\coloneqq\frac{\log X_{\left(n\right)}-\log X_{\left(1\right)}}{\log\rho}+1,\qquad A_{n}\coloneqq\frac{\log X_{\left(n\right)}-\log\nu}{\log\rho}+1,\label{eq:MnAn}
\end{equation}
so that w.p.1 $\mathrm{size}\in\left\{ \left\lfloor M_{n}\right\rfloor ,\left\lfloor M_{n}\right\rfloor +1\right\} $.
As we study how size grows with $n$, we simply take size to be $M_{n}$.
In the $\mathrm{Pareto}\left(\nu,\beta\right)$ setting, if one feels
certain that $X_{\left(1\right)}$ occupies the left-most bin, or,
if one wishes to start counting from $\left\lceil \log_{\rho}\nu\right\rceil $,
then $A_{n}$ measures size. When using exponential histograms, one
might wish to amortize memory allocation as in §17.1 of \citet{CLRS01}.

\subsection{The Exponential Setting}

This section characterizes $\mathscr{L}\left(M_{n}\right)$ when $X_{1},X_{2},\ldots,X_{n}\stackrel{\mathrm{iid}}{\sim}\mathrm{Exp}\left(\lambda\right)$.
Note that:

\begin{restatable}{theorem}{distexpm}

\label{thm:distExpM}For $n\geq2$, $X_{1},X_{2},\ldots,X_{n}\stackrel{\mathrm{iid}}{\sim}\mathrm{Exp}\left(\lambda\right)$,
and $M_{n}$ in (\ref{eq:MnAn}), we have
\begin{align}
F_{M_{n}}\left(\mu\right) & =\left(n-1\right)B\left(\frac{\rho^{\mu-1}+n-1}{\rho^{\mu-1}-1},n-1\right)\label{eq:ExpFM}\\
f_{M_{n}}\left(\mu\right) & =\frac{n\rho^{\mu-1}F_{M_{n}}\left(\mu\right)\log\rho}{\left(\rho^{\mu-1}-1\right)^{2}}\left\{ \psi\left(\frac{n\rho^{\mu-1}}{\rho^{\mu-1}-1}\right)-\psi\left(\frac{\rho^{\mu-1}+n-1}{\rho^{\mu-1}-1}\right)\right\} \label{eq:ExpfM}\\
F_{M_{n}}^{-1}\left(q\right) & \sim\log_{\rho}\left(\frac{n\log\left(n-1\right)+\log\left(\nicefrac{1}{q}\right)}{\log\left(\nicefrac{1}{q}\right)}\right)+1\sim\log_{\rho}n-\log_{\rho}\left(-\log q\right),\label{eq:ExpFMinv}
\end{align}
for $\mu>1$ and $0<q<1$. The asymptotic results in (\ref{eq:ExpFMinv})
hold as either or both $n\rightarrow\infty$ and $q\rightarrow1^{-}$.
The first asymptotic result in (\ref{eq:ExpFMinv}) also holds as
$q\rightarrow0^{+}$ while $n$ remains fixed.

\end{restatable}

\begin{figure}[!t]
\hfill{}\includegraphics[scale=0.525]{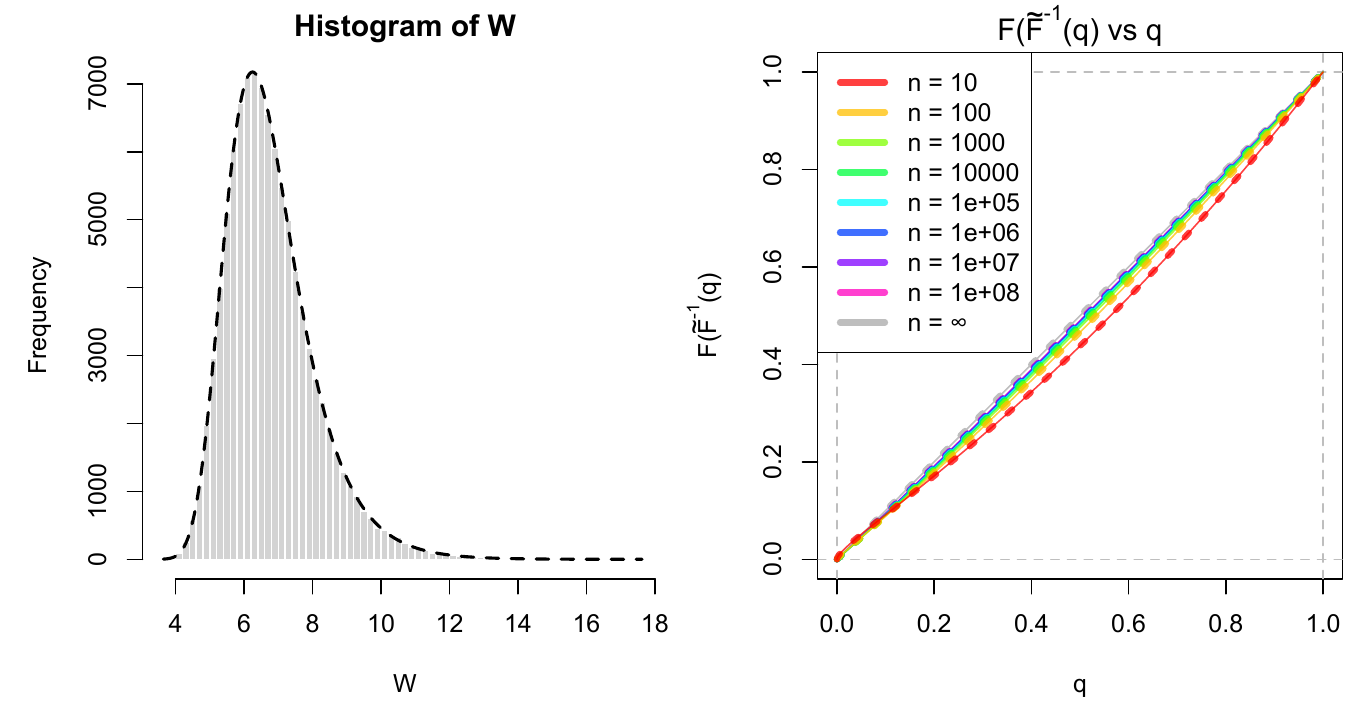}\hfill{}

\caption{The distribution of $W_{n}$ under exponential sampling. The histogram
shows $10^{5}$ simulated values $W_{n}=\log\left(\left.X_{\left(n\right)}\right/X_{\left(1\right)}\right)=\left(M_{n}-1\right)\log\rho$,
when $n=100$ and $\lambda=1$. $\mathscr{L}\left(W_{n}\right)$ and
$\mathscr{L}\left(M_{n}\right)$ do not depend on $\lambda$. The
dashed curve gives (rescaled) $f_{W_{n}}$ in (\ref{eq:ExpfW}). The
curves on the right give $F_{W_{n}}(\tilde{F}_{W_{n}}^{-1}\left(q\right))$
versus $0<q<1$ and $1\protect\leq\log_{10}n\protect\leq8$, for $\tilde{F}_{W_{n}}^{-1}$
the left-most approximation of $F_{W_{n}}^{-1}$ on the right-hand
side of (\ref{eq:ExpFWinv}). Theorem \ref{thm:distExpM} shows that
$\lim_{n\rightarrow\infty}F_{W_{n}}(\tilde{F}_{W_{n}}^{-1}\left(q\right))=q$.}

\label{fig:ExpDistWn}
\end{figure}

\begin{proof}
The proof, which uses (2.3.3) of \citet{DN03}, appears in Appendix
\ref{sec:Size-Proofs}.
\end{proof}
Figure \ref{fig:ExpDistWn} shows $\mathscr{L}\left(W_{n}\right)=\mathscr{L}\left(\left(M_{n}-1\right)\log\rho\right)$
and $F_{W_{n}}\left(\tilde{F}_{W_{n}}^{-1}\left(q\right)\right)$.
Note: 
\begin{enumerate}
\item $\mathscr{L}\left(W_{n}\right)$ and $\mathscr{L}\left(M_{n}\right)$
do not depend on $\lambda$ (\emph{cf.}\ \citet{R53} for $\mathscr{L}\left(X_{\left(i\right)}\right)$). 
\item For most values of $n$ and $q$, $F_{W_{n}}\left(\tilde{F}_{W_{n}}^{-1}\left(q\right)\right)\leq q\implies\tilde{F}_{W_{n}}^{-1}\left(q\right)\leq F_{W_{n}}^{-1}\left(q\right)$. 
\item $F_{M_{n}}^{-1}\left(q\right)$ grows like $\log_{\rho}n$ plus the
quantile function for $\mathrm{Gumbel}\left(0,\nicefrac{1}{\log\rho}\right)$.
\end{enumerate}
The following lemma gives wide bounds for $\mathbb{E}M_{n}$ and $\mathrm{Var}\left(M_{n}\right)$:

\begin{restatable}{lemma}{bndexpmomm}

\label{lem:bndExpMomM}For $n\geq2$, $X_{1},X_{2},\ldots,X_{n}\stackrel{\mathrm{iid}}{\sim}\mathrm{Exp}\left(\lambda\right)$,
$M_{n}$ in (\ref{eq:MnAn}), and
\begin{align*}
\lambda_{1} & \coloneqq\frac{n\log n}{n-1}\sim\log n\\
\lambda_{2} & \coloneqq\frac{n\left\{ \log^{2}\left(n-1\right)+\nicefrac{\pi^{2}}{3}+2\mathrm{Li}_{2}\left(\frac{n}{n-1}\right)\right\} }{n-1}\sim\log^{2}n\\
\delta_{1} & \coloneqq\frac{n\left(n-2\right)\left\{ \pi\left(\sqrt{2n-1}-1\right)-2\tan^{-1}\left(\frac{n-1}{\sqrt{2n-1}}\right)\right\} }{2\sqrt{\left(2n-1\right)\left(2n-3\right)}\left(n-1\right)}\sim\frac{\pi\sqrt{2n}}{4}\\
\delta_{2} & \coloneqq\frac{4n\left(n-2\right)\left\{ \log\left(\frac{n-1}{4}\right)+2\sinh^{-1}\left(\frac{1}{\sqrt{n-1}}\right)\right\} }{\sqrt{2n-3}\left(n-1\right)}\sim2\sqrt{2n}\log n,
\end{align*}
we have
\begin{align*}
\left\{ \mathbb{E}M_{n}-1\right\} \log\rho & \in\left[\lambda_{1},\lambda_{1}+\delta_{1}\right]\,\tilde{\in}\,\left[\log n,\frac{\pi\sqrt{2n}}{4}\right]\\
\mathrm{Var}\left(M_{n}\right)\log^{2}\rho & \in\left[\lambda_{2}-\left(\lambda_{1}+\delta_{1}\right)^{2},\lambda_{2}+\delta_{2}-\lambda_{1}^{2}\right]\,\tilde{\in}\,\left[-\frac{\pi^{2}n}{8},2\sqrt{2n}\log n\right],
\end{align*}
where $\tilde{\in}$ indicates asymptotic upper and lower bounds as
$n$ becomes large.

\end{restatable}

\begin{figure}[!t]
\hfill{}\includegraphics[scale=0.44]{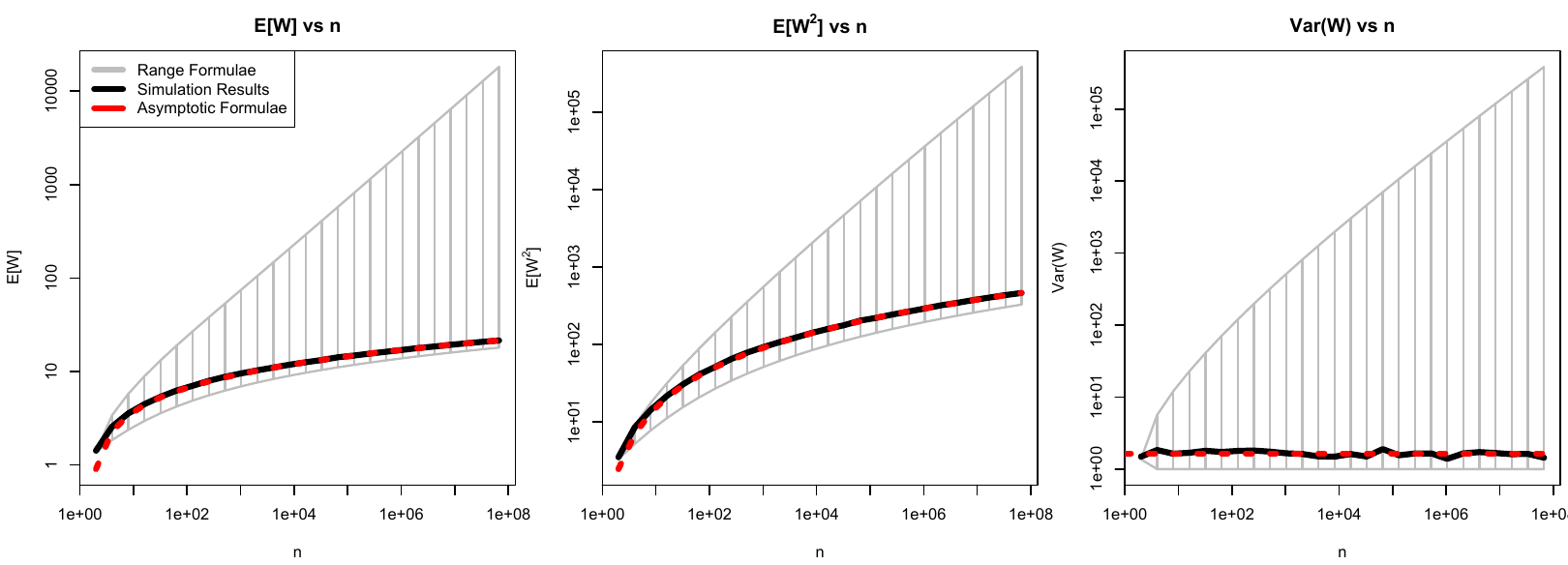}\hfill{}

\caption{The moments of $W_{n}\protect\coloneqq\log\left(\nicefrac{X_{\left(n\right)}}{X_{\left(1\right)}}\right)$
under exponential sampling. Both axes use the log scale. The range
formulae (in gray) come from Lemma \ref{lem:bndExpMomM}. We truncate
variance lower bounds at 1. Simulations (in black) use 1000 repetitions.
The asymptotic formulae (in red) come from Proposition \ref{prop:asympExpM}.}

\label{fig:ExpMomBnds}
\end{figure}

\begin{proof}
The proof, which uses bounds for the beta function provided by \citet{C07},
appears in Appendix \ref{sec:Size-Proofs}.
\end{proof}
The widths of the intervals for $\mathbb{E}M_{n}$ and $\mathrm{Var}\left(M_{n}\right)$
in Lemma \ref{lem:bndExpMomM} grow large as $n\rightarrow\infty$
(see Figure \ref{fig:ExpMomBnds}, which truncates variance lower
bounds at one). That said, the following proposition provides excellent
asymptotic approximations.

\begin{restatable}{proposition}{asympexpm}

\label{prop:asympExpM}For $X_{1},X_{2},\ldots,X_{n}\stackrel{\mathrm{iid}}{\sim}\mathrm{Exp}\left(\lambda\right)$
and $M_{n}$ in (\ref{eq:MnAn}), we have
\begin{equation}
Y_{n}\coloneqq\left(M_{n}-1\right)\log\rho-\log n-\log\log n\implies\mathrm{Gumbel}\left(0,1\right)\label{eq:ExpMtoGumbel}
\end{equation}
as $n\rightarrow\infty$, which implies that, for $n$ large,
\begin{align}
M_{n} & \stackrel{\cdot}{\sim}\mathrm{Gumbel}\left(1+\frac{\log n+\log\log n}{\log\rho},\frac{1}{\log\rho}\right)\label{eq:ExpApproxM}\\
\mathbb{E}M_{n} & \sim1+\gamma+\frac{\log n+\log\log n}{\log\rho}\sim\frac{\log n}{\log\rho}\label{eq:ExpApproxEM}\\
\mathrm{Var}\left(M_{n}\right) & \sim\frac{\pi^{2}}{6\log^{2}\rho},\textrm{ and }\mathrm{Skew}\left(M_{n}\right)\sim\frac{12\sqrt{6}\zeta\left(3\right)}{\pi^{3}}.\label{eq:ExpApproxVarM}
\end{align}

\end{restatable}
\begin{proof}
We start with (\ref{eq:ExpMtoGumbel}). Fixing $y\in\mathbb{R}$,
we have
\begin{align}
\Pr\left(Y_{n}\leq y\right) & =\Pr\left(M_{n}\leq1+\frac{y+\log n+\log\log n}{\log\rho}\right)\\
 & =\left(n-1\right)B\left(1+\frac{n}{e^{y}n\log n-1},n-1\right)\label{eq:useFW}\\
 & \sim\frac{\Gamma\left(1+\frac{n}{e^{y}n\log n-1}\right)}{\left(n-1\right)^{\frac{n}{e^{y}n\log n-1}}}\longrightarrow e^{-e^{-y}},\label{eq:takeLim}
\end{align}
where (\ref{eq:useFW}) uses (\ref{eq:ExpFM}) and (\ref{eq:takeLim})
uses a well-known approximation of $\lim_{x\rightarrow\infty}B\left(c,x\right)$
and the continuity of $\Gamma$ at 1. This, with Fact \ref{fact:linGumbel},
proves (\ref{eq:ExpMtoGumbel}) and (\ref{eq:ExpApproxM}). Results
(\ref{eq:ExpApproxEM}) and (\ref{eq:ExpApproxVarM}) use well-known
properties of $\mathrm{Gumbel}\left(\mu,\sigma\right)$ (Table \ref{tab:distributions}).
\end{proof}
\begin{restatable}{myremark}{altasymp}

\label{rem:altAsymp}One can repurpose the proof above to show that,
for $y\in\mathbb{R}$, $\lim_{n\rightarrow\infty}\Pr\left(W_{n}-\log n\leq y\right)=0$,
which implies that $W_{n}-\log n\rightarrow\infty$ in probability,
as $n\rightarrow\infty$. That is, convergence in distribution relies
on the $\log\log n$ term in the standardization of $W_{n}=\left(M_{n}-1\right)\log\rho$
(\emph{cf.}\ (\ref{eq:lim-y})).

\end{restatable}

Proposition \ref{prop:asympExpM} shows that, while $\mathbb{E}M_{n}$
grows like $\log_{\rho}n$, as $n\rightarrow\infty$, $\mathrm{Var}\left(M_{n}\right)$
is asymptotically bounded (Figure \ref{fig:ExpMomBnds}, red and black
curves). Past a certain point, as $M_{n}$ grows, its variance does
not. This is also a property of certain occupancy counts (see Theorems
\ref{thm:Karlin} and \ref{thm:Bogachev-etal}, which come from \citet{K67}
and \citet{BGY08}).

\subsection{The Pareto Setting}

We now characterize $\mathscr{L}\left(M_{n}\right)$ when $X_{1},X_{2},\ldots,X_{n}\stackrel{\mathrm{iid}}{\sim}\mathrm{Pareto}\left(\nu,\beta\right)$.
Note that:

\begin{restatable}{theorem}{paretocdf}

\label{thm:ParetoCDF}For $n\geq2$, $X_{1},\ldots,X_{n}\stackrel{\mathrm{iid}}{\sim}\mathrm{Pareto}\left(\nu,\beta\right)$,
and $M_{n}$ in (\ref{eq:MnAn}), we have
\begin{align}
F_{M_{n}}\left(\mu\right) & =\left(1-\rho^{-\beta\left(\mu-1\right)}\right)^{n-1}\label{eq:ParetoFM}\\
f_{M_{n}}\left(\mu\right) & =\beta\left(n-1\right)\left(1-\rho^{-\beta\left(\mu-1\right)}\right)^{n-2}\rho^{-\beta\left(\mu-1\right)}\log\rho\label{eq:ParetofM}\\
F_{M_{n}}^{-1}\left(q\right) & =1-\frac{1}{\beta\log\rho}\log\left(1-q^{1\left/\left(n-1\right)\right.}\right)\sim\frac{\log n-\log\left(-\log q\right)}{\beta\log\rho},\label{eq:ParetoFMinv}
\end{align}
for $\mu>1$ and $0<q<1$, where the asymptotic result holds as either
or both $n\rightarrow\infty$ and $q\rightarrow1^{-}$. 

\end{restatable}
\begin{proof}
The proof, which uses (2.3.3) of \citet{DN03}, appears in Appendix
\ref{sec:Size-Proofs}.
\end{proof}
The $\mathrm{Pareto}\left(\nu,\beta\right)$ setting---unlike the
$\mathrm{Exp}\left(\lambda\right)$ setting (see Lemma \ref{lem:bndExpMomM}
and Proposition \ref{prop:asympExpM})---permits direct computation
of $\mathbb{E}M_{n}$ and $\mathrm{Var}\left(M_{n}\right)$.

\begin{restatable}{proposition}{paretomgfm}

\label{prop:ParetoMGF-M}For $X_{1},\ldots,X_{n}\stackrel{\mathrm{iid}}{\sim}\mathrm{Pareto}\left(\nu,\beta\right)$
with $n\geq2$ and $M_{n}$ in (\ref{eq:MnAn}), we have $\mathbb{E}e^{tM_{n}}=\left(n-1\right)e^{t}B\left(1-\nicefrac{t}{\beta\log\rho},n-1\right)$,
for $t<\beta\log\rho$, so that
\begin{align*}
\mathbb{E}M_{n} & =1+\frac{\psi\left(n\right)+\gamma}{\beta\log\rho}\sim\frac{\log n}{\beta\log\rho}\\
\mathrm{Var}\left(M_{n}\right) & =\frac{\frac{\pi^{2}}{6}-\psi_{1}\left(n\right)}{\beta^{2}\log^{2}\rho}\sim\frac{\pi^{2}}{6\beta^{2}\log^{2}\rho}\\
\mathrm{Skew}\left(M_{n}\right) & =\frac{2\zeta\left(3\right)+\psi_{2}\left(n\right)}{\left[\frac{\pi^{2}}{6}-\psi_{1}\left(n\right)\right]^{\nicefrac{3}{2}}}\sim\frac{12\sqrt{6}\zeta\left(3\right)}{\pi^{3}},
\end{align*}
where the asymptotic results hold as $n\rightarrow\infty$.

\end{restatable}
\begin{proof}
The proof, which uses Theorem \ref{thm:ParetoCDF}, appears in Appendix
\ref{sec:Size-Proofs}.
\end{proof}
Note that, as $n\rightarrow\infty$,
\begin{equation}
\left\{ \begin{array}{c}
F_{\mathrm{Pareto}\left(\nu,\beta\right),M_{n}}^{-1}\left(q\right)\lesseqqgtr F_{\mathrm{Exp}\left(\lambda\right),M_{n}}^{-1}\left(q\right)\\
\mathbb{E}_{\mathrm{Pareto}\left(\nu,\beta\right)}M_{n}\lesseqqgtr\mathbb{E}_{\mathrm{Exp}\left(\lambda\right)}M_{n}\\
\mathrm{Var}_{\mathrm{Pareto}\left(\nu,\beta\right)}\left(M_{n}\right)\lesseqqgtr\mathrm{Var}_{\mathrm{Exp}\left(\lambda\right)}\left(M_{n}\right)
\end{array}\right\} \iff\beta\gtreqqless1\label{eq:compMom}
\end{equation}
(Theorems \ref{thm:distExpM} and \ref{thm:ParetoCDF}; Propositions
\ref{prop:asympExpM} and \ref{prop:ParetoMGF-M}). If we dial up
(down) the size of Pareto outliers, an exponential histogram of $\mathrm{Pareto}\left(\nu,\beta\right)$
data becomes larger (smaller) than one of $\mathrm{Exp}\left(\lambda\right)$
data, and size has a larger (smaller) variance. It seems strange to
conclude that a histogram of $\mathrm{Exp}\left(\lambda\right)$ values
is larger than a histogram of $\mathrm{Pareto}\left(\nu,\beta\right)$
values, but see our discussion of $\mathrm{Uniform}\left(0,1\right)$
in §\ref{subsec:Heavy-and-Light-Tailed}. We turn to the asymptotic
distribution of $M_{n}$ under $\mathrm{Pareto}\left(\nu,\beta\right)$
sampling.

\begin{restatable}{proposition}{paretolimm}

\label{prop:ParetoLim-M}For $X_{1},X_{2},\ldots,X_{n}\stackrel{\mathrm{iid}}{\sim}\mathrm{Pareto}\left(\nu,\beta\right)$
and $M_{n}$ as in (\ref{eq:MnAn}), we have $\left(M_{n}-1\right)\left(\beta\log\rho\right)-\log n\implies\mathrm{Gumbel}\left(0,1\right)$
as $n\rightarrow\infty$, which implies that, for $n$ large, 
\[
M_{n}\stackrel{\cdot}{\sim}\mathrm{Gumbel}\left(1+\frac{\log n}{\beta\log\rho},\frac{1}{\beta\log\rho}\right),
\]
so that $\mathbb{E}M_{n}\sim\frac{\log n}{\beta\log\rho}$, $\mathrm{Var}\left(M_{n}\right)\sim\frac{\pi^{2}}{6\beta^{2}\log^{2}\rho}$,
and $\mathrm{Skew}\left(M_{n}\right)\sim\frac{12\sqrt{6}\zeta\left(3\right)}{\pi^{3}}$.

\end{restatable}
\begin{proof}
Fixing $y\in\mathbb{R}$ we note that
\begin{align}
\Pr\left(\left(M_{n}-1\right)\left(\beta\log\rho\right)-\log n\leq y\right) & =\Pr\left(M_{n}\leq1+\frac{y+\log n}{\beta\log\rho}\right)\\
 & =\left(1-\frac{e^{-y}}{n}\right)^{n-1}\stackrel[\infty]{n}{\longrightarrow}e^{-e^{-y}},\label{eq:useCDF}
\end{align}
where (\ref{eq:useCDF}) uses (\ref{eq:ParetoFM}). This proves the
first result. The second uses Fact \ref{fact:linGumbel}.
\end{proof}
Comparing Propositions \ref{prop:asympExpM} and \ref{prop:ParetoLim-M},
we see that the $\mathrm{Exp}\left(\lambda\right)$ setting needs
the $\log\log n$ term to place $M_{n}$ into the Gumbel domain of
attraction while the $\mathrm{Pareto}\left(\nu,\beta\right)$ setting
does not (see Remark \ref{rem:altAsymp}). We finally consider $\mathscr{L}\left(A_{n}\right)$:

\begin{restatable}{theorem}{paretoa}

\label{thm:Parato-A}For $X_{1},X_{2},\ldots,X_{n}\stackrel{\mathrm{iid}}{\sim}\mathrm{Pareto}\left(\nu,\beta\right)$
and $A_{n}$ in (\ref{eq:MnAn}), we have
\begin{align}
F_{A_{n}}\left(a\right) & =\left(1-\rho^{-\beta\left(a-1\right)}\right)^{n}\label{eq:ParetoFA}\\
f_{A_{n}}\left(a\right) & =n\beta\left(1-\rho^{-\beta\left(a-1\right)}\right)^{n-1}\rho^{-\beta\left(a-1\right)}\log\rho\label{eq:ParetofA}\\
F_{A_{n}}^{-1}\left(q\right) & =1-\frac{\log\left(1-q^{\nicefrac{1}{n}}\right)}{\beta\log\rho}\sim\frac{\log n-\log\left(-\log q\right)}{\beta\log\rho},\label{eq:ParetoFAinv}
\end{align}
for $a>1$, $0<q<1$, where the asymptotic result holds as either
or both $n\rightarrow\infty$ and $q\rightarrow1^{-}$. This gives
$\mathbb{E}e^{tA_{n}}=ne^{t}B\left(1-\nicefrac{t}{\beta\log\rho},n\right),$
for $t<\beta\log\rho$, so that
\begin{align}
\mathbb{E}A_{n} & =1+\frac{\psi\left(n+1\right)+\gamma}{\beta\log\rho}\sim\frac{\log n}{\beta\log\rho}\label{eq:EA}\\
\mathrm{Var}\left(A_{n}\right) & =\frac{\frac{\pi^{2}}{6}-\psi_{1}\left(n\right)}{\beta^{2}\log^{2}\rho}\sim\frac{\pi^{2}}{6\beta^{2}\log^{2}\rho}\label{eq:VarA}\\
\mathrm{Skew}\left(A_{n}\right) & =\frac{2\zeta\left(3\right)+\psi_{2}\left(n\right)+\frac{2}{n^{3}}}{\left[\frac{\pi^{2}}{6}-\psi_{1}\left(n\right)\right]^{\nicefrac{3}{2}}}\sim\frac{12\sqrt{6}\zeta\left(3\right)}{\pi^{3}}.\label{eq:SkewA}
\end{align}
Finally, we have $\left(A_{n}-1\right)\beta\log\rho-\log n\implies\mathrm{Gumbel}\left(0,1\right)$
as $n\rightarrow\infty$, which implies that, for $n$ large,
\begin{equation}
A_{n}\stackrel{\cdot}{\sim}\mathrm{Gumbel}\left(1+\frac{\log n}{\beta\log\rho},\frac{1}{\beta\log\rho}\right),\label{eq:AGumbel}
\end{equation}
so that $\mathbb{E}A_{n}\sim\frac{\log n}{\beta\log\rho}$, $\mathrm{Var}\left(A_{n}\right)\sim\frac{\pi^{2}}{6\beta^{2}\log^{2}\rho}$,
and $\mathrm{Skew}\left(A_{n}\right)\sim\frac{12\sqrt{6}\zeta\left(3\right)}{\pi^{3}}$.

\end{restatable}
\begin{proof}
See Theorem \ref{thm:ParetoCDF} and Propositions \ref{prop:ParetoMGF-M}--\ref{prop:ParetoLim-M}.
Another approach to (\ref{eq:AGumbel}): the $\log_{\rho}\left(\nicefrac{X_{i}}{\nu}\right)$
are i.i.d.\ $\mathrm{Exp}\left(\beta\log\rho\right)$ and $Z_{\left(n\right)}-\log n\implies\mathrm{Gumbel}\left(0,1\right)$,
as $n\rightarrow\infty$, for $Z_{1},Z_{2},\ldots,Z_{n}$ i.i.d.\ $\mathrm{Exp}\left(1\right)$
(see page 83 of \citet{D05}).
\end{proof}
We note that $F_{M_{n}}\left(\mu\right)=F_{A_{n-1}}\left(\mu\right)$
(see (\ref{eq:ParetoFM}) and (\ref{eq:ParetoFA})): One loses a degree
of freedom in approximating $\nu$ with $X_{\left(1\right)}$. To
summarize, as $n\rightarrow\infty$, the $\mathrm{Exp}\left(\lambda\right)$
and $\mathrm{Pareto}\left(\nu,\beta\right)$ settings give similar
results, modulo $\nicefrac{1}{\beta}$ terms in $\mathrm{Pareto}\left(\nu,\beta\right)$
expressions, so that row-wise expressions below are asymptotically-equivalent:
\[
\begin{array}{rrl}
F_{\mathrm{Exp}\left(\lambda\right),M_{n}}^{-1}\left(q\right) & \beta F_{\mathrm{Pareto}\left(\nu,\beta\right),M_{n}}^{-1}\left(q\right) & \log_{\rho}n-\log_{\rho}\left(-\log q\right)\\
\mathbb{E}_{\mathrm{Exp}\left(\lambda\right)}M_{n} & \beta\mathbb{E}_{\mathrm{Pareto}\left(\nu,\beta\right)}M_{n} & \log_{\rho}n\\
\mathrm{Var}_{\mathrm{Exp}\left(\lambda\right)}\left(M_{n}\right) & \beta^{2}\mathrm{Var}_{\mathrm{Pareto}\left(\nu,\beta\right)}\left(M_{n}\right) & \left(\nicefrac{\pi}{\sqrt{6}\log\rho}\right)^{2}.
\end{array}
\]

\section{Accuracy \label{sec:Accuracy}}

Section \ref{sec:Introduction} presents two takes on accuracy: (\ref{eq:distSumBnj})
gives $n^{-1}\bar{B}_{n,k}\rightarrow F\left(\rho^{k}\right)$ w.p.1
as $n\rightarrow\infty$ while (\ref{eq:absError}) and following
argue that 
\[
\left|X_{\left(\left\lfloor 1+\left(n-1\right)q\right\rfloor \right)}-\hat{X}_{q}\right|\leq\epsilon X_{\left(\left\lfloor 1+\left(n-1\right)q\right\rfloor \right)},
\]
for $\hat{X}_{q}$ the exponential histogram estimate and $0<\epsilon<1$.
The former presents the proximity of an estimate and a population
quantity; the latter presents the proximity of two estimates. We extend
the former view of accuracy. In particular, we present well-known
confidence bands for CDF $F$ and suggest that these provide information
about $F^{-1}$ (see Figure \ref{fig:ConfBands}). Noting that organizations
are often interested in extreme quantiles, we also approximate the
largest and second largest quantile one can estimate under $\mathrm{Exp}\left(\lambda\right)$
and $\mathrm{Pareto}\left(\nu,\beta\right)$ sampling.

\begin{figure}[!t]
\subfloat[DKW confidence bands for $\mathrm{Exp}\left(1\right)$ data.]{\hfill{}\includegraphics[scale=0.405]{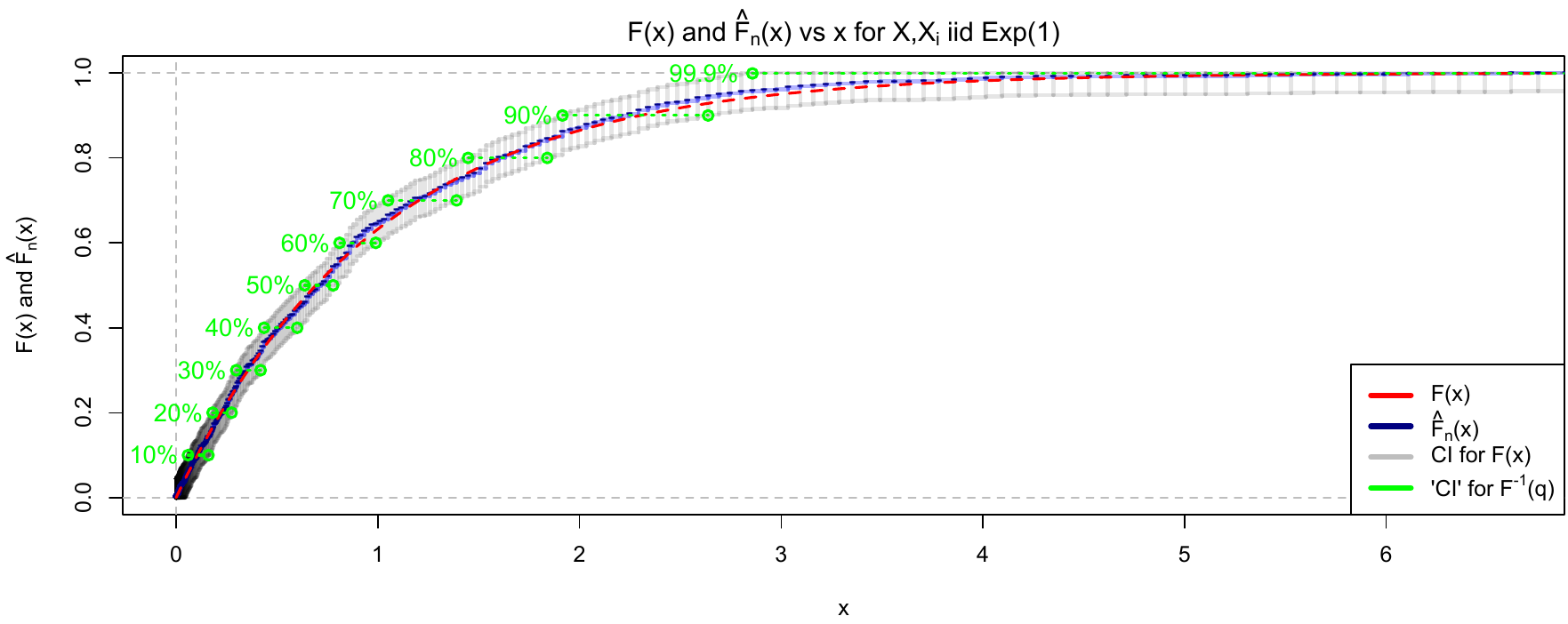}\hfill{}

}

\subfloat[DKW confidence bands for $\mathrm{Pareto}\left(1,1\right)$ data.]{\hfill{}\includegraphics[scale=0.405]{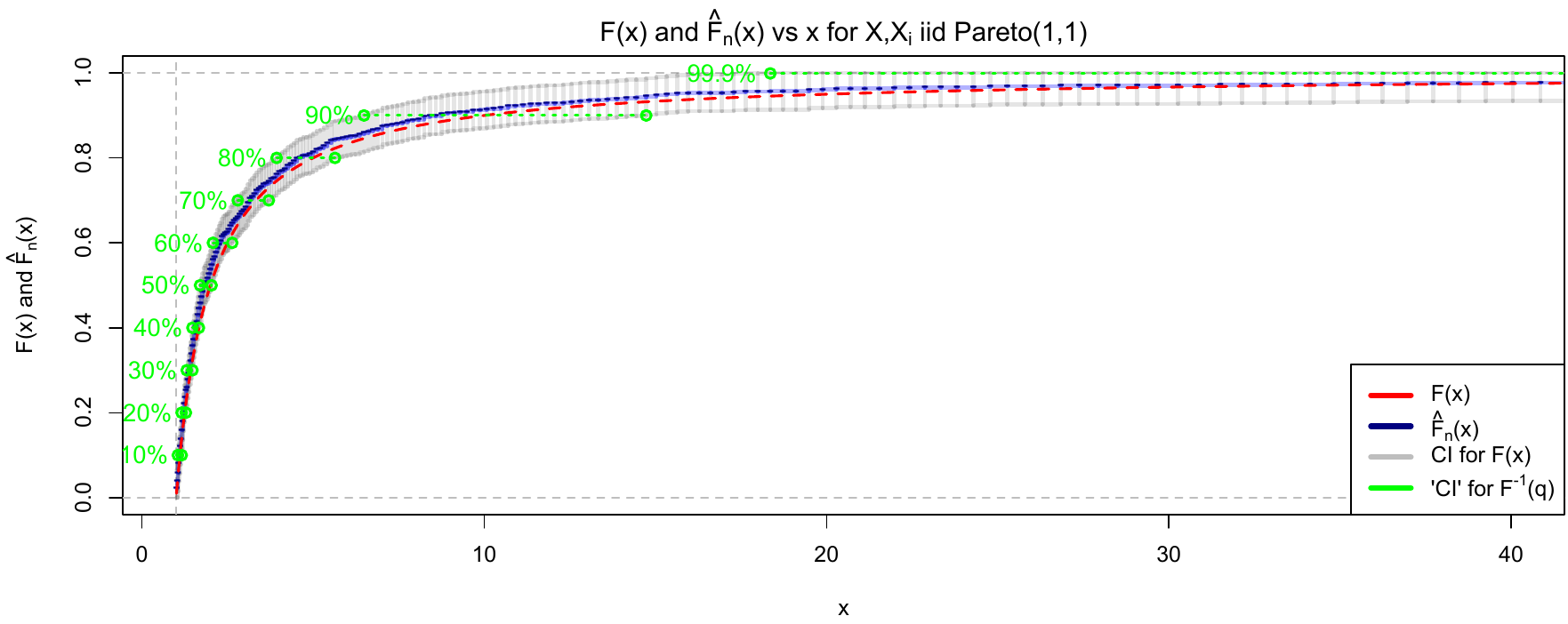}\hfill{}

}

\caption{DKW bands for $\mathrm{Exp}\left(1\right)$ and $\mathrm{Pareto}\left(1,1\right)$
data. Panels show: the CDF in red; the empirical CDF for $n=1000$
simulated data points in navy; the DKW confidence bands in gray; and
inferred quantile intervals in green. The inferred intervals for $q=0.999$
continue to $+\infty$. We use $\alpha=0.05$ and $\rho=\frac{1.01}{0.99}$.}

\label{fig:ConfBands}
\end{figure}

For the empirical CDF $\hat{F}_{n}\left(x\right)\coloneqq\frac{1}{n}\sum_{i=1}^{n}\mathbf{1}_{\left\{ X_{i}\leq x\right\} }$
and $\delta>0$, the Dvoretzky-Kiefer-Wolfowitz (DKW) inequality gives
\[
\Pr\left(\sup_{x\in\mathbb{R}}\left|\hat{F}_{n}\left(x\right)-F\left(x\right)\right|>\delta\right)\leq2e^{-2n\delta^{2}}
\]
when $X_{1},\ldots,X_{n}\stackrel{\mathrm{iid}}{\sim}F$ (\citet{DKW56,M90}).
With $\sup_{x\in\mathbb{R}}$ inside $\Pr\left(\cdot\right)$, the
DKW inequality gives $1-\alpha$ confidence bands for $\left(F\left(x\right)\right)_{x\in\mathbb{R}}$,
namely, $\hat{F}_{n}\left(x\right)\pm\sqrt{\left.\log\left(\nicefrac{2}{\alpha}\right)\right/\left(2n\right)}$,
which become wider (more narrow) as $\alpha\rightarrow0^{+}$ (as
$n\rightarrow\infty$). Observing that 
\[
\left\{ k\in\mathbb{Z}:\left|\frac{1}{n}\bar{B}_{n,k}-F\left(\rho^{k}\right)\right|>\delta\right\} \subset\left\{ x\in\mathbb{R}:\left|\hat{F}_{n}\left(x\right)-F\left(x\right)\right|>\delta\right\} ,
\]
we have 
\[
\Pr\left(\sup_{k\in\mathbb{Z}}\left|\frac{1}{n}\bar{B}_{n,k}-F\left(\rho^{k}\right)\right|>\delta\right)\leq\Pr\left(\sup_{x\in\mathbb{R}}\left|\hat{F}_{n}\left(x\right)-F\left(x\right)\right|>\delta\right),
\]
so that $\frac{1}{n}\bar{B}_{n,k}\pm\sqrt{\left.\log\left(\nicefrac{2}{\alpha}\right)\right/\left(2n\right)}$
gives $1-\alpha$ confidence bands for $\left(F\left(\rho^{k}\right)\right)_{k\in\mathbb{Z}}$. 

Figure \ref{fig:ConfBands} shows 95\% confidence bands for $\mathrm{Exp}\left(1\right)$
and $\mathrm{Pareto}\left(1,1\right)$. Both circumscribe their CDF,
an outcome one expects 9025 times out of 10,000. What do such bands
tell us about $F^{-1}$? While these are not confidence bands for
$F^{-1}$, they are illustrative, \emph{e.g.}, we expect an interval
for $F^{-1}\left(0.999\right)$ to look like $\left(C,\infty\right)$.
With a DKW-type inequality for 
\[
\int_{-\infty}^{\infty}\left|\hat{F}_{n}\left(x\right)-F\left(x\right)\right|dx=\int_{0}^{1}\left|\hat{F}_{n}^{-1}\left(q\right)-F^{-1}\left(q\right)\right|dq
\]
(something like \citet{SM96,JK83}) one could derive confidence bands
for both $\left(F\left(x\right)\right)_{x\in\mathbb{R}}$ and $\left(F^{-1}\left(q\right)\right)_{q\in\left[0,1\right]}$.\footnote{\citet{R00,ABBRS13} give confidence bands for the quantile function
in the i.i.d.\ normal setting. If we thought the data were normal,
we could use these.}

\subsection{The Exponential Setting}

Organizations guard against extreme events, \emph{e.g.}, high latency,
high error rate, extreme heat, and excessive traffic. Say latency
has an exponential distribution. We consider the following: If an
organization uses an exponential histogram to store latency data,
will it be able to accurately estimate extreme quantiles like $F^{-1}\left(0.99999\right)$?
The concern here is not accuracy per se, but estimability.

We bound both the expected mass to the right of the histogram and
the expected mass in the right-most bin. That is, for final bin $\left\lceil J_{n}\right\rceil $,
we bound
\begin{equation}
p_{\mathrm{last}}\coloneqq\Pr\left(\rho^{\left\lceil J_{n}\right\rceil -1}<X\leq\rho^{\left\lceil J_{n}\right\rceil }\right)\quad\textrm{and}\quad p_{\mathrm{tail}}\coloneqq\Pr\left(X>\rho^{\left\lceil J_{n}\right\rceil }\right),\label{eq:plast-ptail}
\end{equation}
for $X,X_{1},X_{2},\ldots,X_{n}\stackrel{\mathrm{iid}}{\sim}\mathrm{Exp}\left(\lambda\right)$
and $J_{n}\coloneqq\log_{\rho}X_{\left(n\right)}$.

\begin{restatable}{theorem}{bndexpprob}

\label{thm:BndExpProb}For $X,X_{1},X_{2},\ldots,X_{n}\stackrel{\mathrm{iid}}{\sim}\mathrm{Exp}\left(\lambda\right)$,
$J_{n}\coloneqq\log_{\rho}X_{\left(n\right)}$, and $p_{\mathrm{last}}$
and $p_{\mathrm{tail}}$ in (\ref{eq:plast-ptail}), we have
\begin{gather}
p_{\mathrm{last}}\leq nB\left(1+\nicefrac{1}{\rho},n\right)-nB\left(1+\rho,n\right)\sim\frac{\Gamma\left(1+\nicefrac{1}{\rho}\right)}{n^{\nicefrac{1}{\rho}}}-\frac{\Gamma\left(1+\rho\right)}{n^{\rho}}\label{eq:ExpLastProb}\\
\frac{\Gamma\left(1+\rho\right)}{n^{\rho}}\sim nB\left(1+\rho,n\right)\leq p_{\mathrm{tail}}\leq\frac{1}{n+1}\sim\frac{1}{n},\label{eq:ExpRest}
\end{gather}
where the limits hold as $n\rightarrow\infty$.

\end{restatable}
\begin{proof}
The proof appears in Appendix \ref{sec:Accuracy-Proofs}.
\end{proof}

\begin{table}
\hfill{}%
\begin{tabular}{r|lr@{\extracolsep{0pt}.}l|lr@{\extracolsep{0pt}.}l}
 & \multicolumn{3}{c|}{$\mathrm{Exp}\left(\lambda\right)$} & \multicolumn{3}{c}{$\mathrm{Pareto}\left(\nu,1\right)$}\tabularnewline
$\log_{10}n$ & $p_{\mathrm{last}}$ & \multicolumn{2}{c|}{$q_{\mathrm{max}}$} & $p_{\mathrm{last}}$ & \multicolumn{2}{c}{$q_{\mathrm{max}}$}\tabularnewline
\hline 
1 & $7.3\times10^{-3}$ & 0&9127160 & $3.6\times10^{-3}$ & 0&9108911\tabularnewline
2 & $1.7\times10^{-3}$ & 0&9909028 & $4.0\times10^{-4}$ & 0&9902951\tabularnewline
3 & $2.6\times10^{-4}$ & 0&9991236 & $4.0\times10^{-5}$ & 0&9990208\tabularnewline
4 & $3.5\times10^{-5}$ & 0&9999163 & $4.0\times10^{-6}$ & 0&9999020\tabularnewline
5 & $4.5\times10^{-6}$ & 0&9999920 & $4.0\times10^{-7}$ & 0&9999902\tabularnewline
6 & $5.4\times10^{-7}$ & 0&9999992 & $4.0\times10^{-8}$ & 0&9999990\tabularnewline
7 & $6.4\times10^{-8}$ & 0&9999999 & $4.0\times10^{-9}$ & 0&9999999\tabularnewline
8 & $7.3\times10^{-9}$ & 1&0000000 & $4.0\times10^{-10}$ & 1&0000000\tabularnewline
\end{tabular}\hfill{}

\caption{Rounded upper bounds for $p_{\mathrm{last}}$ and $q_{\mathrm{max}}\protect\coloneqq1-p_{\mathrm{tail}}$
(see (\ref{eq:plast-ptail})) when $\rho=\frac{1.01}{0.99}\approx1.0202$.
See (\ref{eq:ExpLastProb}) and (\ref{eq:ExpRest}) for $\mathrm{Exp}\left(\lambda\right)$
and (\ref{eq:ParTailRest}) for $\mathrm{Pareto}\left(\nu,\beta\right)$.}

\label{tab:tailMass}
\end{table}

Table \ref{tab:tailMass} bounds the expected mass of the last bin
and the largest quantile we can hope to estimate under $\mathrm{Exp}\left(\lambda\right)$
sampling when $\rho=\frac{1.01}{0.99}$ and $1\leq\log_{10}n\leq8$.
We see that $p_{\mathrm{last}}\lesssim\nicefrac{1}{n}$ and $q_{\mathrm{max}}\coloneqq1-p_{\mathrm{tail}}\lesssim1-\nicefrac{1}{n}$.
For $\rho$ fixed we have $p_{\mathrm{last}}=O\left(n^{\nicefrac{-1}{\rho}}\right)$
and $p_{\mathrm{tail}}=\mathcal{O}\left(\nicefrac{1}{n}\right)$.

\subsection{The Pareto Setting}

We turn to the setting in which $X,X_{1},X_{2},\ldots,X_{n}\stackrel{\mathrm{iid}}{\sim}\mathrm{Pareto}\left(\nu,\beta\right)$.

\begin{restatable}{theorem}{bndparprob}

\label{thm:BndParProb}For $X,X_{1},X_{2},\ldots,X_{n}\stackrel{\mathrm{iid}}{\sim}\mathrm{Pareto}\left(\nu,\beta\right)$,
$J_{n}\coloneqq\log_{\rho}X_{\left(n\right)}$, and $p_{\mathrm{last}}$
and $p_{\mathrm{tail}}$ in (\ref{eq:plast-ptail}), we have
\begin{equation}
p_{\mathrm{last}}\leq\frac{\rho^{\beta}-\nicefrac{1}{\rho^{\beta}}}{n+1}\quad\textrm{and}\quad\frac{\nicefrac{1}{\rho^{\beta}}}{n+1}\leq p_{\mathrm{tail}}\leq\frac{1}{n+1}.\label{eq:ParTailRest}
\end{equation}

\end{restatable}
\begin{proof}
The proof, which uses Theorem \ref{thm:Parato-A}, appears in Appendix
\ref{sec:Accuracy-Proofs}.
\end{proof}
Table \ref{tab:tailMass} bounds $p_{\mathrm{last}}$ and $q_{\mathrm{max}}$
under $\mathrm{Pareto}\left(\nu,1\right)$ sampling when $\rho=\frac{1.01}{0.99}$
and $1\leq\log_{10}n\leq8$. In this setting $p_{\mathrm{last}}\lesssim\nicefrac{4}{100n}$
and we again have $q_{\mathrm{max}}\lesssim1-\nicefrac{1}{n}$. The
mass of the last bin increases as $\beta$ or $\rho$ increase, \emph{i.e.},
as outliers become scarce or as bin sizes increase (so that $\epsilon$
increases). For $\beta$ and $\rho$ fixed, both $p_{\mathrm{last}}$
and $p_{\mathrm{tail}}$ are $\mathcal{O}\left(\nicefrac{1}{n}\right)$.

\section{Occupancy \label{sec:Occupancy}}

In this section and the next we adopt a new mental picture. Our histogram
now becomes an infinite (or, in $\mathrm{Exp}\left(\lambda\right)$'s
case, doubly-infinite) sequence of probability urns (or, buckets).
Each value $X_{n}\sim F$ becomes a ball thrown independently into
an urn, the $n$th ball falling into urn $k$ with probability 
\[
p_{k}\coloneqq\Pr\left(\rho^{k-1}<X\leq\rho^{k}\right),
\]
for $X\sim F$. While $\sum_{k\in\mathbb{Z}}p_{k}=1$, we let $Y_{i}\coloneqq\left\lceil \log_{\rho}X_{i}\right\rceil $,
for $i\geq1$, so that i.i.d.\ $Y_{i}\sim\left(p_{j}\right)$ summarize
the urn indices of the i.i.d.\ $X_{i}\sim F$. We say bin $k\in\mathbb{Z}$
is occupied at time $n\geq1$ if $B_{n,k}>0$.

Occupancy, 
\begin{equation}
K_{n}\coloneqq\sum_{j\in\mathbb{Z}}\mathbf{1}_{\left\{ B_{n,j}>0\right\} }=\sum_{j\in\mathbb{Z}}\sum_{r=1}^{\infty}\mathbf{1}_{\left\{ B_{n,j}=r\right\} }\eqqcolon\sum_{r=1}^{\infty}K_{n,r},
\end{equation}
the number of occupied bins, is bounded above by size, \emph{i.e.},
$K_{n}\leq M_{n}$. $K_{n,r}$ gives the number of bins that summarize
$r$ of the $X_{1},X_{2},\ldots,X_{n}$. Our goal in studying occupancy
(and gap sizes in §\ref{sec:Longest-Gap}) is to determine whether
one wastes a significant amount of space by storing (an amortized
superset of) the counts
\begin{equation}
\left\{ B_{n,Y_{\left(1\right)}},B_{n,Y_{\left(1\right)}+1},\ldots,B_{n,Y_{\left(n\right)}}\right\} .\label{eq:fullList}
\end{equation}
We approximate the number of non-zero $B_{n,k}$ in (\ref{eq:fullList}).
This section and the next confirm that, for $\mathrm{Exp}\left(\lambda\right)$
and $\mathrm{Pareto}\left(\nu,\beta\right)$, one does not waste significant
space in storing the full list of counts in (\ref{eq:fullList})---our
hunch going into this analysis.

Occupancy in the presence of an infinite number of urns is an area
of enduring concern in applied probability---work bookended by \citet{K67}'s
seminal insights and \citet{GHP07}'s excellent review. We review
two, key results below, which sections \ref{subsec:Occupancy-Exp}
and \ref{subsec:Occupancy-Par} then apply. Regarding $\mathbb{E}K_{n}$
we have:

\begin{restatable}{theorem}{karlin}

\label{thm:Karlin}For $Y_{i}$ i.i.d.\ with support $\mathbb{Z}_{+}\coloneqq\left\{ 1,2,\ldots\right\} $,
let $p_{k}\coloneqq\Pr\left(Y_{1}=k\right)$ and assume that $p_{1}\geq p_{2}\geq\cdots$.
For $x>0$ let $\alpha\left(x\right)\coloneqq\max\left\{ k\geq1:p_{k}\geq\nicefrac{1}{x}\right\} $.
If $\alpha$ is slowly-varying (\emph{i.e.}, $\alpha\left(cx\right)\sim\alpha\left(x\right)$
for each fixed $c>0$, as $x\rightarrow\infty$), then $\mathbb{E}K_{n}\sim\alpha\left(n\right)$,
as $n\rightarrow\infty$.

\end{restatable}
\begin{proof}
Please see Theorem $1'$ in section 2 of \citet{K67}. Put ``$\gamma=0.$''
\end{proof}
\citet{DGP18} gives finite-$n$ bounds for $\mathbb{E}K_{n}$. For
$\mathrm{Var}\left(K_{n}\right)$ we have:

\begin{restatable}{theorem}{bogachevetal}

\label{thm:Bogachev-etal}For $Y_{i}$ i.i.d.\ with support $\mathbb{Z}_{+}\coloneqq\left\{ 1,2,\ldots\right\} $,
let $p_{k}\coloneqq\Pr\left(Y_{1}=k\right)$ and assume that $p_{1}\geq p_{2}\geq\cdots$.
Let $v_{n}\coloneqq\mathrm{Var}\left(K_{n}\right)$. Then, 
\begin{enumerate}
\item \label{enu:Bogachev-etal-1.1}$\left(v_{n}\right)_{n=1}^{\infty}$
approaches a finite limit $\iff$ $\lim_{j\rightarrow\infty}\frac{p_{j+k}}{p_{j}}=\frac{1}{2}$
for some $k\geq1$, in which case $\lim_{n\rightarrow\infty}v_{n}=k$;
\item \label{enu:Bogachev-etal-1.2}$\left(v_{n}\right)_{n=1}^{\infty}$
is bounded $\iff$ $\limsup_{j\rightarrow\infty}\frac{p_{j+k}}{p_{j}}\leq\frac{1}{2}$
for some $k\geq1$. Then, 
\[
k_{\mathrm{min}}\coloneqq\inf\left\{ k\geq1:\limsup_{j\rightarrow\infty}\,\frac{p_{j+k}}{p_{j}}\leq\frac{1}{2}\right\} 
\]
provides an asymptotically sharp bound for $\limsup_{n\rightarrow\infty}v_{n}$.
\end{enumerate}
\end{restatable}
\begin{proof}
Please refer to Theorems 1.1 and 1.2 of \citet{BGY08}.
\end{proof}

\subsection{The Exponential Setting \label{subsec:Occupancy-Exp}}

In applying Theorems \ref{thm:Karlin}--\ref{thm:Bogachev-etal}
to $\mathrm{Exp}\left(\lambda\right)$ data, we encounter the following
mismatch: these theorems call for a singly-infinite sequence of urns
with decreasing probabilities while $\mathrm{Exp}\left(\lambda\right)$
produces a doubly-infinite sequence of urns with increasing and then
decreasing probabilities.\footnote{$\Pr\left(\rho^{k-1}<X\leq\rho^{k}\right)$ is unimodal $\impliedby-\log_{\rho}X\sim\mathrm{Gumbel}\left(\log_{\rho}\lambda,\frac{1}{\log\rho}\right)$,
$X\sim\mathrm{Exp}\left(\lambda\right)$.} For $X\sim\mathrm{Exp}\left(\lambda\right)$, $x>0$, and
\begin{align}
\kappa_{\lambda,\rho}\coloneqq1+\log_{\rho}\left(\frac{\log\rho}{\lambda\left(\rho-1\right)}\right) & =\underset{-\infty<y<\infty}{\arg\max}\Pr\left(\rho^{y-1}<X\leq\rho^{y}\right)\label{eq:kappa}\\
 & =\underset{-\infty<y<\infty}{\arg\max}\left\{ \exp\left(-\lambda\rho^{y-1}\right)-\exp\left(-\lambda\rho^{y}\right)\right\} ,
\end{align}
we therefore split $F\left(x\right)=\Pr\left(X\leq x\right)$ into
two parts:
\begin{align*}
F\left(x\right) & =F_{0}\left(x\right)\Pr\left(X\leq\kappa_{\lambda,\rho}\right)+F_{1}\left(x\right)\Pr\left(X>\kappa_{\lambda,\rho}\right)\\
 & =F_{0}\left(x\right)\left(1-\rho^{-\frac{\rho}{\rho-1}}\right)+F_{1}\left(x\right)\rho^{-\frac{\rho}{\rho-1}},
\end{align*}
where
\begin{align}
F_{0}\left(x\right) & \coloneqq\left\{ \begin{array}{ll}
0 & \textrm{ if }x\leq0\\
\frac{1-e^{-\lambda x}}{1-\rho^{-\frac{\rho}{\rho-1}}} & \textrm{ if }0<x<\frac{\rho\log\rho}{\lambda\left(\rho-1\right)}\\
1 & \textrm{ if }x\geq\frac{\rho\log\rho}{\lambda\left(\rho-1\right)}
\end{array}\right.\label{eq:F0}\\
F_{1}\left(x\right) & \coloneqq\left\{ \begin{array}{ll}
0 & \textrm{ if }x\leq\frac{\rho\log\rho}{\lambda\left(\rho-1\right)}\\
1-e^{-\lambda\left(x-\frac{\rho\log\rho}{\lambda\left(\rho-1\right)}\right)} & \textrm{ if }x>\frac{\rho\log\rho}{\lambda\left(\rho-1\right)}.
\end{array}\right.\label{eq:F1}
\end{align}
We analyze $F_{0}$ and $F_{1}$ separately and then stitch the results
so-obtained back together. Regarding $\mathbb{E}K_{n}$ we have:

\begin{restatable}{proposition}{karlinexp}

\label{prop:KarlinExp}Note that:
\begin{enumerate}
\item \label{enu:E0Kn}If $X_{1},X_{2},\ldots,X_{n}\stackrel{\mathrm{iid}}{\sim}F_{0}$
in (\ref{eq:F0}), then, as $n\rightarrow\infty$,
\[
\mathbb{E}K_{n}\sim\log_{\rho}\left(\frac{n\log\rho}{1-\rho^{-\frac{\rho}{\rho-1}}}\right)+\frac{1}{\log\rho}W_{0}\left(-\frac{\left(1-\rho^{-\frac{\rho}{\rho-1}}\right)\sqrt{\rho}}{\left(\rho-1\right)n}\right)\sim\log_{\rho}n.
\]
\item \label{enu:E1Kn}If $X_{1},X_{2},\ldots,X_{n}\stackrel{\mathrm{iid}}{\sim}F_{1}$
in (\ref{eq:F1}) and $\eta_{\rho,n}\coloneqq\log\left(n\log\rho\right)+\frac{\rho\log\rho}{\rho-1}$,
then, as $n\rightarrow\infty$,
\begin{align*}
\mathbb{E}K_{n} & \sim\log_{\rho}\left(-\frac{\rho-1}{\log\rho}W_{-1}\left(-\frac{\rho^{-\frac{\rho}{\rho-1}}}{n\log\rho}\right)\right)\\
 & \leq\log_{\rho}\left(\frac{\rho-1}{\log\rho}\left(\eta_{\rho,n}+\sqrt{2\left(\eta_{\rho,n}-1\right)}\right)\right)\sim\log_{\rho}\log_{\rho}n.
\end{align*}
\item \label{enu:EKn}If $X_{1},X_{2},\ldots,X_{n}\stackrel{\mathrm{iid}}{\sim}\mathrm{Exp}\left(\lambda\right)$,
then 
\[
\mathbb{E}K_{n}\left\{ \begin{array}{l}
\approx\mathbb{E}_{0}K_{n_{0}}+\mathbb{E}_{1}K_{n_{1}}\\
\lesssim\log_{\rho}n\textrm{ as }n\rightarrow\infty,
\end{array}\right.
\]
where ``$\mathbb{E}_{b}$'' indicates expectation with respect to
$F_{b}$ in (\ref{eq:F0}) or (\ref{eq:F1}) and 
\[
n_{0}\coloneqq n\times\left(1-\rho^{-\frac{\rho}{\rho-1}}\right)\textrm{ and }n_{1}\coloneqq n\times\rho^{-\frac{\rho}{\rho-1}}
\]
give the expected numbers of values below and above cutoff $\rho^{\kappa_{\lambda,\rho}}=\frac{\rho\log\rho}{\lambda\left(\rho-1\right)}$.
\end{enumerate}
\end{restatable}
\begin{proof}
The proof, which uses Theorem \ref{thm:Karlin}, appears in Appendix
\ref{sec:Occupancy-Proofs}.
\end{proof}
\begin{figure}[!t]
\subfloat[Occupancy counts and expected counts. Box plots show simulated $K_{n}$
for $10^{3}$ data sets with $1\protect\leq\log_{10}n\protect\leq8$.
Red curves use the full approximations of $\mathbb{E}K_{n}$ in Proposition
\ref{prop:KarlinExp}. Blue curves use summary approximations $\log_{\rho}n$
(outer panels) and $\log_{\rho}\log_{\rho}n$ (central panel).]{\hfill{}\includegraphics[scale=0.47]{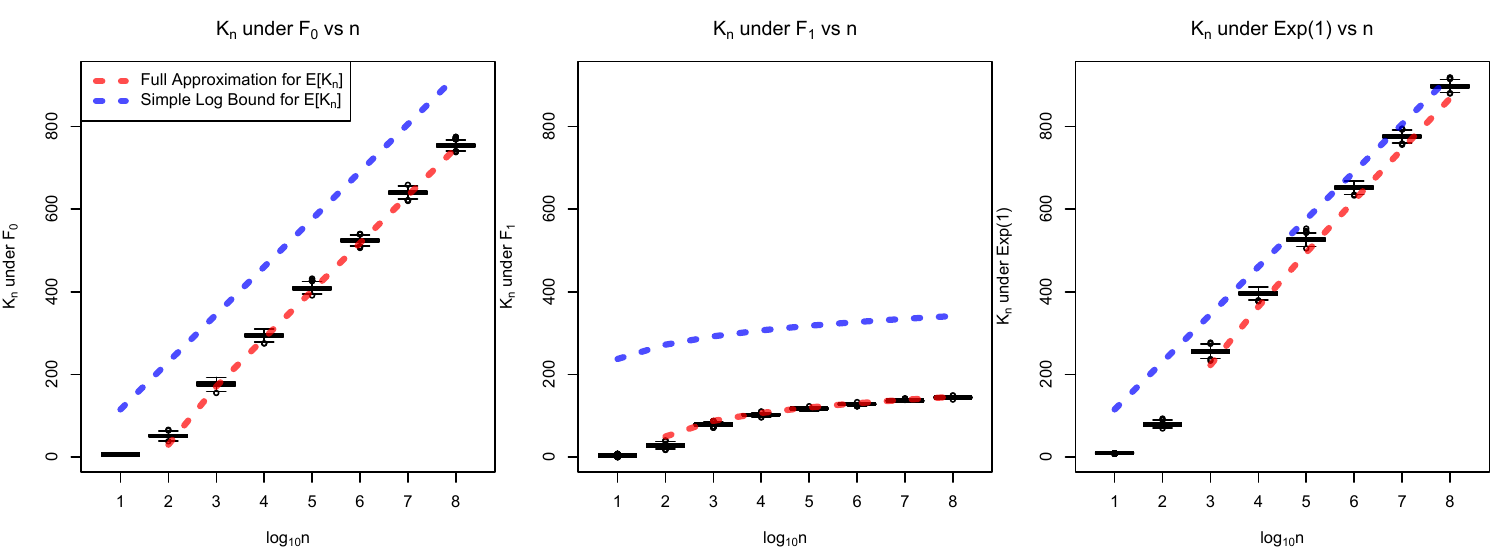}\hfill{}

\label{sub-fig:expKandEK}}

\subfloat[The variance of occupancy counts. Green curves give the variances
of box plots in Figure \ref{sub-fig:expKandEK}. Gray lines give bounds
$\left(\left\lceil \log_{\rho}2\right\rceil ,1,\left\lceil \log_{\rho}2\right\rceil +1\right)=\left(35,1,36\right)$
from Proposition \ref{prop:BogachevExp}.]{\hfill{}\includegraphics[scale=0.47]{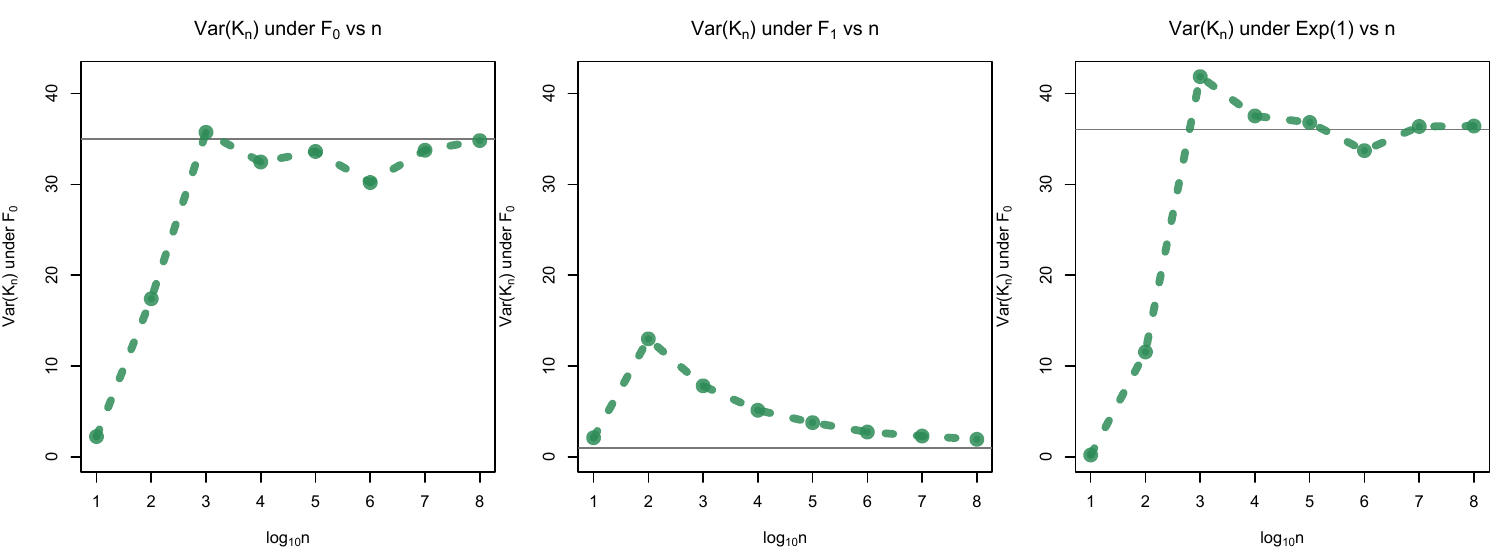}\hfill{}

\label{sub-fig:expVarKn}}

\caption{Occupancy in the $\mathrm{Exp}\left(1\right)$ setting (Propositions
\ref{prop:KarlinExp}--\ref{prop:BogachevExp}). Left panels sample
from $F_{0}$ in (\ref{eq:F0}) with $\lambda=1$; middle panels sample
from $F_{1}$ in (\ref{eq:F1}) with $\lambda=1$; right panels sample
from $\mathrm{Exp}\left(1\right)$. Histograms use $\rho=\nicefrac{1.01}{0.99}\approx1.02$.}
\end{figure}

Figure \ref{sub-fig:expKandEK} shows good correspondence between
our approximations of $\mathbb{E}K_{n}$ and simulated occupancy counts,
especially under $F_{0}$ and $F_{1}$. Under $F$, approximations
$\mathbb{E}_{0}K_{n_{0}}+\mathbb{E}_{1}K_{n_{1}}$ and $\log_{\rho}n$
under- and overshoot occupancy. Regarding $\mathrm{Var}\left(K_{n}\right)$
we have:

\begin{restatable}{proposition}{bogachevexp}

\label{prop:BogachevExp}Letting $\mathbb{Z}_{+}\coloneqq\left\{ 1,2,\ldots\right\} $,
we note that:
\begin{enumerate}
\item \label{enu:Var0Kn}If $X_{1},X_{2},\ldots,X_{n}\stackrel{\mathrm{iid}}{\sim}F_{0}$
in (\ref{eq:F0}), then
\begin{align*}
\log_{\rho}2\in\mathbb{Z}_{+} & \implies\lim_{n\rightarrow\infty}\mathrm{Var}\left(K_{n}\right)=\log_{\rho}2\\
\log_{\rho}2\in\left.\mathbb{R}\right\backslash \mathbb{Z}_{+} & \implies\limsup_{n\rightarrow\infty}\mathrm{Var}\left(K_{n}\right)\leq\left\lceil \log_{\rho}2\right\rceil .
\end{align*}
\item \label{enu:Var1Kn}If $X_{1},X_{2},\ldots,X_{n}\stackrel{\mathrm{iid}}{\sim}F_{1}$
in (\ref{eq:F1}), then $\limsup_{n\rightarrow\infty}\mathrm{Var}\left(K_{n}\right)\leq1$.
\item \label{enu:VarKn}If $X_{1},X_{2},\ldots,X_{n}\stackrel{\mathrm{iid}}{\sim}\mathrm{Exp}\left(\lambda\right)$,
then, as $n\rightarrow\infty$, 
\[
\mathrm{Var}\left(K_{n}\right)\left\{ \begin{array}{l}
\approx\mathrm{Var}_{0}\left(K_{n_{0}}\right)+\mathrm{Var}_{1}\left(K_{n_{1}}\right)\\
\lesssim\log_{\rho}^{2}n,
\end{array}\right.
\]
where $\mathrm{Var}_{b}\left(X\right)\coloneqq\mathbb{E}_{b}\left[\left(X-\mathbb{E}_{b}X\right)\right]$
takes the variance using $F_{b}$ in (\ref{eq:F0}) or (\ref{eq:F1})
and 
\[
n_{0}\coloneqq n\times\left(1-\rho^{-\frac{\rho}{\rho-1}}\right)\textrm{ and }n_{1}\coloneqq n\times\rho^{-\frac{\rho}{\rho-1}}
\]
give the expected numbers of values below and above cutoff $\rho^{\kappa_{\lambda,\rho}}=\frac{\rho\log\rho}{\lambda\left(\rho-1\right)}$.
Note finally that $\limsup_{n\rightarrow\infty}\left\{ \mathrm{Var}_{0}\left(K_{n_{0}}\right)+\mathrm{Var}_{1}\left(K_{n_{1}}\right)\right\} \leq\left\lceil \log_{\rho}2\right\rceil +1$.
\end{enumerate}
\end{restatable}
\begin{proof}
The proof, which uses Theorem \ref{thm:Bogachev-etal}, appears in
Appendix \ref{sec:Occupancy-Proofs}.
\end{proof}
Figure \ref{sub-fig:expVarKn} compares our approximations of $\limsup_{n\rightarrow\infty}\mathrm{Var}\left(K_{n}\right)$
with the variances of simulated occupancies. Simulations corroborate
our analytic results for $F_{0}$ and $F_{1}$, $\left\lceil \log_{\rho}2\right\rceil $
and 1. Simulations also show that our analytic result for $\mathrm{Exp}\left(\lambda\right)$,
$\log_{\rho}^{2}n$, is not tight. In fact, $\left\lceil \log_{\rho}2\right\rceil +1$,
the bound suggested by $\limsup_{n\rightarrow\infty}\left\{ \mathrm{Var}_{0}\left(K_{n_{0}}\right)+\mathrm{Var}_{1}\left(K_{n_{1}}\right)\right\} $,
works well. The variance result in Proposition \ref{prop:asympExpM}
provides corroborating evidence.

\begin{figure}[!t]
\hfill{}\subfloat[Under $F_{0}$ in (\ref{eq:F0}).]{\includegraphics[scale=0.33]{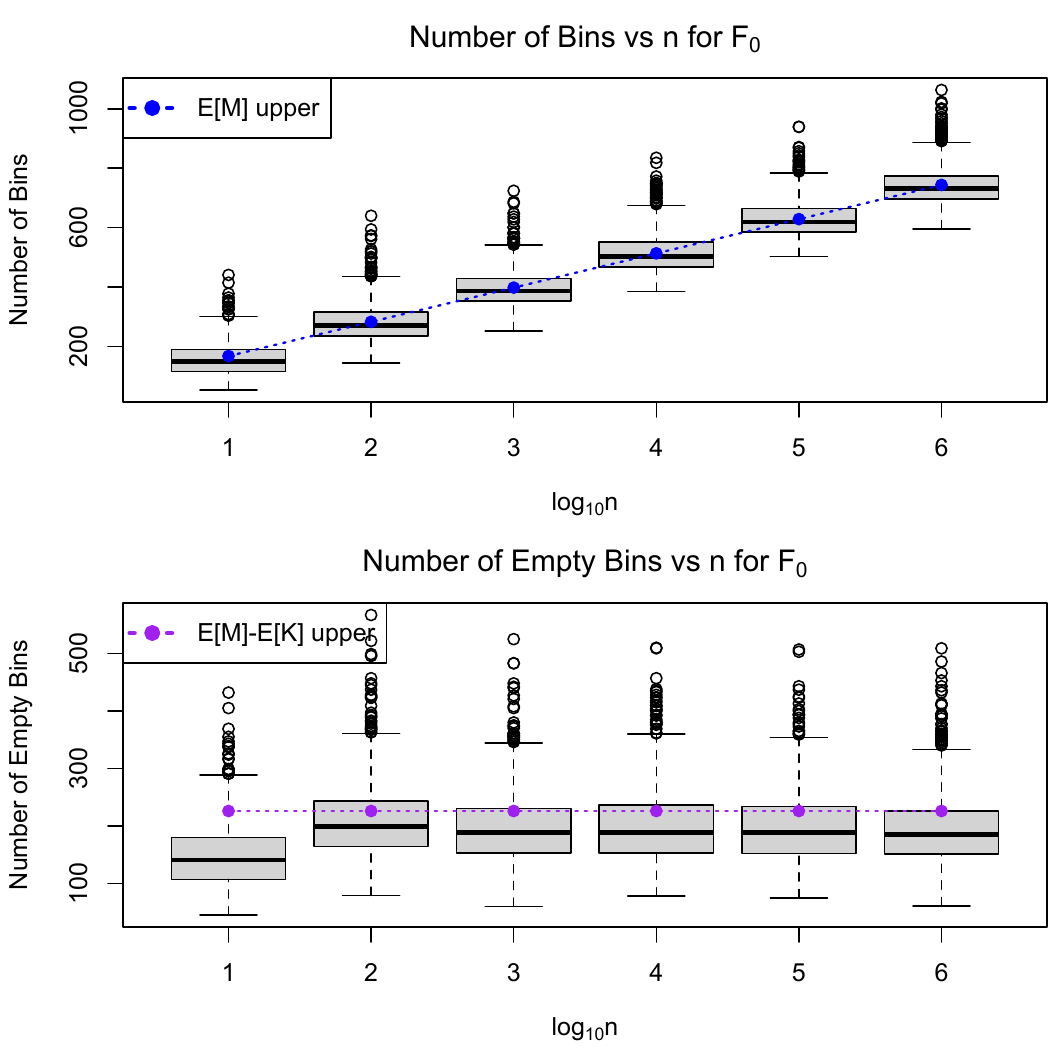}

}\subfloat[Under $F_{1}$ in (\ref{eq:F1}).]{\includegraphics[scale=0.33]{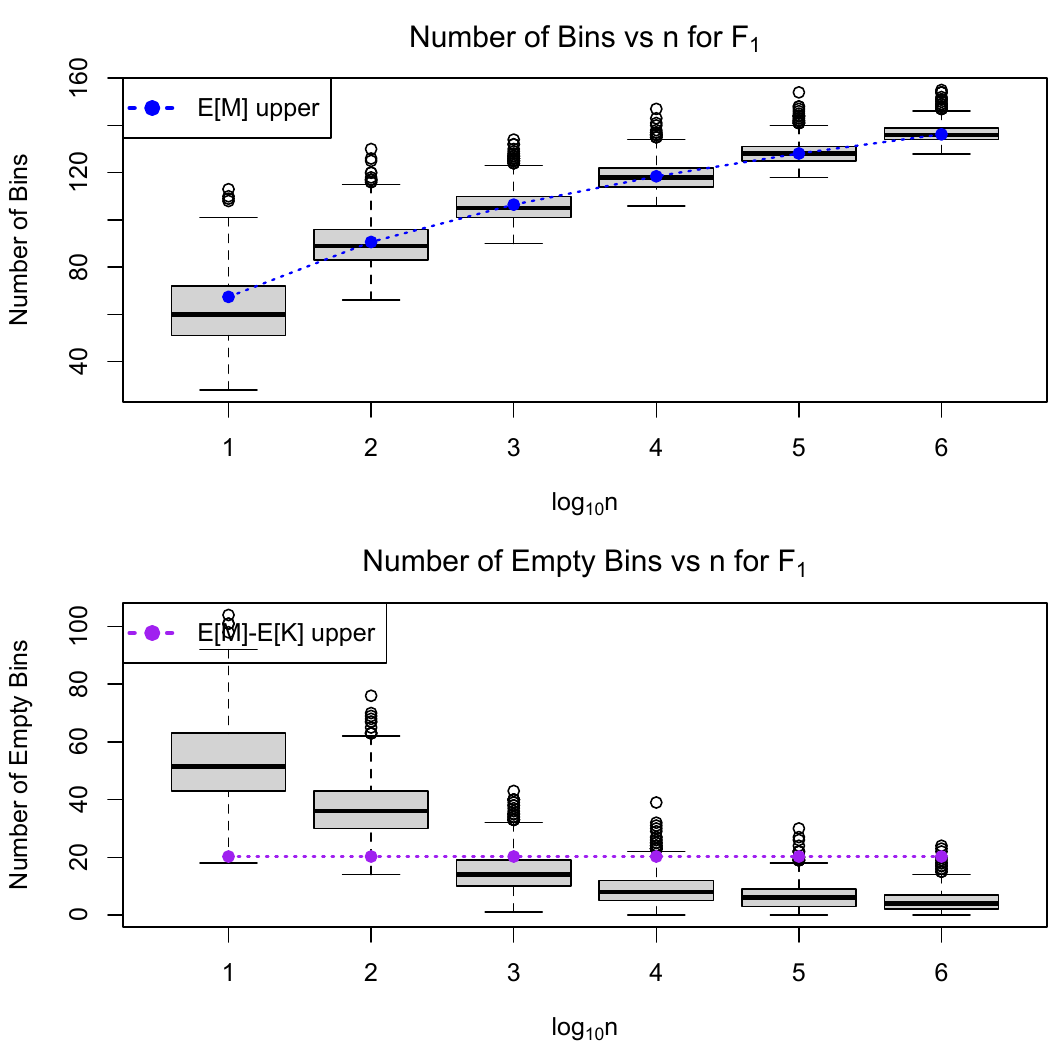}

}\hfill{}

\caption{The number of (empty) bins versus $1\protect\leq\log_{10}n\protect\leq6$
under $F_{0}$ and $F_{1}$ in (\ref{eq:F0}) and (\ref{eq:F1}).
We use $\rho=\frac{1.01}{0.99}$. See Theorem \ref{thm:bndExpMissing}
for the asymptotic bounds.}

\label{fig:emptyBins}
\end{figure}

Comparing Propositions \ref{prop:asympExpM} and \ref{prop:KarlinExp}
we see that both $\mathbb{E}M_{n}$ and $\mathbb{E}K_{n}$ are $\sim\log_{\rho}n$---an
asymptotically insignificant number of bins are empty. That said,
how many bins \emph{are} empty? The following bounds the mean number
as $n\rightarrow\infty$.

\begin{restatable}{theorem}{mnminuskn}

\label{thm:bndExpMissing}Let $\mathbb{E}_{b}$ represent expectation
with respect to $F_{b}$ in (\ref{eq:F0}) or (\ref{eq:F1}), $p_{0}\coloneqq1-\rho^{-\frac{\rho}{\rho-1}}$,
and $n\rightarrow\infty$. Then, for empty cell count $E_{n}\coloneqq M_{n}-K_{n}$,
\[
\begin{array}{lll}
\mathbb{E}_{0}M_{n}\lesssim\log_{\rho}\left(\frac{n\rho\log\rho}{\left(\rho-1\right)p_{0}}\right)+\frac{\gamma}{\log\rho}+1 & \implies & \mathbb{E}_{0}E_{n}\lesssim\log_{\rho}\left(\frac{\rho}{\rho-1}\right)+\frac{\gamma}{\log\rho}+1\\
\mathbb{E}_{1}M_{n}\lesssim\log_{\rho}\left(\frac{\left(\rho-1\right)\left(\gamma+\log n\right)}{\rho\log\rho}+1\right)+1 & \implies & \mathbb{E}_{1}E_{n}\lesssim\log_{\rho}\left(\frac{3}{2}\right).
\end{array}
\]
These results imply that $\mathbb{E}E_{n}\lesssim\log_{\rho}\left(\frac{3\rho}{2\left(\rho-1\right)}\right)+\frac{\gamma}{\log\rho}+1$
in the $\mathrm{Exp}\left(\lambda\right)$ setting.

\end{restatable}
\begin{proof}
The proof, which uses Proposition \ref{prop:KarlinExp}, appears in
Appendix \ref{sec:Occupancy-Proofs}.
\end{proof}
Note in particular that, with $\rho=\frac{1.01}{0.99}$, $\mathbb{E}_{0}E_{n}\lesssim226$
and $\mathbb{E}_{1}E_{n}\lesssim21$. Under $\mathrm{Exp}\left(\lambda\right)$,
as the number of summarized values grows large, the expected number
of empty bins is bounded above by 247 (see Figure \ref{fig:emptyBins}).
That said, bounds on the expected number of empty bins grow large
as $\rho\rightarrow1^{+}$.

\subsection{The Pareto Setting \label{subsec:Occupancy-Par}}

For $X_{i}\stackrel{\mathrm{iid}}{\sim}\mathrm{Pareto}\left(\nu,\beta\right)$,
we have bin indices $Y_{i}\stackrel{\mathrm{iid}}{\sim}\mathrm{Geometric}\left(1-\nicefrac{1}{\rho^{\beta}}\right)$,
so that:

\begin{restatable}{theorem}{paretoocc}

\label{thm:paretoOcc}For $\mathbb{Z}_{+}\coloneqq\left\{ 1,2,\ldots\right\} $,
$X_{1},X_{2},\ldots,X_{n}\stackrel{\mathrm{iid}}{\sim}\mathrm{Pareto}\left(\nu,\beta\right)$,
and empty bin count $E_{n}\coloneqq M_{n}-K_{n}$, we note that:
\begin{enumerate}
\item When $n\rightarrow\infty$:
\[
\mathbb{E}K_{n}\sim1+\frac{\log_{\rho}\left(\left(1-\nicefrac{1}{\rho^{\beta}}\right)n\right)}{\beta},\quad\mathbb{E}E_{n}\sim\frac{1}{\beta}\left[\log_{\rho}\left(\frac{1}{1-\nicefrac{1}{\rho^{\beta}}}\right)+\frac{\gamma}{\log\rho}\right].
\]
\item Letting $\xi_{\beta,\rho}\coloneqq\frac{1}{\beta}\log_{\rho}2$, we
have:
\begin{align*}
\xi_{\beta,\rho}\in\mathbb{Z}_{+} & \implies\lim_{n\rightarrow\infty}\mathrm{Var}\left(K_{n}\right)=\xi_{\beta,\rho}\\
\xi_{\beta,\rho}\in\mathbb{R}\left/\mathbb{Z}_{+}\right. & \implies\limsup_{n\rightarrow\infty}\mathrm{Var}\left(K_{n}\right)\leq\left\lceil \xi_{\beta,\rho}\right\rceil .
\end{align*}
\end{enumerate}
\end{restatable}
\begin{proof}
The proof, which uses Theorems \ref{thm:Karlin} and \ref{thm:Bogachev-etal},
appears in Appendix \ref{sec:Occupancy-Proofs}.
\end{proof}
The results given above mirror those for $\mathrm{Exp}\left(\lambda\right)$.
1.\ The number of empty bins is bounded as $n\rightarrow\infty$
(\emph{cf.}\ Theorem \ref{thm:bndExpMissing}). We expect 225 empty
bins when $\beta=1$ and $\rho=\frac{1.01}{0.99}$. 2.\ The expected
number of empty bins is not bounded in $\beta$ or $\rho$. If either
or both of these become smaller ($i.e.$, if we increase outlier magnitude
or histogram accuracy), we increase the expected number of occupied
\emph{and} empty bins. 3.\ Finally, $\mathrm{Var}_{\mathrm{Pareto}\left(\nu,\beta\right)}\left(K_{n}\right)$,
$\mathrm{Var}_{\mathrm{Pareto}\left(\nu,\beta\right)}\left(M_{n}\right)$
(Proposition \ref{prop:ParetoMGF-M}), and it appears $\mathrm{Var}_{\mathrm{Exp}\left(\lambda\right)}\left(K_{n}\right)$
(Figure \ref{sub-fig:expVarKn}) plateau as $n\rightarrow\infty$.

\section{Longest Gap \label{sec:Longest-Gap}}

While occupancy grows apace with size under $\mathrm{Exp}\left(\lambda\right)$
and $\mathrm{Pareto}\left(\nu,\beta\right)$ sampling (§\ref{sec:Occupancy}),
this section looks more closely at missingness. In particular we consider
\[
L_{n}\coloneqq\textrm{length of the longest gap in }\left\{ Y_{\left(1\right)},Y_{\left(1\right)}+1,\ldots,Y_{\left(n\right)}\right\} ,
\]
where a \emph{gap} is a contiguous sequence of empty bins. The largest
gap eats up only a fraction of the empty bins, \emph{e.g.}, $5\mathbb{E}L_{n}\approx\mathbb{E}E_{n}$
in the $\mathrm{Exp}\left(\lambda\right)$ and $\mathrm{Pareto}\left(\nu,1\right)$
settings when $\rho=\nicefrac{1.01}{0.99}$ (§\ref{subsec:Exp-length}--§\ref{subsec:Par-length}).

The excellent paper by \citet{GH09} shows that the asymptotic fate
of $L_{n}$ depends on and implies properties for $\left\{ \nicefrac{\Pr\left(Y_{1}\geq k+1\right)}{\Pr\left(Y_{1}\geq k\right)}\right\} _{k\geq0}$.

\begin{restatable}{theorem}{grubelhitczenko}

\label{thm:Grubel-Hitczenko}For $k\geq0$ and $Y_{1},Y_{2},\ldots$
i.i.d.\ with support $\mathbb{N}\coloneqq\left\{ 0,1,\ldots\right\} $,
let $p_{k}\coloneqq\Pr\left(Y_{1}=k\right)$ and $q_{k}\coloneqq\sum_{j=k}^{\infty}p_{j}=\Pr\left(Y_{1}\geq k\right)$.
Then, 
\begin{enumerate}
\item \label{enu:Grubel-Hitczenko-1}$\sum_{k=0}^{\infty}\frac{q_{k+1}}{q_{k}}<\infty\iff\Pr\left(\lim_{n\rightarrow\infty}L_{n}=0\right)=1$;
\item \label{enu:Grubel-Hitczenko-2}$\lim_{k\rightarrow\infty}\frac{q_{k+1}}{q_{k}}=0\iff\lim_{n\rightarrow\infty}\Pr\left(L_{n}=0\right)=1$; 
\item \label{enu:Grubel-Hitczenko-3}$\lim_{k\rightarrow\infty}\frac{q_{k+1}}{q_{k}}=1\implies\Pr\left(\lim_{n\rightarrow\infty}L_{n}=\infty\right)=1$.
\end{enumerate}
\end{restatable}
\begin{proof}
Please refer to Theorems 1, 2, and 3 of \citet{GH09}.
\end{proof}
While Theorem \ref{thm:Grubel-Hitczenko} does not consider the setting
$\lim_{k\rightarrow\infty}\nicefrac{q_{k+1}}{q_{k}}=r\in\left(0,1\right)$,
Theorem \ref{thm:Grubel-Hitczenko-Geom} (also from \citet{GH09})
considers this setting's exemplar, the geometric distribution.

\begin{restatable}{theorem}{grubelhitczenkogeom}

\label{thm:Grubel-Hitczenko-Geom}If $Y_{1},Y_{2},\ldots$ are i.i.d.\ with
$\Pr\left(Y_{i}=k\right)=\left(1-p\right)^{k}p$, for $0<p<1$ and
$k\geq0$, we have 
\begin{equation}
\Pr\left(\overline{V}\leq\left(l-1\right)\tau_{p}\right)\leq\Pr\left(L_{n}\leq l\right)\leq\Pr\left(\overline{V}_{n}\leq\left(l+1\right)\tau_{p}\right),\label{eq:bndPLn}
\end{equation}
for $n,l\geq1$, where $\tau_{p}\coloneqq-\log\left(1-p\right)$ and
\[
\overline{V}_{n}\coloneqq\max\left\{ V_{1},V_{2},\ldots,V_{n-1}\right\} \uparrow\max_{1\leq i<\infty}V_{i}\eqqcolon\overline{V},
\]
for independent $V_{i}\sim\mathrm{Exp}\left(i\right)$, so that, for
$v>0$, $\Pr\left(\overline{V}_{n}\leq v\right)=\prod_{i=1}^{n-1}\left(1-e^{-iv}\right)$
and $\Pr\left(\overline{V}\leq v\right)=\prod_{i=1}^{\infty}\left(1-e^{-iv}\right)$.
Further, there is a subsequence $\left\{ n_{m}\right\} _{m=1}^{\infty}$
such that $\mathscr{L}\left(L_{n_{m}}\right)$ approaches a non-degenerate
distribution as $m\rightarrow\infty$.

\end{restatable}
\begin{proof}
Please refer to Theorem 5 of \citet{GH09}, which specifies a family
of $\left(\textrm{subsequence},\textrm{limiting distribution}\right)$
pairs indexed by $\eta\in\left[0,1\right]$.
\end{proof}
The upper and lower bounds for $\left\{ \Pr\left(L_{n}\leq l\right)\right\} _{n,l\geq1}$
in (\ref{eq:bndPLn}) provide upper and lower bounds for $\mathbb{E}L_{n}$
and $\mathrm{Var}\left(L_{n}\right)$, which we summarize in the following
lemma. Figure \ref{fig:LnGeomVsP} compares the derived bounds with
simulated data.

\begin{restatable}{lemma}{ghbndmom}

\label{lem:GHBndMom}For $\left\{ Y_{i}\right\} _{i=1}^{\infty}$
i.i.d.\ with $\Pr\left(Y_{i}=k\right)=\left(1-p\right)^{k}p$, for
$0<p<1$, we note that
\begin{align*}
\mathbb{E}L_{n} & \in\left[\frac{\left(1-p\right)^{2}}{p}-\frac{\left(1-p\right)^{2n}}{1-\left(1-p\right)^{n}},\:\frac{1.5}{p}+1\right]\\
\mathbb{E}L_{n}^{2} & \in\left[\frac{2\left(1-p\right)^{2}}{p^{2}}-\frac{2\left(1-p\right)^{2n}}{\left(1-\left(1-p\right)^{n}\right)^{2}},\:\frac{3.3}{p^{2}}+\frac{3}{2p}+2\right]\\
\mathrm{Var}\left(L_{n}\right) & \in\left[0,\:\frac{3.3}{p^{2}}+\frac{3}{2p}+2-\left(\frac{\left(1-p\right)^{2}}{p}-\frac{\left(1-p\right)^{2n}}{1-\left(1-p\right)^{n}}\right)^{2}\right],
\end{align*}
so that
\[
\begin{array}{rccc}
 & \textrm{as }n\rightarrow\infty & \textrm{as }p\rightarrow0^{+} & \textrm{as }p\rightarrow1^{-}\\
\mathbb{E}L_{n}\,\tilde{\in} & \left[\frac{\left(1-p\right)^{2}}{p},\:\frac{1.5}{p}+1\right], & \left[\frac{n-1}{np},\:\frac{1.5}{p}\right], & \left[0,\:2.5\right],\\
\mathbb{E}L_{n}^{2}\,\tilde{\in} & \left[\frac{2\left(1-p\right)^{2}}{p^{2}},\:\frac{3.3}{p^{2}}+\frac{3}{2p}+2\right], & \left[\frac{2n^{2}\left(1-p\right)^{2}-2}{n^{2}p^{2}},\:\frac{3.3}{p^{2}}\right], & \left[0,\:6.8\right],\\
\mathrm{Var}\left(L_{n}\right)\,\tilde{\in} & \left[0,\:\frac{2.3}{p^{2}}+\frac{5.5}{p}\right], & \left[0,\:\frac{2.3n^{2}+2n-1}{n^{2}p^{2}}\right], & \left[0,\:6.8\right],
\end{array}
\]
where $\tilde{\in}$ indicates asymptotic upper and lower bounds.

\end{restatable}
\begin{proof}
The proof, which uses Theorem \ref{thm:Grubel-Hitczenko-Geom}, appears
in Appendix \ref{sec:Longest-Gap-Proofs}.
\end{proof}
\begin{figure}[!t]
\hfill{}\includegraphics[scale=0.45]{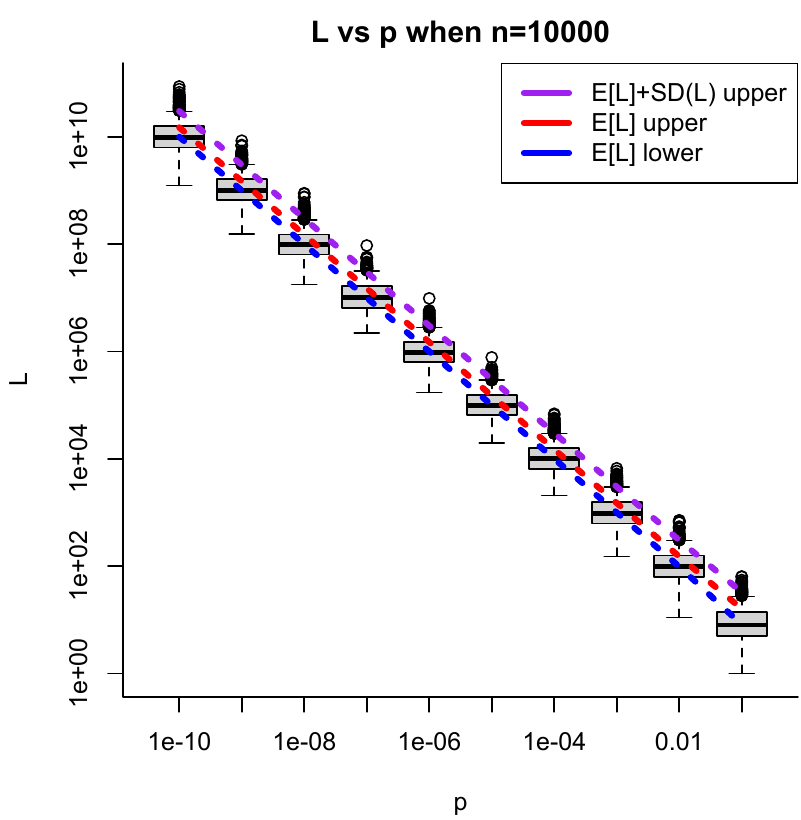}\hfill{}

\caption{The largest gaps in $\mathrm{Geometric}\left(p\right)$ samples. Each
box plot summarizes $10^{3}$ data sets of size $n=10^{4}$. Lemma
\ref{lem:GHBndMom} bounds $\mathbb{E}L_{n}$ and $\mathrm{Var}\left(L_{n}\right)$.
Red and blue curves give upper and lower bounds for $\mathbb{E}L_{n}$.
The purple curve sums the upper bounds for $\mathbb{E}L_{n}$ and
$\sqrt{\mathrm{Var}\left(L_{n}\right)}$. Axes use the log scale.
Skewness at times places the median (bold horizontal line) below the
lower bounds for $\mathbb{E}L_{n}$.}

\label{fig:LnGeomVsP}
\end{figure}

\subsection{The Exponential Setting \label{subsec:Exp-length}}

Our analysis here (as in §\ref{subsec:Occupancy-Exp}) hinges on splitting
the $\mathrm{Exp}\left(\lambda\right)$ CDF into two parts: $F_{0}$
and $F_{1}$ ((\ref{eq:kappa})--(\ref{eq:F1})). Note that, under
$F_{0}$, bin masses are nearly geometric.

\begin{restatable}{conjecture}{ghexpfzero}

\label{conj:GHExpF0}If $X_{1},X_{2},\ldots,X_{n}\stackrel{\mathrm{iid}}{\sim}F_{0}$
in (\ref{eq:F0}), then, as $n\rightarrow\infty$,
\[
\mathbb{E}L_{n}\,\tilde{\in}\,\left[\frac{1}{\rho\left(\rho-1\right)},\:\frac{2.5\rho-1}{\rho-1}\right]\textrm{ and }\mathrm{Var}\left(L_{n}\right)\,\tilde{\in}\,\left[0,\:\frac{\rho\left(7.8\rho-5.5\right)}{\left(\rho-1\right)^{2}}\right],
\]
where $\tilde{\in}$ indicates asymptotic upper and lower bounds.

\end{restatable}
\begin{proof}
The basic idea is that the independent $Y_{i}$ are nearly $\mathrm{Geometric}\left(1-\nicefrac{1}{\rho}\right)$.
If they \emph{were} $\mathrm{Geometric}\left(1-\nicefrac{1}{\rho}\right)$,
the result would follow (Lemma \ref{lem:GHBndMom}). To see that the
$Y_{i}$ are nearly $\mathrm{Geometric}\left(1-\nicefrac{1}{\rho}\right)$,
we recall the proof of Proposition \ref{prop:BogachevExp} part \ref{enu:Var0Kn},
which gives, for $k\leq0$,
\begin{align*}
p_{k} & \coloneqq\Pr\left(\rho^{\kappa_{\lambda,\rho}+k-1}<X_{1}\leq\rho^{\kappa_{\lambda,\rho}+k}\right)\\
 & =\frac{\rho^{-\frac{\rho^{k}}{\rho-1}}-\rho^{-\frac{\rho^{k+1}}{\rho-1}}}{1-\rho^{-\frac{\rho}{\rho-1}}}\sim\frac{\rho^{-\frac{\rho^{k+\nicefrac{1}{2}}}{\rho-1}}\rho^{k}\log\rho}{1-\rho^{-\frac{\rho}{\rho-1}}},
\end{align*}
as $k\rightarrow-\infty$. This then implies that $\nicefrac{p_{k-1}}{p_{k}}\sim\rho^{-1+\rho^{k-\nicefrac{1}{2}}}$,
which approaches $\nicefrac{1}{\rho}$ (very quickly) as $k\rightarrow-\infty$.
\end{proof}
\begin{figure}[!t]
\hfill{}\includegraphics[scale=0.5]{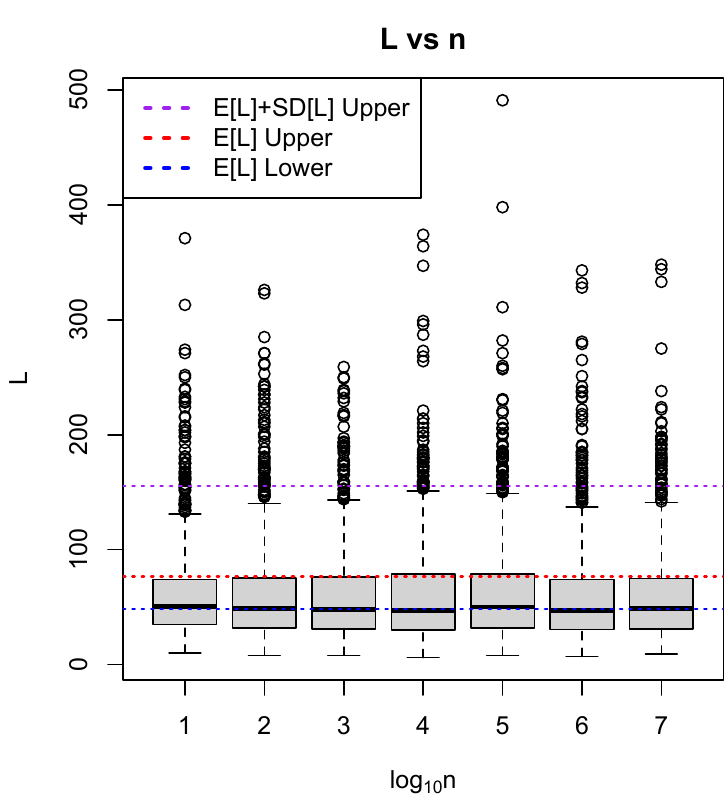}\hfill{}

\caption{Largest gap size when $X_{i}\stackrel{\mathrm{iid}}{\sim}F_{0}$.
Box plots show $10^{3}$ data sets of size $1\protect\leq\log_{10}n\protect\leq7$.
Blue and red lines give the bounds for $\mathbb{E}L_{n}$ in Conjecture
\ref{conj:GHExpF0}. The purple line sums the upper bounds for $\mathbb{E}L_{n}$
and $\sqrt{\mathrm{Var}\left(L_{n}\right)}$. We use $\rho=\frac{1.01}{0.99}\approx1.02$.
Simulations bolster the veracity of our claim in Conjecture \ref{conj:GHExpF0}.}

\label{fig:LnUnderF0}
\end{figure}

Do Theorem \ref{thm:Grubel-Hitczenko-Geom}'s conclusions apply more
generally to $q_{k}$ for which $\nicefrac{q_{k+1}}{q_{k}}\rightarrow r\in\left(0,1\right)$?
Perhaps (\citet{GH09}). Simulations support Conjecture \ref{conj:GHExpF0}
and the weak dependence of Lemma \ref{lem:GHBndMom}'s intervals on
$n$ (Figure \ref{fig:LnUnderF0}). Turning to $F_{1}$ we have:

\begin{restatable}{proposition}{ghlimlfone}

\label{prop:GHLimLF1}If $X_{1},\ldots,X_{n}\stackrel{\mathrm{iid}}{\sim}F_{1}$
in (\ref{eq:F1}), then $\Pr\left(\lim_{n\rightarrow\infty}L_{n}=0\right)=1$.

\end{restatable}
\begin{proof}
The proof, which uses Theorem \ref{thm:Grubel-Hitczenko}, appears
in Appendix \ref{sec:Longest-Gap-Proofs}.
\end{proof}
In the $\mathrm{Exp}\left(\lambda\right)$ setting with $\rho=\frac{1.01}{0.99}$
we expect a longest gap of length $\left[48,77\right]$ bins to the
left of $\nicefrac{1.01}{\lambda}$. The longest gap subsumes perhaps
20\% of the empty bins (see §\ref{subsec:Occupancy-Exp}), and its
length has standard deviation $\leq79$ bins (Figure \ref{fig:LnUnderF0}).
We expect the longest gap to become longer and its length more variable
as $\rho$ (and $\epsilon$) becomes smaller. On the other hand, the
largest gap above $\nicefrac{1.01}{\lambda}$ has length zero and
$\mathbb{E}X_{\left(n\right)}\sim\frac{\gamma+\log n}{\lambda}$,
as $n\rightarrow\infty$ (see page 83 of \citet{D05}).

\subsection{The Pareto Setting \label{subsec:Par-length}}

Recall that, in the $\mathrm{Pareto}\left(\nu,\beta\right)$ setting,
the bin masses are geometric, yielding:

\begin{restatable}{proposition}{ghlimlpar}

\label{prop:GHLimLPar}If $X_{1},X_{2},\ldots,X_{n}\stackrel{\mathrm{iid}}{\sim}\mathrm{Pareto}\left(\nu,\beta\right)$,
then
\[
\begin{array}{rccc}
 & \textrm{as }n\rightarrow\infty & \textrm{as }\beta\rightarrow0^{+} & \textrm{as }\beta\rightarrow\infty\\
\mathbb{E}L_{n}\,\tilde{\in} & \left[\frac{1}{\rho^{\beta}\left(\rho^{\beta}-1\right)},\:\frac{2.5\rho^{\beta}-1}{\rho^{\beta}-1}\right], & \left[\frac{n-1}{n\left(\rho^{\beta}-1\right)},\:\frac{1.5}{\rho^{\beta}-1}\right], & \left[0,\:2.5\right],\\
\mathrm{Var}\left(L_{n}\right)\,\tilde{\in} & \left[0,\:\frac{\rho^{\beta}\left(7.8\rho^{\beta}-5.5\right)}{\left(\rho^{\beta}-1\right)^{2}}\right], & \left[0,\:\frac{2.3n^{2}+2n-1}{n^{2}\left(\rho^{\beta}-1\right)^{2}}\right], & \left[0,\:6.8\right],
\end{array}
\]
where $\tilde{\in}$ indicates asymptotic upper and lower bounds.

\end{restatable}
\begin{proof}
Apply Lemma \ref{lem:GHBndMom} noting that $Y_{i}\stackrel{\mathrm{iid}}{\sim}\mathrm{Geometric}\left(1-\nicefrac{1}{\rho^{\beta}}\right)$
(\emph{cf.}\ (\ref{eq:ParGeom})).
\end{proof}
In the $\mathrm{Pareto}\left(\nu,1\right)$ setting with $\rho=\frac{1.01}{0.99}$
and $n\rightarrow\infty$, we expect a longest gap of length $\left[48,77\right]$
bins, subsuming perhaps 20\% of the empty bins (see §\ref{subsec:Occupancy-Par}).
The length of the longest gap has standard deviation $\leq79$ bins.
The longest gap becomes longer and its length more variable as $\rho$
or $\beta$ become smaller, \emph{i.e.}, as we increase histogram
precision or outlier size.

\section{Conclusions \label{sec:Conclusions}}

\begin{table}
\makebox[\textwidth][c]{

\begin{tabular}{r|c|c}
 & $\mathrm{Exp}\left(\lambda\right)$ & $\mathrm{Pareto}\left(\nu,\beta\right)$\tabularnewline
\hline 
$F_{\mathrm{Size}}^{-1}\left(q\right)$ & $\log_{\rho}n-\log_{\rho}\left(-\log q\right)$ & $\beta^{-1}\left(\log_{\rho}n-\log_{\rho}\left(-\log q\right)\right)$\tabularnewline
$\mathscr{L}\left(\mathrm{Size}\right)$ & $\mathrm{Gumbel}\left(1+\frac{\log n+\log\log n}{\log\rho},\frac{1}{\log\rho}\right)$ & $\mathrm{Gumbel}\left(1+\frac{\log n}{\beta\log\rho},\frac{1}{\beta\log\rho}\right)$\tabularnewline
$\Pr\left(\textrm{Largest Bin}\right)$ & $e^{\nicefrac{1}{\rho}}\Gamma\left(1+\nicefrac{1}{\rho}\right)n^{\nicefrac{-1}{\rho}}$ & $\left(\rho^{\beta}-\nicefrac{1}{\rho^{\beta}}\right)n^{-1}$\tabularnewline
$\Pr\left(\textrm{Right Tial}\right)$ & $n^{-1}$ & $n^{-1}$\tabularnewline
$\mathbb{E}\left[\mathrm{Occupancy}\right]$ & $\log_{\rho}n$ & $\beta^{-1}\log_{\rho}n$\tabularnewline
$\mathbb{E}\left[\mathrm{\#Empty}\right]$ & $\log_{\rho}\left(\frac{3\rho}{2\left(\rho-1\right)}\right)+\frac{\gamma}{\log\rho}+1$ & $\frac{1}{\beta}\left[\log_{\rho}\left(\frac{1}{1-\nicefrac{1}{\rho^{\beta}}}\right)+\frac{\gamma}{\log\rho}\right]$\tabularnewline
$\mathrm{Var}\left(\mathrm{Occupancy}\right)$ & ``$\left\lceil \log_{\rho}2\right\rceil +1$'' & $\left\lceil \beta^{-1}\log_{\rho}2\right\rceil $\tabularnewline
$\mathbb{E}\left[\textrm{Largest Gap}\right]$ & ``$\left.\left(2.5\rho-1\right)\right/\left(\rho-1\right)$'' & $\left.\left(2.5\rho^{\beta}-1\right)\right/\left(\rho^{\beta}-1\right)$\tabularnewline
$\mathrm{Var}\left(\textrm{Largest Gap}\right)$ & ``$\rho\left(7.8\rho-5.5\right)\left/\left(\rho-1\right)^{2}\right.$'' & $\rho^{\beta}\left(7.8\rho^{\beta}-5.5\right)\left/\left(\rho^{\beta}-1\right)^{2}\right.$\tabularnewline
\end{tabular}

}

\caption{Key results and simulation-supported conjectures, which use quotes.
$F_{\mathrm{Size}}^{-1}\left(q\right)$ and $\mathscr{L}\left(\mathrm{Size}\right)$
give large-$n$ approximations; the remaining rows give large-$n$
upper bounds. Except for $\mathscr{L}\left(\mathrm{Size}\right)$,
$\Pr\left(\textrm{Largest Bin}\right)$, $\mathbb{E}\left[\mathrm{\#Empty}\right]$,
and $\mathrm{Var}\left(\mathrm{Occupancy}\right)$, the $\mathrm{Exp}\left(\lambda\right)$
results equal $\mathrm{Pareto}\left(\nu,\beta\right)$ results with
$\beta=1$.}

\label{tab:Summary}
\end{table}

Table \ref{tab:Summary} compares results for $\mathrm{Exp}\left(\lambda\right)$
and $\mathrm{Pareto}\left(\nu,\beta\right)$. Having observed $n$
data points, we expect we will be able to approximate something at
least as large as the $\left(1-\nicefrac{1}{n}\right)$th quantile
of $\mathrm{Exp}\left(\lambda\right)$ or $\mathrm{Pareto}\left(\nu,\beta\right)$.
Note further that
\[
\left\{ \begin{array}{rl}
F_{\mathrm{Pareto}\left(\nu,\beta\right),\mathrm{Size}}^{-1}\left(q\right) & \lesseqqgtr F_{\mathrm{Exp}\left(\lambda\right),\mathrm{Size}}^{-1}\left(q\right)\\
\mathbb{E}_{\mathrm{Pareto}\left(\nu,\beta\right)}\left[\mathrm{Occupancy}\right] & \lesseqqgtr\mathbb{E}_{\mathrm{Exp}\left(\lambda\right)}\left[\mathrm{Occupancy}\right]\\
\mathbb{E}_{\mathrm{Pareto}\left(\nu,\beta\right)}\left[\textrm{Largest Gap}\right] & \lesseqqgtr\mathbb{E}_{\mathrm{Exp}\left(\lambda\right)}\left[\textrm{Largest Gap}\right]\\
\mathrm{Var}_{\mathrm{Pareto}\left(\nu,\beta\right)}\left(\textrm{Largest Gap}\right) & \lesseqqgtr\mathrm{Var}_{\mathrm{Exp}\left(\lambda\right)}\left(\textrm{Largest Gap}\right)
\end{array}\right\} \iff\beta\gtreqqless1
\]
The above pairs of values differ only with respect to the Pareto $\beta$
parameter. For results that differ beyond setting $\beta=1$:
\begin{itemize}
\item $\mathscr{L}_{\mathrm{Exp}\left(\lambda\right)}\left(\mathrm{Size}\right)$
has a $\log\log n$ term $\mathscr{L}_{\mathrm{Pareto}\left(\nu,1\right)}\left(\mathrm{Size}\right)$
lacks (Remark \ref{rem:altAsymp});
\item $\mathrm{Exp}\left(\lambda\right)$ (of course) puts more mass on
its final bin than $\mathrm{Pareto}\left(\nu,1\right)$;
\item $\mathbb{E}_{\mathrm{Exp}\left(\lambda\right)}\left[\mathrm{\#Empty}\right]-\log_{\rho}\left(\nicefrac{3}{2}\right)-1=\mathbb{E}_{\mathrm{Pareto}\left(\nu,1\right)}\left[\mathrm{\#Empty}\right]$;
and
\item $\mathrm{Var}_{\mathrm{Exp}\left(\lambda\right)}\left(\mathrm{Occupancy}\right)-1=\mathrm{Var}_{\mathrm{Pareto}\left(\nu,1\right)}\left(\mathrm{Occupancy}\right)$.
\end{itemize}
Of the eight Pareto values that depend on $\beta$, only $\Pr\left(\textrm{Largest Bin}\right)$
grows larger as $\beta$ grows larger. This makes sense: larger $\beta$
means fewer Pareto outliers.

In the end, we see that, if we set $\beta=1$ and ignore $\Pr\left(\textrm{Largest Bin}\right)$,
exponential histograms holding $\mathrm{Exp}\left(\lambda\right)$
and $\mathrm{Pareto}\left(\nu,\beta\right)$ data have remarkably
similar properties. While $\mathrm{Pareto}\left(\nu,\beta\right)$
is heavy-tailed to the right, relative to the histogram $\mathrm{Exp}\left(\lambda\right)$
is (in some sense) heavy-tailed to the left. While $\mathrm{Pareto}\left(\nu,\beta\right)$
avoids overly small values, $\mathrm{Exp}\left(\lambda\right)$---like
$\mathrm{Uniform}\left(0,1\right)$ (see §\ref{subsec:Heavy-and-Light-Tailed})---does
not.

Relative-error-guaranteed exponential histograms work well on $\mathrm{Exp}\left(\lambda\right)$
and $\mathrm{Pareto}\left(\nu,\beta\right)$ data. 1.\ Sizes grow
like $\log_{\rho}n$. 2.\ Missing masses to the right of the largest
bin decrease like $\nicefrac{1}{n}$. 3.\ Numbers of unused bins
plateau with $n$. 4.\ Longest gaps occupy a small fraction of the
total number of empty bins. The \citet{CKLTV21} approach to relative-error-guaranteed
quantile estimation uses a much more involved algorithm---one that
requires more space (see (\ref{eq:ReqSketchSpace})).

Exponential histograms probably efficiently store many types of data.
Take, \emph{e.g.}, Cauchy data. While this setting calls for two histograms:
$\mathcal{H}_{\left(-\infty,0\right)}$ and $\mathcal{H}_{\left(0,\infty\right)}$,
we imagine four: $\mathcal{H}_{\left(-\infty,-1\right)}$, $\mathcal{H}_{\left[-1,0\right)}$,
$\mathcal{H}_{\left(0,1\right]}$, and $\mathcal{H}_{\left(1,\infty\right)}$.
We expect $\mathcal{H}_{\left[-1,0\right)}$, $\mathcal{H}_{\left(0,1\right]}$,
and a histogram holding $\mathrm{Uniform}\left(0,1\right)$ values
to have similar size, accuracy, occupancy, and gap sizes. We expect
$\mathcal{H}_{\left(-\infty,-1\right)}$, $\mathcal{H}_{\left(1,\infty\right)}$,
and a histogram holding $\mathrm{Pareto}\left(1,1\right)$ values
to have similar size, accuracy, occupancy, and gap sizes. In particular,
we expect total size to be $\mathcal{O}\left(\log n\right)$, as $n\rightarrow\infty$.

\section*{Acknowledgements}

I thank J.F.\ Huard for sharing his enthusiasm for computation on
data streams.

\appendix

\section{Proofs for Section \ref{sec:Size} \label{sec:Size-Proofs}}

\distexpm*
\begin{proof}
Letting $W_{n}\coloneqq\log X_{\left(n\right)}-\log X_{\left(1\right)}$,
we start by showing that
\begin{align}
F_{W_{n}}\left(w\right) & =\left(n-1\right)B\left(1+\frac{n}{e^{w}-1},n-1\right)\label{eq:ExpFW}\\
f_{W_{n}}\left(w\right) & =\frac{ne^{w}F_{W_{n}}\left(w\right)}{\left(e^{w}-1\right)^{2}}\left\{ \psi\left(\frac{ne^{w}}{e^{w}-1}\right)-\psi\left(1+\frac{n}{e^{w}-1}\right)\right\} \label{eq:ExpfW}\\
F_{W_{n}}^{-1}\left(q\right) & \sim\log\left(\frac{n\log\left(n-1\right)}{\log\nicefrac{1}{q}}+1\right)\sim\log n-\log\left(-\log q\right),\label{eq:ExpFWinv}
\end{align}
for $w>0$ and $0<q<1$, where the asymptotic results hold in the
settings described in the theorem statement. Letting $Y_{i}\coloneqq\log X_{i}$,
for $1\leq i\leq n$, and fixing $y\in\mathbb{R}$, we note that 
\[
F_{Y_{i}}\left(y\right)=\Pr\left(Y_{i}\leq y\right)=\Pr\left(X_{i}\leq e^{y}\right)=1-\exp\left(-\lambda e^{y}\right),
\]
which gives $f_{Y_{i}}\left(y\right)=\lambda\exp\left(-\lambda e^{y}+y\right)$.
Equation (2.3.3) of \citet{DN03} then gives
\begin{align}
F_{W_{n}}\left(w\right) & =n\lambda\int_{-\infty}^{\infty}e^{-\lambda e^{y}+y}\left[e^{-\lambda e^{y}}-e^{-\lambda e^{y}e^{w}}\right]^{n-1}dy\\
 & =n\int_{0}^{\infty}e^{-z}\left[e^{-z}-e^{-ze^{w}}\right]^{n-1}dz\\
 & =n\sum_{j=0}^{n-1}\binom{n-1}{j}\left(-1\right)^{n-1-j}\int_{0}^{\infty}e^{-z\left[\left(n-1-j\right)e^{w}+j+1\right]}dz\label{eq:useBinThm}\\
 & =n\sum_{j=0}^{n-1}\frac{\binom{n-1}{j}\left(-1\right)^{n-1-j}}{\left(n-1-j\right)e^{w}+j+1}=n\sum_{i=0}^{n-1}\frac{\binom{n-1}{i}\left(-1\right)^{i}}{ie^{w}+n-i}\label{eq:j-to-i}\\
 & =\,_{2}F_{1}\left(-\left(n-1\right),\frac{n}{e^{w}-1};1+\frac{n}{e^{w}-1};1\right)\\
 & =\left(n-1\right)B\left(1+\frac{n}{e^{w}-1},n-1\right),\label{eq:useHypergeo}
\end{align}
where (\ref{eq:useBinThm}) uses the binomial theorem, (\ref{eq:j-to-i})
uses $i=n-1-j$, and (\ref{eq:useHypergeo}) uses an identity for
the hypergeometric function $_{2}F_{1}$.\footnote{See \href{https://functions.wolfram.com/07.23.03.0002.01}{https://functions.wolfram.com/07.23.03.0002.01}.}
This gives (\ref{eq:ExpFW}), and so also (\ref{eq:ExpfW}). To see
(\ref{eq:ExpFWinv}) note that 
\begin{align}
\lim_{y\rightarrow\infty}yB\left(1+\frac{\log\left(\nicefrac{1}{q}\right)}{\log y},y\right) & =\lim_{y\rightarrow\infty}\frac{y\Gamma\left(1+\frac{\log\left(\nicefrac{1}{q}\right)}{\log y}\right)}{y^{1+\frac{\log\left(\nicefrac{1}{q}\right)}{\log y}}}=q\label{eq:lim-y}\\
\lim_{q\rightarrow0^{+}}yB\left(1+\frac{\log\left(\nicefrac{1}{q}\right)}{\log y},y\right) & =\lim_{q\rightarrow0^{+}}\frac{y\Gamma\left(y\right)}{\left(1+\frac{\log\left(\nicefrac{1}{q}\right)}{\log y}\right)^{y}}=0\label{eq:lim-q0}\\
\lim_{q\rightarrow1^{-}}yB\left(1+\frac{\log\left(\nicefrac{1}{q}\right)}{\log y},y\right) & =yB\left(1,y\right)=1,\label{eq:lim-q1}
\end{align}
for $y>0$, where (\ref{eq:lim-y})--(\ref{eq:lim-q0}) use a well-known
approximation for $\lim_{y\rightarrow\infty}B\left(c,y\right)$ and
(\ref{eq:lim-y}) and (\ref{eq:lim-q1}) use the continuity of $\Gamma$
and $\log$ and $B$ at 1. Setting
\begin{equation}
1+\frac{n}{e^{w}-1}=1+\frac{\log\left(\nicefrac{1}{q}\right)}{\log\left(n-1\right)}\label{eq:solveThis}
\end{equation}
and solving for $w$ yields (\ref{eq:ExpFWinv}); \emph{i.e.}, we
have (\ref{eq:ExpFW})--(\ref{eq:ExpFWinv}). Results (\ref{eq:ExpFM})--(\ref{eq:ExpFMinv})
follow by noting that $M_{n}=\nicefrac{W_{n}}{\log\rho}+1$, so that
$F_{M_{n}}\left(\mu\right)=F_{W_{n}}\left(\left(\mu-1\right)\log\rho\right)$,
which completes the proof.
\end{proof}
Lemma \ref{lem:Cerone} helps us prove Lemma \ref{lem:bndExpMomM}.
\begin{lem}
\label{lem:Cerone}For $x,y>1$ and $A\left(x\right)\coloneqq\frac{x-1}{x\sqrt{2x-1}}$,
we have $0\leq\nicefrac{1}{xy}-B\left(x,y\right)\leq A\left(x\right)A\left(y\right)\leq A\left(\nicefrac{\left(3+\sqrt{5}\right)}{2}\right)^{2}\approx0.09017.$
\end{lem}

\begin{proof}
This is Theorem 7 on page 80 of \citet{C07}.
\end{proof}
\bndexpmomm*
\begin{proof}
Letting $W_{n}\coloneqq\log X_{\left(n\right)}-\log X_{\left(1\right)}$,
we first prove a related result for $W_{n}$. Using (\ref{eq:ExpFW})
and Lemma \ref{lem:Cerone} and starting with $\mathbb{E}W_{n}$ we
have 
\begin{align}
\mathbb{E}W_{n} & =\int_{0}^{\infty}\left\{ 1-F_{W_{n}}\left(w\right)\right\} dw\label{eq:bndEW}\\
 & =\int_{0}^{\infty}\left\{ 1-\left(n-1\right)B\left(1+\frac{n}{e^{w}-1},n-1\right)\right\} dw\\
 & \geq n\int_{0}^{\infty}\frac{dw}{e^{w}+n-1}=\lambda_{1}
\end{align}
and
\begin{equation}
\mathbb{E}W_{n}-\lambda_{1}\leq\frac{n\left(n-2\right)}{\sqrt{2n-3}}\int_{0}^{\infty}\frac{1}{e^{w}+n-1}\sqrt{\frac{e^{w}-1}{e^{w}+2n-1}}dw=\delta_{1}.
\end{equation}
Now turning to $\mathbb{E}W_{n}^{2}$ and again using (\ref{eq:ExpFW})
and Lemma \ref{lem:Cerone} we have
\begin{align}
\mathbb{E}W_{n}^{2} & =2\int_{0}^{\infty}w\left\{ 1-F_{W_{n}}\left(w\right)\right\} dw\label{eq:startBndEW2}\\
 & \geq2n\int_{0}^{\infty}\frac{w}{e^{w}+n-1}dw=\lambda_{2}
\end{align}
and
\begin{align}
\mathbb{E}W_{n}^{2}-\lambda_{2} & \leq\frac{2n\left(n-2\right)}{\sqrt{2n-3}}\int_{0}^{\infty}\frac{w}{e^{w}+n-1}\sqrt{\frac{e^{w}-1}{e^{w}+2n-1}}dw\\
 & =\frac{2n\left(n-2\right)}{\sqrt{2n-3}}\int_{1}^{\infty}\frac{\log z}{z\left(z+n-1\right)}\sqrt{\frac{z-1}{z+2n-1}}dz\\
 & \leq\frac{2n\left(n-2\right)}{\sqrt{2n-3}}\int_{1}^{\infty}\frac{\log z}{\sqrt{z}\left(z+n-1\right)^{\nicefrac{3}{2}}}dz=\delta_{2}.\label{eq:bndEW2}
\end{align}
Combining (\ref{eq:bndEW}) through (\ref{eq:bndEW2}) yields bounds
for $\mathrm{Var}\left(W_{n}\right)=\mathbb{E}W_{n}^{2}-\left(\mathbb{E}W_{n}\right)^{2}$,
and noting that $M_{n}=\nicefrac{W_{n}}{\log\rho}+1$ gives the desired
result.
\end{proof}
Lemma \ref{lem:ParSizeInequal} helps us prove Theorem \ref{thm:ParetoCDF}.
\begin{lem}
\label{lem:ParSizeInequal}For $e^{-m}<q<1$\emph{ }(so that $0<-\log q<m$),
we have 
\[
-\frac{\log q}{m}\leq q^{-1/m}-1\leq-\frac{\log q}{m+\log q}.
\]
\end{lem}

\begin{proof}
This follows from the inequality $x+1\leq e^{x}\leq\left(1-x\right)^{-1}$,
for $\left|x\right|<1$. If we let $e^{x}=q^{-1/m}$, \emph{i.e.},
$x=-\frac{\log q}{m}\in\left(0,1\right)$, then the result follows.
\end{proof}
\paretocdf*
\begin{proof}
Letting $W_{n}\coloneqq\log X_{\left(n\right)}-\log X_{\left(1\right)}$,
we start by showing that
\begin{align}
F_{W_{n}}\left(w\right) & =\left(1-e^{-\beta w}\right)^{n-1}\label{eq:ParetoFW}\\
f_{W_{n}}\left(w\right) & =\left(n-1\right)\beta\left(1-e^{-w\beta}\right)^{n-2}e^{-\beta w}\label{eq:ParetofW}\\
F_{W_{n}}^{-1}\left(q\right) & =-\frac{\log\left(1-q^{\nicefrac{1}{\left(n-1\right)}}\right)}{\beta}\sim\frac{\log n-\log\left(-\log q\right)}{\beta},\label{eq:ParetoFWinv}
\end{align}
for $w>0$ and $0<q<1$, where the asymptotic results hold in the
settings described in the theorem statement. Letting $Y_{i}\coloneqq\log X_{i}$,
for $1\leq i\leq n$, and fixing $y>\log\nu$, we note that 
\[
F_{Y_{i}}\left(y\right)=\Pr\left(Y_{i}\leq y\right)=\Pr\left(X_{i}\leq e^{y}\right)=1-\left(\nicefrac{\nu}{e^{y}}\right)^{\beta},
\]
which gives $f_{Y_{i}}\left(y\right)=\left(\nicefrac{\nu}{e^{y}}\right)^{\beta}\beta$.
Equation (2.3.3) of \citet{DN03} then gives
\begin{align}
F_{W_{n}}\left(w\right) & =n\beta\int_{\log\nu}^{\infty}\left(\nicefrac{\nu}{e^{y}}\right)^{\beta}\left[\left(\nicefrac{\nu}{e^{y}}\right)^{\beta}-\left(\nicefrac{\nu}{e^{y+w}}\right)^{\beta}\right]^{n-1}dy\\
 & =n\beta\left(1-e^{-\beta w}\right)^{n-1}\int_{0}^{1}x^{\beta n-1}dx=\left(1-e^{-\beta w}\right)^{n-1},\label{eq:useTransform}
\end{align}
where (\ref{eq:useTransform}) uses transformation $x=\nicefrac{\nu}{e^{y}}$.
This gives (\ref{eq:ParetoFW}), (\ref{eq:ParetofW}), and the first
part of (\ref{eq:ParetoFWinv}). The second part of (\ref{eq:ParetoFWinv})
follows from Lemma \ref{lem:ParSizeInequal} with $m=n-1$. Results
(\ref{eq:ParetoFM}) through (\ref{eq:ParetoFMinv}) then follow from
(\ref{eq:ParetoFW}) through (\ref{eq:ParetoFWinv}) because $M_{n}=\nicefrac{W_{n}}{\log\rho}+1$,
so that $F_{M_{n}}\left(\mu\right)=F_{W_{n}}\left(\left(\mu-1\right)\log\rho\right)$,
which completes the proof.
\end{proof}
\paretomgfm*
\begin{proof}
Using (\ref{eq:ParetofM}) in Theorem \ref{thm:ParetoCDF} we have
\begin{align*}
\mathbb{E}e^{tM_{n}} & =\left(n-1\right)\beta\log\rho\int_{1}^{\infty}e^{t\mu}\left(1-\rho^{-\beta\left(\mu-1\right)}\right)^{n-2}\rho^{-\beta\left(\mu-1\right)}d\mu\\
 & =\left(n-1\right)e^{t}\int_{0}^{1}x^{-t\left/\left(\beta\log\rho\right)\right.}\left(1-x\right)^{n-2}dx\\
 & =\left(n-1\right)e^{t}B\left(1-\nicefrac{t}{\beta\log\rho},n-1\right).
\end{align*}
The remaining statements follow by taking $\mathbb{E}M_{n}^{k}=\left.\frac{d^{k}}{dt^{k}}\mathbb{E}e^{tM_{n}}\right|_{t=0}$,
for $k\geq1$, and noting that $\psi\left(x\right)\sim\log x$, $\psi_{1}\left(x\right)\sim\nicefrac{1}{x}$,
and $\psi_{2}\left(x\right)\sim\nicefrac{-1}{x^{2}}$, as $x\rightarrow\infty$.
\end{proof}

\section{Proofs for Section \ref{sec:Accuracy} \label{sec:Accuracy-Proofs}}

\bndexpprob*
\begin{proof}
For $x\in\mathbb{R}$ we first note that
\begin{align}
F_{J_{n}}\left(x\right) & =\Pr\left(X\leq\rho^{x}\right)^{n}=\left(1-\exp\left(-\lambda\rho^{x}\right)\right)^{n}\label{eq:CDFJn}\\
f_{J_{n}}\left(x\right) & =n\lambda\left(1-\exp\left(-\lambda\rho^{x}\right)\right)^{n-1}\rho^{x}\exp\left(-\lambda\rho^{x}\right)\log\rho\label{eq:densityJn}
\end{align}
as in Theorem \ref{thm:Parato-A}. Using (\ref{eq:densityJn}) and
$y\in\left\{ e^{-\lambda\rho^{x-1}},e^{-\lambda\rho^{x}},e^{-\lambda\rho^{x+1}}\right\} $
we then have
\begin{align*}
\mathbb{E}e^{-\lambda\rho^{J_{n}-1}} & =n\rho\int_{0}^{1}y^{\rho}\left(1-y^{\rho}\right)^{n-1}dy=nB\left(1+\nicefrac{1}{\rho},n\right)\\
\mathbb{E}e^{-\lambda\rho^{J_{n}}} & =n\int_{0}^{1}y\left(1-y\right)^{n-1}dy=nB\left(2,n\right)=\frac{1}{n+1}\\
\mathbb{E}e^{-\lambda\rho^{J_{n}+1}} & =\frac{n}{\rho}\int_{0}^{1}y^{\nicefrac{1}{\rho}}\left(1-y^{\nicefrac{1}{\rho}}\right)^{n-1}dy=nB\left(1+\rho,n\right),
\end{align*}
which, because $\left(\rho^{\left\lceil z\right\rceil -1},\rho^{\left\lceil z\right\rceil }\right]\subset\left(\rho^{z-1},\rho^{z+1}\right]$,
for $z\in\mathbb{R}$, leads to
\begin{align}
\Pr\left(\rho^{\left\lceil J_{n}\right\rceil -1}<X\leq\rho^{\left\lceil J_{n}\right\rceil }\right) & \leq\Pr\left(\rho^{J_{n}-1}<X\leq\rho^{J_{n}+1}\right)\\
 & =\mathbb{E}\left[e^{-\lambda\rho^{J_{n}-1}}-e^{-\lambda\rho^{J_{n}+1}}\right]\\
 & =nB\left(1+\nicefrac{1}{\rho},n\right)-nB\left(1+\rho,n\right)\\
 & \sim\frac{\Gamma\left(1+\nicefrac{1}{\rho}\right)}{n^{\nicefrac{1}{\rho}}}-\frac{\Gamma\left(1+\rho\right)}{n^{\rho}},\label{eq:useBetaLim}
\end{align}
where (\ref{eq:useBetaLim}) uses $B\left(c,n\right)\sim\Gamma\left(c\right)n^{-c}$
as $n\rightarrow\infty$. This gives (\ref{eq:ExpLastProb}); (\ref{eq:ExpRest})
uses a similar argument.
\end{proof}
\bndparprob*
\begin{proof}
Theorem \ref{thm:Parato-A} gives $\mathbb{E}e^{tJ_{n}}=n\nu^{\nicefrac{t}{\log\rho}}B\left(1-\nicefrac{t}{\beta\log\rho},n\right)$,
for $t<\beta\log\rho$. Then, because $\left(\rho^{\left\lceil x\right\rceil -1},\rho^{\left\lceil x\right\rceil }\right]\subset\left(\rho^{x-1},\rho^{x+1}\right]$,
for $x\in\mathbb{R}$, we have
\begin{align}
\Pr\left(\rho^{\left\lceil J_{n}\right\rceil -1}<X\leq\rho^{\left\lceil J_{n}\right\rceil }\right) & \leq\Pr\left(\rho^{J_{n}-1}<X\leq\rho^{J_{n}+1}\right)\\
 & =\mathbb{E}\left[\left(\nicefrac{\nu}{\rho^{J_{n}-1}}\right)^{\beta}-\left(\nicefrac{\nu}{\rho^{J_{n}+1}}\right)^{\beta}\right]\\
 & =\nu^{\beta}\left(\rho^{\beta}-\nicefrac{1}{\rho^{\beta}}\right)\mathbb{E}e^{-\beta J_{n}\log\rho}\\
 & =n\left(\rho^{\beta}-\nicefrac{1}{\rho^{\beta}}\right)B\left(2,n\right)=\frac{\rho^{\beta}-\nicefrac{1}{\rho^{\beta}}}{n+1}.\label{eq:useMGF}
\end{align}
where (\ref{eq:useMGF}) uses $\mathbb{E}e^{tJ_{n}}$ at $t=-\beta\log\rho$,
bounding $p_{\mathrm{last}}$. Bounds for $p_{\mathrm{tail}}$ follow
in a similar manner.
\end{proof}

\section{Proofs for Section \ref{sec:Occupancy} \label{sec:Occupancy-Proofs}}

Lemma \ref{lem:logBinomial} helps us prove Proposition \ref{prop:KarlinExp}.
\begin{lem}
\label{lem:logBinomial}If $X\sim\mathrm{Binomial}\left(n,p\right)$,
then $\mathbb{E}\left[\log\left(X+1\right)\right]\leq\log\left(np+1\right)$
and $\mathbb{E}\left[\log\left(\log\left(X+1\right)+1\right)\right]\leq\log\left(\log\left(np+1\right)+1\right)$.
\end{lem}

\begin{proof}
We note first that both 
\begin{align*}
f_{1}\left(x\right) & \coloneqq-\log\left(x+1\right)\textrm{ and}\\
f_{2}\left(x\right) & \coloneqq-\log\left(\log\left(x+1\right)+1\right)
\end{align*}
are convex. Then, by Jensen's inequality, we have $\mathbb{E}f_{i}\left(X\right)\geq f_{i}\left(\mathbb{E}X\right)=f_{i}\left(np\right)$,
for $i=1,2$. Dividing each inequality through by $-1$ gives the
desired result.
\end{proof}
\karlinexp*
\begin{proof}
Parts \ref{enu:E0Kn} and \ref{enu:E1Kn} use Theorem \ref{thm:Karlin};
part \ref{enu:EKn} uses Lemma \ref{lem:logBinomial}. Note that:
\begin{enumerate}
\item When $X_{i}\stackrel{\mathrm{iid}}{\sim}F_{0}$ in (\ref{eq:F0}),
we have
\begin{equation}
p_{k}\coloneqq\Pr\left(\rho^{\kappa_{\lambda,\rho}+k-1}<X_{1}\leq\rho^{\kappa_{\lambda,\rho}+k}\right)=\frac{\rho^{-\frac{\rho^{k}}{\rho-1}}-\rho^{-\frac{\rho^{k+1}}{\rho-1}}}{1-\rho^{-\frac{\rho}{\rho-1}}},\label{eq:pF0}
\end{equation}
for $k\leq0$. By the mean value theorem there is a $k^{*}\in\left(k,k+1\right)$
such that
\[
\frac{\rho^{-\frac{\rho^{k}}{\rho-1}}-\rho^{-\frac{\rho^{k+1}}{\rho-1}}}{1-\rho^{-\frac{\rho}{\rho-1}}}=\frac{\rho^{-\frac{\rho^{k^{*}}}{\rho-1}}\rho^{k^{*}}\log^{2}\rho}{\left(\rho-1\right)\left(1-\rho^{-\frac{\rho}{\rho-1}}\right)}.
\]
Note in particular that, for any $\rho>1$,
\begin{equation}
\lim_{k\rightarrow-\infty}\frac{\rho^{-\frac{\rho^{k}}{\rho-1}}-\rho^{-\frac{\rho^{k+1}}{\rho-1}}}{\rho^{-\frac{\rho^{k+\nicefrac{1}{2}}}{\rho-1}}\rho^{k}\log\rho}=1,\label{eq:expLim1}
\end{equation}
which implies that
\begin{equation}
\frac{\rho^{-\frac{\rho^{k}}{\rho-1}}-\rho^{-\frac{\rho^{k+1}}{\rho-1}}}{1-\rho^{-\frac{\rho}{\rho-1}}}\sim\frac{\rho^{-\frac{\rho^{k+\nicefrac{1}{2}}}{\rho-1}}\rho^{k}\log\rho}{1-\rho^{-\frac{\rho}{\rho-1}}},\label{eq:pF0Sim}
\end{equation}
as $k\rightarrow-\infty$. We therefore have $\alpha\left(x\right)\coloneqq-\min\left\{ k\leq0:p_{k}\geq\nicefrac{1}{x}\right\} $
\begin{align*}
 & \sim-\min\left\{ k\leq0:\frac{\rho^{-\frac{\rho^{k+\nicefrac{1}{2}}}{\rho-1}}\rho^{k}\log\rho}{1-\rho^{-\frac{\rho}{\rho-1}}}\geq\nicefrac{1}{x}\right\} \\
 & \sim\log_{\rho}\left(\frac{x\log\rho}{1-\rho^{-\frac{\rho}{\rho-1}}}\right)+\frac{1}{\log\rho}W_{0}\left(-\frac{\left(1-\rho^{-\frac{\rho}{\rho-1}}\right)\sqrt{\rho}}{x\left(\rho-1\right)}\right),
\end{align*}
which, citing Theorem \ref{thm:Karlin} and noting that $\lim_{x\rightarrow\infty}W_{0}\left(\nicefrac{-c}{x}\right)=0$,
gives the desired result.
\item When $X_{i}\stackrel{\mathrm{iid}}{\sim}F_{1}$ in (\ref{eq:F1}),
we have
\begin{equation}
p_{k}\coloneqq\Pr\left(\rho^{\kappa_{\lambda,\rho}+k-1}<X_{1}\leq\rho^{\kappa_{\lambda,\rho}+k}\right)=\rho^{-\frac{\rho^{k}-\rho}{\rho-1}}-\rho^{-\frac{\rho^{k+1}-\rho}{\rho-1}},\label{eq:pF1}
\end{equation}
for $k\geq1$. By the mean value theorem there is a $k^{*}\in\left(k,k+1\right)$
such that
\begin{equation}
\rho^{-\frac{\rho^{k}-\rho}{\rho-1}}-\rho^{-\frac{\rho^{k+1}-\rho}{\rho-1}}=\frac{\rho^{-\frac{\rho^{k^{*}}-\rho}{\rho-1}}\rho^{k^{*}}\log^{2}\rho}{\rho-1}.\label{eq:kstar}
\end{equation}
Letting $k^{*}\coloneqq k+\epsilon\eqqcolon k+1-\delta$, for $0<\epsilon,\delta<1$,
and taking the ratio of the two sides of (\ref{eq:kstar}), we note
that
\[
\frac{\rho-1}{\log^{2}\rho}\left\{ \frac{\rho^{\frac{\rho^{k+\epsilon}-\rho^{k}}{\rho-1}}-\rho^{\frac{\rho^{k+1-\delta}-\rho^{k+1}}{\rho-1}}}{\rho^{k+\epsilon}}\right\} \approx\frac{\rho-1}{\log^{2}\rho}\left\{ \frac{\rho^{\frac{\epsilon\rho^{k}\log\rho}{\rho-1}}-\rho^{-\frac{\delta\rho^{k+1}\log\rho}{\rho-1}}}{\rho^{k+\epsilon}}\right\} ,
\]
by Taylor series expansion, which implies that the expression above
approaches zero as $k$ grows if $\epsilon=\mathcal{O}\left(\nicefrac{1}{\rho^{k}}\right)$
and infinity otherwise, giving
\[
\rho^{-\frac{\rho^{k}-\rho}{\rho-1}}-\rho^{-\frac{\rho^{k+1}-\rho}{\rho-1}}\approx\frac{\rho^{-\frac{\rho^{k}-\rho}{\rho-1}}\rho^{k}\log^{2}\rho}{\rho-1}.
\]
In this setting we therefore have $\alpha\left(x\right)\coloneqq\max\left\{ k\geq1:p_{k}\geq\nicefrac{1}{x}\right\} $
\begin{align*}
 & \approx\max\left\{ k\geq1:\frac{\rho^{-\frac{\rho^{k}-\rho}{\rho-1}}\rho^{k}\log^{2}\rho}{\rho-1}\geq\nicefrac{1}{x}\right\} \\
 & \sim\log_{\rho}\left(-\frac{\rho-1}{\log\rho}W_{-1}\left(-\frac{\rho^{-\frac{\rho}{\rho-1}}}{x\log\rho}\right)\right)\\
 & \leq\log_{\rho}\left(\frac{\rho-1}{\log\rho}\left(\eta_{\rho,x}+\sqrt{2\left(\eta_{\rho,x}-1\right)}\right)\right),
\end{align*}
where the last step uses Theorem 1 of \citet{C13} and $\eta_{\rho,x}\coloneqq\log\left(x\log\rho\right)+\frac{\rho\log\rho}{\rho-1}$.
Theorem \ref{thm:Karlin} then gives the result.
\item When $X_{i}\stackrel{\mathrm{iid}}{\sim}\mathrm{Exp}\left(\lambda\right)$,
we let
\begin{equation}
N_{n}\coloneqq\sum_{j=1}^{n}\mathbf{1}_{\left\{ X_{j}\leq\frac{\rho\log\rho}{\lambda\left(\rho-1\right)}\right\} }\sim\mathrm{Binomial}\left(n,1-\rho^{-\frac{\rho}{\rho-1}}\right)\label{eq:Nn}
\end{equation}
be the number of $X_{1},X_{2},\ldots,X_{n}$ that fall below $\frac{\rho\log\rho}{\lambda\left(\rho-1\right)}$
and note that
\begin{align}
\mathbb{E}K_{n}=\mathbb{E}\left[\mathbb{E}\left[\left.K_{n}\right|N_{n}\right]\right] & =\mathbb{E}_{0}K_{N_{n}}+\mathbb{E}_{1}K_{n-N_{n}}\label{eq:subHere}\\
 & \approx\mathbb{E}_{0}K_{n_{0}}+\mathbb{E}_{1}K_{n_{1}},\label{eq:subMean}
\end{align}
where ``$\mathbb{E}_{b}$'' gives expectation with respect to $F_{b}$
and (\ref{eq:subMean}) substitutes the means of $N_{n}$ and $n-N_{n}$
in for their values in (\ref{eq:subHere}). Continuing we have
\begin{align}
\mathbb{E}K_{n} & =\sum_{j=0}^{n}\binom{n}{j}\left(1-\rho^{-\frac{\rho}{\rho-1}}\right)^{j}\rho^{-\frac{\rho\left(n-j\right)}{\rho-1}}\left(\mathbb{E}_{0}K_{j}+\mathbb{E}_{1}K_{n-j}\right)\label{eq:E01}\\
 & \lesssim\mathbb{E}\left[\log_{\rho}\left(N_{n}+1\right)\right]+\mathbb{E}\left[\log_{\rho}\left(\log_{\rho}\left(n-N_{n}+1\right)+1\right)\right]\label{eq:use12}\\
 & \leq\log_{\rho}\left(n\left(1-\rho^{-\frac{\rho}{\rho-1}}\right)+1\right)+\log_{\rho}\left(\log_{\rho}\left(n\rho^{-\frac{\rho}{\rho-1}}+1\right)+1\right)\label{eq:useBinLem}\\
 & \sim\log_{\rho}\left(n\left(1-\rho^{-\frac{\rho}{\rho-1}}\right)\right)+\log_{\rho}\log_{\rho}\left(n\rho^{-\frac{\rho}{\rho-1}}\right)\sim\log_{\rho}n,
\end{align}
where (\ref{eq:use12}) uses parts \ref{enu:E0Kn} and \ref{enu:E1Kn}
and (\ref{eq:useBinLem}) uses Lemma \ref{lem:logBinomial}.
\end{enumerate}
\end{proof}
Lemma \ref{lem:log2Binomial} helps us prove Proposition \ref{prop:BogachevExp}.
\begin{lem}
\label{lem:log2Binomial}If $X\sim\mathrm{Binomial}\left(n,p\right)$,
then $\mathbb{E}\left[\log^{2}\left(X+1\right)\right]\lesssim\log^{2}\left(np+1\right)$
and $\mathbb{E}\left[\log^{2}\left(\log\left(X+1\right)+1\right)\right]\lesssim\log^{2}\left(\log\left(np+1\right)+1\right)$,
for $0<p<1$ and $n\rightarrow\infty$.
\end{lem}

\begin{proof}
For $x\geq0$ and 
\begin{align*}
f_{1}\left(x\right) & \coloneqq\log^{2}\left(x+1\right)\\
f_{2}\left(x\right) & \coloneqq\log^{2}\left(\log\left(x+1\right)+1\right),
\end{align*}
we note that 
\begin{align*}
f''_{1}\left(x\right) & \gtreqqless0\iff x\lesseqqgtr e-1\approx1.71828\\
f''_{2}\left(x\right) & \gtreqqless0\iff x\lesseqqgtr x_{0}\approx0.63788,
\end{align*}
which implies that $f_{1}$ and $f_{2}$ are concave for $x\geq2$
and $x\geq1$. Assuming that $n\geq3$ we therefore have
\begin{align}
\mathbb{E}f_{1}\left(X\right) & =np\left(1-p\right)^{n-1}\log^{2}2+\sum_{j=2}^{n}\binom{n}{j}p^{j}\left(1-p\right)^{n-j}f_{1}\left(j\right)\\
 & \leq np\left(1-p\right)^{n-1}\log^{2}2+\theta_{n,p}f_{1}\left(\frac{np}{\theta_{n,p}}\left(1-\left(1-p\right)^{n-1}\right)\right),\label{eq:useJen}
\end{align}
where (\ref{eq:useJen}) uses Jensen's inequality, $\theta_{n,p}\coloneqq1-\left(1-p\right)^{n}-np\left(1-p\right)^{n-1}$,
and 
\[
\mathbb{E}\left[\left.X\right|X\geq2\right]=\frac{np}{\theta_{n,p}}\left(1-\left(1-p\right)^{n-1}\right).
\]
Sending $n\rightarrow\infty$ gives the result. A similar proof gives
the second statement.
\end{proof}
\bogachevexp*
\begin{proof}
Parts \ref{enu:Var0Kn} and \ref{enu:Var1Kn} use Theorem \ref{thm:Bogachev-etal};
part \ref{enu:VarKn} uses Lemma \ref{lem:log2Binomial}. Note that:
\begin{enumerate}
\item As in (\ref{eq:pF0}) and (\ref{eq:pF0Sim}) we note that 
\[
p_{j}=\frac{\rho^{-\frac{\rho^{j}}{\rho-1}}-\rho^{-\frac{\rho^{j+1}}{\rho-1}}}{1-\rho^{-\frac{\rho}{\rho-1}}}\sim\frac{\rho^{-\frac{\rho^{j+\nicefrac{1}{2}}}{\rho-1}}\rho^{j}\log\rho}{1-\rho^{-\frac{\rho}{\rho-1}}}
\]
for $j\rightarrow-\infty$, which gives $\lim_{j\rightarrow-\infty}\nicefrac{p_{j-1}}{p_{j}}=\lim_{j\rightarrow-\infty}\rho^{-1+\rho^{j-\nicefrac{1}{2}}}=\nicefrac{1}{\rho}$,
so that, for $k\geq1$, 
\[
\lim_{j\rightarrow-\infty}\frac{p_{j-k}}{p_{j}}=\lim_{j\rightarrow-\infty}\prod_{i=1}^{k}\frac{p_{j-k+i-1}}{p_{j-k+i}}=\prod_{i=1}^{k}\lim_{j\rightarrow-\infty}\frac{p_{j-k+i-1}}{p_{j-k+i}}=\frac{1}{\rho^{k}}.
\]
The stated results then follow from parts \ref{enu:Bogachev-etal-1.1}
and \ref{enu:Bogachev-etal-1.2} of Theorem \ref{thm:Bogachev-etal}.
\item As in (\ref{eq:pF1}) we note that $p_{j}=\rho^{-\frac{\rho^{j}-\rho}{\rho-1}}-\rho^{-\frac{\rho^{j+1}-\rho}{\rho-1}}$,
so that $\frac{p_{j+1}}{p_{j}}=\frac{1-\rho^{-\rho^{j+1}}}{\rho^{\rho^{j}}-1}\rightarrow0$
as $j\rightarrow\infty$. The stated result then follows from part
\ref{enu:Bogachev-etal-1.2} of Theorem \ref{thm:Bogachev-etal}.
\item Define $N_{n}$ as in (\ref{eq:Nn}). Assume for simplicity that $n_{0}$
and $n_{1}$ are integers. Let $\mathcal{R}_{0}\coloneqq\left(0,\frac{\rho\log\rho}{\lambda\left(\rho-1\right)}\right]$
and $\mathcal{R}_{1}\coloneqq\left(\frac{\rho\log\rho}{\lambda\left(\rho-1\right)},\infty\right)$.
We imagine two scenarios:
\begin{enumerate}
\item We have
\begin{align*}
X_{0,1},X_{0,2},\ldots,X_{0,n_{0}} & \sim F_{0}\\
X_{1,1},X_{1,2},\ldots,X_{1,n_{1}} & \sim F_{1},
\end{align*}
for $F_{b}$ in (\ref{eq:F0}) and (\ref{eq:F1}), where all $n=n_{0}+n_{1}$
random variables are independent. Letting $K_{n_{b}}^{\left(b\right)}$
be the number of occupied cells in $\mathcal{R}_{b}$, we note that
$K_{n_{0}}^{\left(0\right)}$ and $K_{n_{1}}^{\left(1\right)}$ are
independent and the number of occupied cells in $\mathcal{R}_{0}\cup\mathcal{R}_{1}$
is $K_{n}\coloneqq K_{n_{0}}^{\left(0\right)}+K_{n_{1}}^{\left(1\right)}$,
so that
\[
\mathrm{Var}\left(K_{n}\right)=\mathrm{Var}\left(K_{n_{0}}^{\left(0\right)}\right)+\mathrm{Var}\left(K_{n_{1}}^{\left(1\right)}\right)=\mathrm{Var}_{0}\left(K_{n_{0}}\right)+\mathrm{Var}_{1}\left(K_{n_{1}}\right).
\]
\item When $X_{1},X_{2},\ldots,X_{n}\stackrel{\mathrm{iid}}{\sim}\mathrm{Exp}\left(\lambda\right)$,
we have $\mathrm{Var}\left(K_{n}\right)$
\begin{align}
 & =\mathrm{Var}\left(\mathbb{E}\left[\left.K_{n}\right|N_{n}\right]\right)+\mathbb{E}\left[\mathrm{Var}\left(\left.K_{n}\right|N_{n}\right)\right]\\
 & \leq\mathrm{Var}\left(\mathbb{E}_{0}\left[\left.K_{N_{n}}\right|N_{n}\right]+\mathbb{E}_{1}\left[\left.K_{n-N_{n}}\right|N_{n}\right]\right)\label{eq:varKn-2}\\
 & \quad+\mathbb{E}\left[\mathrm{Var}_{0}\left(\left.K_{N_{n}}\right|N_{n}\right)+\mathrm{Var}_{1}\left(\left.K_{n-N_{n}}\right|N_{n}\right)\right]\label{eq:varKn-3}\\
 & \lesssim\mathrm{Var}\left\{ \log_{\rho}\left(N_{n}+1\right)+\log_{\rho}\left(\log_{\rho}\left(n-N_{n}+1\right)+1\right)\right\} \label{eq:varKn-4}\\
 & \lesssim\log_{\rho}^{2}\left(\left(1-\rho^{-\frac{\rho}{\rho-1}}\right)n+1\right)\sim\log_{\rho}^{2}\left(n\right)\label{eq:varKn-5}
\end{align}
where (\ref{eq:varKn-2}) uses $K_{n}=K_{N_{n}}^{\left(0\right)}+K_{n-N_{n}}^{\left(1\right)}$
and $\mathrm{Cov}\left(\left.K_{N_{n}}^{\left(0\right)},K_{n-N_{n}}^{\left(1\right)}\right|N_{n}\right)\leq0$
almost surely; (\ref{eq:varKn-4}) uses Proposition \ref{prop:KarlinExp}
and parts \ref{enu:Var0Kn} and \ref{enu:Var1Kn} above; and (\ref{eq:varKn-5})
uses $\mathrm{Cov}\left(\log_{\rho}\left(N_{n}+1\right),\log_{\rho}\left(\log_{\rho}\left(n-N_{n}+1\right)+1\right)\right)\leq0$,
$\mathrm{Var}\left(X\right)\leq\mathbb{E}\left[X^{2}\right]$, and
Lemma \ref{lem:log2Binomial}.
\end{enumerate}
\end{enumerate}
\end{proof}
Lemmas \ref{lem:transF0min} and \ref{lem:transF1max} help us prove
Theorem \ref{thm:bndExpMissing}.
\begin{lem}
\label{lem:transF0min}For $X_{1},X_{2},\ldots,X_{n}\stackrel{\mathrm{iid}}{\sim}F_{0}$
in (\ref{eq:F0}), $p_{0}\coloneqq1-\rho^{-\frac{\rho}{\rho-1}}$,
and $n\rightarrow\infty$, 
\[
-\log\left(\frac{n\lambda X_{\left(1\right)}}{p_{0}}\right)\implies\mathrm{Gumbel}\left(0,1\right).
\]
\end{lem}

\begin{proof}
For $x\in\mathbb{R}$ and $n$ large enough we have 
\begin{align*}
\Pr\left(-\log\left(\frac{n\lambda X_{\left(1\right)}}{p_{0}}\right)\leq x\right) & =\Pr\left(X_{1}\geq\frac{p_{0}e^{-x}}{n\lambda}\right)^{n}\\
 & =\left(\frac{\exp\left(-\frac{p_{0}e^{-x}}{n}\right)-\left(1-p_{0}\right)}{p_{0}}\right)^{n}\\
 & =\left(1-\frac{e^{-x}}{n}+\mathcal{O}\left(\frac{1}{n^{2}}\right)\right)^{n}\stackrel[\infty]{n}{\longrightarrow}e^{-e^{-x}}.
\end{align*}
\end{proof}
\begin{lem}
\label{lem:transF1max}For $X_{1},X_{2},\ldots,X_{n}\stackrel{\mathrm{iid}}{\sim}F_{1}$
in (\ref{eq:F1}) and $n\rightarrow\infty$, we have 
\[
\lambda X_{\left(n\right)}-\frac{\rho\log\rho}{\rho-1}-\log n\implies\mathrm{Gumbel}\left(0,1\right).
\]
\end{lem}

\begin{proof}
For $x\in\mathbb{R}$ and $n$ large enough we have
\begin{align*}
\Pr\left(\lambda X_{\left(n\right)}-\frac{\rho\log\rho}{\rho-1}-\log n\leq x\right) & =\Pr\left(X_{1}\leq\frac{x+\log n+\frac{\rho\log\rho}{\rho-1}}{\lambda}\right)^{n}\\
 & =\left(1-\frac{e^{-x}}{n}\right)^{n}\stackrel[\infty]{n}{\longrightarrow}e^{-e^{-x}}.
\end{align*}
\end{proof}
\mnminuskn*
\begin{proof}
We start with 
\begin{align}
\mathbb{E}_{0}M_{n} & \leq\log_{\rho}\left(\frac{\rho\log\rho}{\lambda\left(\rho-1\right)}\right)-\mathbb{E}_{0}\left[\log_{\rho}X_{\left(1\right)}\right]+1\label{eq:useDefs-1}\\
 & \sim\log_{\rho}\left(\frac{\rho\log\rho}{\lambda\left(\rho-1\right)}\right)+\log_{\rho}\left(\frac{n\lambda}{p_{0}}\right)+\frac{\gamma}{\log\rho}+1\label{eq:useLem-1}\\
 & =\log_{\rho}\left(\frac{n\rho\log\rho}{\left(\rho-1\right)p_{0}}\right)+\frac{\gamma}{\log\rho}+1,
\end{align}
where (\ref{eq:useDefs-1}) uses (\ref{eq:MnAn}) and (\ref{eq:F0})
and (\ref{eq:useLem-1}) uses Lemma \ref{lem:transF0min}. This, along
with part \ref{enu:E0Kn} of Proposition \ref{prop:KarlinExp}, gives
the bound for $\mathbb{E}_{0}E_{n}$. Turning to $F_{1}$ we have
\begin{align}
\mathbb{E}_{1}M_{n} & \leq\mathbb{E}_{1}\left[\log_{\rho}X_{\left(n\right)}\right]-\log_{\rho}\left(\frac{\rho\log\rho}{\lambda\left(\rho-1\right)}\right)+1\label{eq:useDefs-2}\\
 & \sim\log_{\rho}\left(\frac{\rho\log\rho}{\lambda\left(\rho-1\right)}+\frac{\gamma+\log n}{\lambda}\right)-\log_{\rho}\left(\frac{\rho\log\rho}{\lambda\left(\rho-1\right)}\right)+1\label{eq:useLem-2}\\
 & =\log_{\rho}\left(\frac{\left(\rho-1\right)\left(\gamma+\log n\right)}{\rho\log\rho}+1\right)+1,
\end{align}
where (\ref{eq:useDefs-2}) uses (\ref{eq:MnAn}) and (\ref{eq:F1})
and (\ref{eq:useLem-2}) uses Lemma \ref{lem:transF1max}. Using \citet{C13}'s
upper bound for $W_{-1}$, we then have
\[
W_{-1}\left(-\frac{\rho^{-\frac{\rho}{\rho-1}}}{n\log\rho}\right)\leq-1-\sqrt{2\zeta_{\rho,n}}-\frac{2\zeta_{\rho,n}}{3},
\]
where $\zeta_{\rho,n}\coloneqq\log\left(n\log\rho\right)+\frac{\rho\log\rho}{\rho-1}-1$.
This, along with part \ref{enu:E1Kn} of Proposition \ref{prop:KarlinExp},
gives the bound for $\mathbb{E}_{1}E_{n}$. The final statement follows
from the first two and $\mathbb{E}E_{n}=\mathbb{E}\left[\mathbb{E}_{0}\left[\left.E_{N_{n}}\right|N_{n}\right]+\mathbb{E}_{1}\left[\left.E_{n-N_{n}}\right|N_{n}\right]\right]$,
with $N_{n}$ as in (\ref{eq:Nn}).
\end{proof}
\paretoocc*
\begin{proof}
Without loss of generality we assume bins $\left\{ \left(\nu\rho^{k-1},\nu\rho^{k}\right]\right\} _{k\geq1}$,
so that the smallest bin is not truncated. For $k\geq1$ we then have
\begin{equation}
p_{k}\coloneqq\Pr\left(\nu\rho^{k-1}<X_{1}\leq\nu\rho^{k}\right)=\rho^{-\beta\left(k-1\right)}\left(1-\rho^{-\beta}\right);\label{eq:ParGeom}
\end{equation}
\emph{i.e.}, the $Y_{i}$ are i.i.d.\ $\mathrm{Geometric}\left(1-\nicefrac{1}{\rho^{\beta}}\right)$,
with support $k\geq1$. Then,
\begin{enumerate}
\item Theorem \ref{thm:Karlin} gives the first result: Note that, for $x>0$,
\[
\alpha\left(x\right)\coloneqq\max\left\{ k\geq1:p_{k}\geq\nicefrac{1}{x}\right\} =\frac{\log_{\rho}\left(\left(1-\nicefrac{1}{\rho^{\beta}}\right)x\right)}{\beta}+1.
\]
The second result then follows from Proposition \ref{prop:ParetoMGF-M}
with $\psi\left(n\right)\sim\log n$.
\item Theorem \ref{thm:Bogachev-etal} gives the result: For $j,k\geq1$,
$\frac{p_{j+k}}{p_{j}}=\frac{1}{2}\implies k=\frac{\log_{\rho}2}{\beta}$.
\end{enumerate}
\end{proof}

\section{Proofs for Section \ref{sec:Longest-Gap} \label{sec:Longest-Gap-Proofs}}

Lemma \ref{lem:AsympRat} helps us prove Lemma \ref{lem:GHBndMom}.
\begin{lem}
\label{lem:AsympRat}For $n\geq1$ and $0<p<1$, $\frac{\left(1-p\right)^{2n}}{1-\left(1-p\right)^{n}}\sim\frac{1}{np}$
and $\frac{\left(1-p\right)^{2n}}{\left(1-\left(1-p\right)^{n}\right)^{2}}\sim\frac{1}{n^{2}p^{2}}$
as $p\rightarrow0^{+}$.
\end{lem}

\begin{proof}
Using induction on $n\geq1$, we note that $\frac{\left(1-p\right)^{2}}{1-\left(1-p\right)}=\frac{1}{p}-2+p\sim\frac{1}{p}$
as $p\rightarrow0^{+}$. Then, assuming that the result holds for
$n\geq1$, we note that 
\begin{equation}
\lim_{p\rightarrow0^{+}}\frac{1-\left(1-p\right)^{n}}{1-\left(1-p\right)^{n+1}}=\lim_{p\rightarrow0^{+}}\frac{n\left(1-p\right)^{n-1}}{\left(n+1\right)\left(1-p\right)^{n}}=\frac{n}{n+1}\label{eq:LHopitalStep}
\end{equation}
by L'Hôpital's rule, so that
\[
\frac{\left(1-p\right)^{2n+2}}{1-\left(1-p\right)^{n+1}}\sim\left(\frac{1-2p+p^{2}}{np}\right)\left(\frac{1-\left(1-p\right)^{n}}{1-\left(1-p\right)^{n+1}}\right)\sim\frac{1}{\left(n+1\right)p}
\]
as $p\rightarrow0^{+}$. The second result follows in the same way,
requiring two applications of L'Hôpital's rule at the step analogous
to (\ref{eq:LHopitalStep}).
\end{proof}
\ghbndmom*
\begin{proof}
The proof uses Theorem \ref{thm:Grubel-Hitczenko-Geom} and Lemma
\ref{lem:AsympRat}. We focus on $\mathbb{E}L_{n}$. Bounds for $\mathbb{E}L_{n}^{2}$
use the same argument and give bounds for $\mathrm{Var}\left(L_{n}\right)=\mathbb{E}L_{n}^{2}-\left(\mathbb{E}L_{n}\right)^{2}$.
Note that $\mathbb{E}L_{n}=\sum_{l=1}^{\infty}\left(1-\Pr\left(L_{n}\leq l\right)\right)$
\begin{align}
 & \leq\sum_{l=1}^{\infty}\left(1-\prod_{i=1}^{\infty}\left(1-\left(1-p\right)^{\left(l-1\right)i}\right)\right)\label{eq:useGHProp-1}\\
 & =\sum_{l=1}^{\infty}\left(1-\exp\left(\sum_{i=1}^{\infty}\log\left(1-\left(1-p\right)^{\left(l-1\right)i}\right)\right)\right)\\
 & =\sum_{l=1}^{\infty}\left(1-\exp\left(-\sum_{i=1}^{\infty}\sum_{j=1}^{\infty}\frac{\left(1-p\right)^{\left(l-1\right)ij}}{j}\right)\right)\\
 & =\sum_{l=1}^{\infty}\left(1-\exp\left(-\sum_{j=1}^{\infty}\frac{\left(1-p\right)^{\left(l-1\right)j}}{j\left(1-\left(1-p\right)^{\left(l-1\right)j}\right)}\right)\right)\label{eq:useFubini-1}\\
 & \leq\sum_{l=1}^{\infty}\left(1-\exp\left(-\sum_{j=1}^{\infty}\frac{\left(1-p\right)^{\left(l-1\right)j}}{j\left(1-\left(1-p\right)^{l-1}\right)}\right)\right)\\
 & =\sum_{l=1}^{\infty}\left(1-\left(1-\left(1-p\right)^{l-1}\right)^{\frac{1}{1-\left(1-p\right)^{l-1}}}\right)\\
 & \leq\sum_{l=1}^{\infty}\left(1.51\left(1-p\right)^{l-1}\wedge1\right)\label{eq:useMinBnd}\\
 & \leq1+\frac{1}{2p}+\frac{151}{100}\sum_{l=\left\lceil \nicefrac{1}{2p}\right\rceil +1}^{\infty}\left(1-p\right)^{l-1}\label{eq:useLogBnd}\\
 & \leq1+\frac{1}{2p}+\frac{151}{100p\sqrt{e}},\label{eq:sqrtEbnd}
\end{align}
where (\ref{eq:useGHProp-1}) uses Theorem \ref{thm:Grubel-Hitczenko-Geom};
(\ref{eq:useFubini-1}) uses Fubini's theorem; (\ref{eq:useMinBnd})
uses 
\[
1-\left(1-x\right)^{\frac{1}{1-x}}\leq\min\left(1.51x,\:1\right),\textrm{ for }x\coloneqq\left(1-p\right)^{l-1}\in\left(0,1\right);
\]
and (\ref{eq:useLogBnd}) and (\ref{eq:sqrtEbnd}) use $\frac{\log\left(\nicefrac{100}{151}\right)}{\log\left(1-p\right)}\leq\frac{1}{2p}$
and $\left(1-p\right)^{\left\lceil \nicefrac{1}{2p}\right\rceil }\leq\frac{1}{\sqrt{e}}$,
for $0<p<1$. With $\frac{151}{100\sqrt{e}}\approx0.9159$, $\mathbb{E}L_{n}\leq1+\frac{1.5}{p}$.
Similarly, $\mathbb{E}L_{n}=\sum_{l=1}^{\infty}\left(1-\Pr\left(L_{n}\leq l\right)\right)$
\begin{align}
 & \geq\sum_{l=1}^{\infty}\left(1-\prod_{i=1}^{n-1}\left(1-\left(1-p\right)^{\left(l+1\right)i}\right)\right)\label{eq:useGHProp-2}\\
 & =\sum_{l=1}^{\infty}\left(1-\exp\left(\sum_{i=1}^{n-1}\log\left(1-\left(1-p\right)^{\left(l+1\right)i}\right)\right)\right)\\
 & =\sum_{l=1}^{\infty}\left(1-\exp\left(-\sum_{i=1}^{n-1}\sum_{j=1}^{\infty}\frac{\left(1-p\right)^{\left(l+1\right)ij}}{j}\right)\right)\\
 & =\sum_{l=1}^{\infty}\left(1-\exp\left(-\sum_{j=1}^{\infty}\frac{\left(1-p\right)^{\left(l+1\right)j}-\left(1-p\right)^{\left(l+1\right)jn}}{j\left(1-\left(1-p\right)^{\left(l-1\right)j}\right)}\right)\right)\\
 & \geq\sum_{l=1}^{\infty}\left(1-\exp\left(-\sum_{j=1}^{\infty}\frac{\left(1-p\right)^{\left(l+1\right)j}}{j}+\sum_{j=1}^{\infty}\frac{\left(1-p\right)^{\left(l+1\right)jn}}{j}\right)\right)\\
 & =\sum_{l=1}^{\infty}\left(1-\frac{1-\left(1-p\right)^{l+1}}{1-\left(1-p\right)^{\left(l+1\right)n}}\right)=\sum_{l=1}^{\infty}\left(\frac{\left(1-p\right)^{l+1}-\left(1-p\right)^{\left(l+1\right)n}}{1-\left(1-p\right)^{\left(l+1\right)n}}\right)\\
 & \geq\sum_{l=1}^{\infty}\left(1-p\right)^{l+1}-\sum_{l=1}^{\infty}\left(1-p\right)^{\left(l+1\right)n}=\frac{\left(1-p\right)^{2}}{p}-\frac{\left(1-p\right)^{2n}}{1-\left(1-p\right)^{n}},
\end{align}
where (\ref{eq:useGHProp-2}) uses Theorem \ref{thm:Grubel-Hitczenko-Geom}.
We therefore have upper and lower bounds for $\mathbb{E}L_{n}$. Bounds
for $\mathbb{E}L_{n}^{2}=2\sum_{l=1}^{\infty}l\left(1-\Pr\left(L_{n}\leq l\right)\right)$
use the same argument and lead to bounds for $\mathrm{Var}\left(L_{n}\right)=\mathbb{E}L_{n}^{2}-\left(\mathbb{E}L_{n}\right)^{2}$.
Limiting results then follow, using Lemma \ref{lem:AsympRat} for
the $p\rightarrow0^{+}$ setting, which completes the proof.
\end{proof}
\ghlimlfone*
\begin{proof}
Without loss of generality (if $\kappa_{\lambda,\rho}\in\left.\mathbb{R}\right\backslash \mathbb{Z})$,
we consider an exponential histogram with bins $\left\{ \left(\rho^{\kappa_{\lambda,\rho}+k-1},\rho^{\kappa_{\lambda,\rho}+k}\right]\right\} _{k=1}^{\infty}$,
so that 
\begin{align*}
p_{k} & \coloneqq\Pr\left(Y_{1}=k\right)=\Pr\left(\rho^{\kappa_{\lambda,\rho}+k-1}<X_{1}\leq\rho^{\kappa_{\lambda,\rho}+k}\right)=\rho^{-\frac{\rho^{k}-\rho}{\rho-1}}-\rho^{-\frac{\rho^{k+1}-\rho}{\rho-1}}\\
q_{k} & \coloneqq\sum_{j=k}^{\infty}p_{j}=\Pr\left(Y_{1}\geq k\right)=\Pr\left(X_{1}\geq\rho^{\kappa_{\lambda,\rho}+k-1}\right)=\rho^{-\frac{\rho^{k}-\rho}{\rho-1}}.
\end{align*}
Now, because $\rho>1$, we see that
\[
\frac{\nicefrac{q_{k+2}}{q_{k+1}}}{\nicefrac{q_{k+1}}{q_{k}}}=\frac{q_{k}q_{k+2}}{q_{k+1}^{2}}=\frac{1}{\rho^{\left(\rho-1\right)\rho^{k}}}\longrightarrow0
\]
as $k\rightarrow\infty$, so that the ratio test and Theorem \ref{thm:Grubel-Hitczenko}
part \ref{enu:Grubel-Hitczenko-1} give the result.
\end{proof}


\begin{thebibliography}{Hartmann \& Schlossnagle(2020)}
\bibitem[Agarwal et al(2013)]{ACHPWY13}Agarwal PK, G Cormode, Z Huang,
JM Phillips, Z Wei, K Yi (2013) ``Mergeable Summaries.'' \emph{ACM
Transactions on Database Systems}, 38:4:26, pages 1--28.

\bibitem[Aldor-Noiman et al(2013)]{ABBRS13}Aldor-Noiman S, Brown
LD, Buja A, Rolke W, \& Stine RA (2013) ``The power to see: a new
graphical test of normality.'' \emph{The American Statistician},
67: 4, pages 249--260.

\bibitem[Bogachev et al(2008)]{BGY08}Bogachev LV, AV Gnedin, \& YV
Yakubovich (2008) ``On the variance of the number of occupied boxes.''
\emph{Advances in Applied Mathematics}, 40: pages 401--432.

\bibitem[Cerone(2007)]{C07}Cerone P (2007) ``Special functions:
Approximations and bounds.'' \emph{Appl Anal Discrete Math}, 1: 1,
pages 72--91.

\bibitem[Chambers et al(2006)]{CJLVW06}Chambers JM, DA James, D Lambert,
S Vander Wiel (2006) ``Monitoring networked applications with incremental
quantile estimation.'' \emph{Statistical Science}, 21(4): pages 463--475.

\bibitem[Chatzigeorgiou(2013)]{C13}Chatzigeorgiou I (2013) ``Bounds
on the lambert function and their application to the outage analysis
of user cooperation.'' IEEE Comm Letters, 17: 8, pages 1505--1508.

\bibitem[Cormen et al(2001)]{CLRS01}Cormen TH, CE Leiserson, RL Rivest,
C Stein (2001) \emph{Introduction to Algorithms, Second Edition}.
MIT Press, Cambridge, MA.

\bibitem[Cormode et al(2021)]{CKLTV21}Cormode G, Z Karnin, E Liberty,
J Thaler, \& P Veselý (2021) ``Relative error streaming quantiles.''
\emph{PODS'21: Proceedings of the 40th ACM SIGMOD-SIGACT-SIGAI Symposium
on Principles of Database Systems}, pages 96--108.

\bibitem[Cormode \& Yi(2020)]{CY20}Cormode G \& K Yi (2020) \emph{Small
Summaries for Big Data}. Cambridge UP, Cambridge, UK.

\bibitem[David \& Nagaraja(2003)]{DN03}David HA \& HN Nagaraja (2003)
Order Statistics, 3rd Edition. John Wiley \& Sons, Hoboken, NJ.

\bibitem[Decrouez et al(2018)]{DGP18}Decrouez G, M Grabchak, Q Paris
(2018) ``Finite sample properties of the mean occupancy counts and
probabilities.'' \emph{Bernoulli}, 24: 3, pages 1910--1941.

\bibitem[Dunning \& Ertl(2019)]{DE19}Dunning T \& O Ertl (2019) ``Computing
Extremely Accurate Quantiles Using t-Digests.'' See \href{https://arxiv.org/abs/1902.04023}{https://arxiv.org/abs/1902.04023}.

\bibitem[Durrett(2005)]{D05}Durrett R (2005) \emph{Probability: Theory
and Examples, 3rd edition}. Brooks/Cole---Thomas Learning, Belmont,
CA.

\bibitem[Dvoretzky et al(1956)]{DKW56}Dvoretzky A, J Kiefer, \& J
Wolfowitz (1956), ``Asymptotic minimax character of the sample distribution
function and of the classical multinomial estimator.'' Annals of
Mathematical Statistics, 27: 3, pages 642--669.

\bibitem[Foss et al(2013)]{FKZ13}Foss S, D Korshunov, \& S Zachary
(2013) \emph{An Introduction to Heavy-Tailed and Subexponential Distributions,
Second Edition}. Springer, New York, NY.

\bibitem[Gnedin et al(2007)]{GHP07}Gnedin A, B Hansen, J Pitman (2007)
``Notes on the occupancy problem with infinitely many boxes: general
asymptotics and power laws.'' \emph{Probab. Surveys}, 4, pages 146--171.

\bibitem[Greenwald \& Khanna(2001)]{GK01}Greenwald M \& S Khanna
(2001) ``Space-efficient online computation of quantile summaries.''
\emph{ACM SIGMOD Record}, 30: 2, pages 58--66.

\bibitem[Grübel \& Hitczenko(2009)]{GH09}Grübel R \& Hitczenko P
(2009) ``Gaps in discrete random samples.'' \emph{Journal of Applied
Probability}, 46: 4, pages 1038-1051.

\bibitem[Haan \& Ferreira(2006)]{HF06}de Haan L \& A Ferreira (2006)
\emph{Extreme Value Theory: An Introduction}. Springer, New York,
NY.

\bibitem[Hartmann \& Schlossnagle(2020)]{HS20}Hartmann H \& T Schlossnagle
(2020) ``Circllhist: the log-linear histograms for IT infrastructure
monitoring.'' See \href{https://arxiv.org/abs/2001.06561}{https://arxiv.org/abs/2001.06561}.

\bibitem[Johnson \& Killeen(1983)]{JK83}Johnson B McK \& T Killeen
(1983) ``An explicit formula for the C.D.F. of the L1 norm of the
Brownian bridge.'' \emph{Ann. Probab.}, 11: 3, pages 807-808.

\bibitem[Karlin(1967)]{K67}Karlin S (1967) ``Central limit theorems
for certain infinite urn schemes.'' \emph{J of Math \& Mech}, 17:
4, pages 373--401.

\bibitem[Karnin et al(2016)]{KLL16}Karnin Z, K Lang, E Liberty (2016)
``Optimal quantile approximation in streams.'' \emph{Proc IEEE FOCS
'16}, pages 71-78.

\bibitem[Massart(1990)]{M90}Massart P (1990) ``The tight constant
in the Dvoretzky--Kiefer--Wolfowitz inequality.'' \emph{Annals
of Probability}, 18: 3, pages 1269--1283.

\bibitem[Masson et al(2019)]{MRL19}Masson C, JE Rim, HK Lee (2019)
``DDSketch: a fast and fully-mergeable quantile sketch with relative-error
guarantees.'' \emph{Proc VLDB Endowment}, 12: 12.

\bibitem[Munro \& Paterson(1980)]{MP80}Munro JI \& MS Paterson (1980)
``Selecting and sorting with limited storage.'' \emph{Theoretical
Computer Science}, 12, pages 315--323.

\bibitem[Rényi(1953)]{R53}Rényi A (1953) ``On the theory of order
statistics.'' \emph{Acta Mathematica Academiae Scientiarum Hungaricae},
4, pages 191--231.

\bibitem[Resnick(1987)]{R87}Resnick SI (1987) \emph{Extreme Values,
Regular Variation, and Point Processes}. Springer, New York, NY.

\bibitem[Rosenkrantz(2000)]{R00}Rosenkrantz WA (2000) ``Confidence
bands for quantile functions: a parametric and graphic alternative
for testing goodness of fit.'' \emph{The American Statistician},
54: 3, pages 185--190.

\bibitem[Schmid \& Trede(1996)]{SM96}Schmid F \& M Trede (1996) ``An
L1-variant of the Cramer-von Mises test.'' \emph{Statistics \& Probability
Letters}, 26: 1, pages 91--96.

\bibitem[Shrivastava et al(2004)]{SBAS04}Shrivastava N, C Buragohain,
D Agrawal, S Suri (2004) ``Medians and beyond: new aggregation techniques
for sensor networks.'' \emph{SenSys '04: Proceedings of the 2nd International
Conference on Embedded Networked Sensor Systems}, pages 239--249.

\bibitem[Tibshirani(2008)]{T08}Tibshirani RJ (2008) ``Fast computation
of the median by successive binning.'' See \href{https://www.stat.cmu.edu/~ryantibs/median/}{https://www.stat.cmu.edu/$\sim$ryantibs/median/}.
\end{thebibliography}
\end{document}